\documentclass{article}
\usepackage{amsmath}
\usepackage{amsthm}

\usepackage{textcomp}
\usepackage{times}
\usepackage{xcolor}
\usepackage{enumitem}
\usepackage{caption}
\usepackage{chngcntr}
\counterwithin{figure}{section}
\counterwithin{table}{section}
\usepackage{tikz}
\usepackage{graphicx}
\usepackage[mathscr]{eucal}
\usepackage{scrextend}
\usepackage{wasysym}
\usepackage{etoolbox}
\makeatletter

\usepackage[all]{xy}

\usepackage{etoolbox}
\makeatletter
\patchcmd{\@makechapterhead}{50\p@}{\chapheadtopskip}{}{}

\patchcmd{\@makeschapterhead}{50\p@}{\chapheadtopskip}{}{}
\makeatother
\newlength{\chapheadtopskip}
\usepackage[parfill]{parskip} 
\usepackage{amssymb}
\usepackage{imakeidx}
\usepackage{enumitem}
\usepackage[mathlines]{lineno}
\DeclareMathAlphabet{\mathpzc}{OT1}{pzc}{m}{it}
\makeatletter
\newcommand{\mylabel}[2]{#2\def\@currentlabel{#2}\label{#1}}
\makeatother
\usepackage{fancyhdr}
\usepackage{enumitem}
\usepackage{xpatch}
\newtheorem{theorem}{Theorem}[section]
\newtheorem{lemma}[theorem]{Lemma}
\newtheorem{obs}[theorem]{Observation}
\newtheorem{defn}[theorem]{Definition}
\newtheorem{prop}[theorem]{Proposition}
\newtheorem{claim}[theorem]{Claim}
\newtheorem{cor}[theorem]{Corollary}

\newtheorem{subclaim}{Subclaim}[theorem]

\newenvironment{claimproof}[1]{\par\noindent\underline{Proof:}\space#1}{\hspace{1mm}$\blacksquare$}

\usepackage[left=2.54cm,top=2.54cm,right=2.54cm,bottom=2.54cm]{geometry}
\usepackage[nodisplayskipstretch]{setspace}
\newlength\FHoffset
\fancyheadoffset{\FHoffset}
\newlength\FHright
\makeatletter
\usepackage{framed}

 \newtheoremstyle{TheoremNum}
        {\topsep}{\topsep}              
        {\itshape}                      
        {}                              
        {\bfseries}                     
        {.}                             
        { }                             
        {\thmname{#1}\thmnote{ \bfseries #3}}
    \theoremstyle{TheoremNum}
    \newtheorem{thmn}{Theorem}

\newtheoremstyle{PropNum}
        {\topsep}{\topsep}              
        {\itshape}                      
        {}                              
        {\bfseries}                     
        {.}                             
        { }                             
        {\thmname{#1}\thmnote{ \bfseries #3}}
    \theoremstyle{PropNum}

\newtheoremstyle{LemmaNum}
        {\topsep}{\topsep}              
        {\itshape}                      
        {}                              
        {\bfseries}                     
        {.}                             
        { }                             
        {\thmname{#1}\thmnote{ \bfseries #3}}
    \theoremstyle{LemmaNum}
    
\makeatletter
\renewcommand\subitem{\@idxitem\nobreak\hspace*{20\p@}}
\renewcommand\subsubitem{\@idxitem\nobreak\hspace*{20\p@}}
\makeatother

\makeatother
\title{Distant Precolored Components Part III: The General Case}
\author{Joshua Nevin}
\date{}
\begin{document}
\maketitle

\begin{center}\textbf{Abstract}\end{center} This is the third in a sequence of three papers in which we prove the following generalization of Thomassen's 5-choosability theorem: Let $G$ be a finite graph embedded on a surface of genus $g$. Then $G$ can be $L$-colored, where $L$ is a list-assignment for $G$ in which every vertex has a 5-list except for a collection of pairwise far-apart components, each precolored with an ordinary 2-coloring, as long as the face-width of $G$ is at least $2^{\Omega(g)}$ and the precolored components are of distance at least $2^{\Omega(g)}$ apart. This provides an affirmative answer to a generalized version of a conjecture of Thomassen and also generalizes a result from 2017 of Dvo\v{r}\'ak, Lidick\'y, Mohar, and Postle about distant precolored vertices.  In a previous paper, we proved that the above result holds for a restricted class of embeddings which have no separating cycles of length three or four. In this paper, we use this special case to prove that the result holds in the general case. 

\section{Background and Motivation}

All graphs in this paper have a finite number of vertices. Given a graph $G$, a \emph{list-assignment} for $G$ is a family of sets $\{L(v): v\in V(G)\}$, where each $L(v)$ is a finite subset of $\mathbb{N}$. The elements of $L(v)$ are called \emph{colors}. A function $\phi:V(G)\rightarrow\bigcup_{v\in V(G)}L(v)$ is called an \emph{$L$-coloring of} $G$ if $\phi(v)\in L(v)$ for each $v\in V(G)$, and $\phi(x)\neq\phi(y)$ for any adjacent vertices $x,y$. Given an $S\subseteq V(G)$ and a function $\phi: S\rightarrow\bigcup_{v\in S}L(v)$, we call $\phi$ an \emph{ $L$-coloring of $S$} if $\phi$ is an $L$-coloring of the induced graph $G[S]$. A \emph{partial} $L$-coloring of $G$ is an $L$-coloring of an induced subgraph of $G$. Likewise, given an $S\subseteq V(G)$, a \emph{partial $L$-coloring} of $S$ is a function $\phi:S'\rightarrow\bigcup_{v\in S'}L(v)$, where $S'\subseteq S$ and $\phi$ is an $L$-coloring of $S'$. Given an integer $k\geq 1$, $G$ is called \emph{$k$-choosable} if it is $L$-colorable for every list-assignment $L$ for $G$ such that $|L(v)|\geq k$ for all $v\in V(G)$. Thomassen demonstrated in \cite{AllPlanar5ThomPap} that all planar graphs are 5-choosable. Actually, Thomassen proved something stronger. 

\begin{theorem}\label{thomassen5ChooseThm}
Let $G$ be a planar graph with facial cycle $C$. Let $xy\in E(C)$ and $L$ be a list assignment for $V(G)$ such that each vertex of $G\setminus C$ has a list of size at least five and each vertex of $C\setminus\{x,y\}$ has a list of size at least three, where $xy$ is $L$-colorable. Then $G$ is $L$-colorable.
\end{theorem}

Theorem \ref{thomassen5ChooseThm} has the following useful corollary.

\begin{cor}\label{CycleLen4CorToThom} Let $G$ be a planar graph with outer cycle $C$ and let $L$ be a list-assignment for $G$ where each vertex of $G\setminus C$ has a list of size at least five.  If $|V(C)|\leq 4$, then any $L$-coloring of $V(C)$ extends to an $L$-coloring of $G$. \end{cor}

We now recall some notions from topological graph theory. Given an embedding $G$ on surface $\Sigma$, the deletion of $G$ partitions $\Sigma$ into a collection of disjoint, open connected components called the \emph{faces} of $G$. Our main objects of study are the subgraphs of $G$ bounding the faces of $G$. Given a subgraph $H$ of $G$, we call $H$ a \emph{facial subgraph} of $G$ if there exists a connected component $U$ of $\Sigma\setminus G$ such that $H=\partial(U)$. We call $H$ is called a \emph{cyclic facial subgraph} (or, more simply, a \emph{facial cycle}) if $H$ is both a facial subgraph of $G$ and a cycle. Given  a cycle $C\subseteq G$, we say that $C$ is \emph{contractible} if it can be contracted on $\Sigma$ to a point, otherwise we say it is \emph{noncontractible}. We now introduce two standard paramaters that measure the extent to which an embedding deviates from planarity. The notion of face-width was introduced by Robertson and Seymour in their work on graph minors and has been studied extensively.

\begin{defn}\label{EWandFWDefn} \emph{Let $\Sigma$ be a surface and let $G$ be an embedding on $\Sigma$. The \emph{edge-width} of $G$, denoted by $\textnormal{ew}(G)$, is the length of the shortest noncontractible cycle in $G$. The \emph{face-width} of $G$, denoted by $\textnormal{fw}(G)$, is the smallest integer $k$ such that there exists a noncontractible closed curve of $\Sigma$ which intersects with $G$ on $k$ points. If $G$ has no noncontractible cycles, then we define $\textnormal{ew}(G)=\infty$, and if $g(\Sigma)=0$, then we define $\textnormal{fw}(G)=\infty$. The face-width of $G$ is also sometimes called the \emph{representativity} of $G$. Some authors consider the face-width to be undefined if $\Sigma=\mathbb{S}^2$, but, for our purposes, adopting the convention that $\textnormal{fw}(G)=\infty$ in this case is much more convenient. } \end{defn}

In this paper, we prove the following result: 

\begin{theorem}\label{5ListHighRepFacesFarMainRes} Let $\Sigma$ be a surface, $G$ be an embedding on $\Sigma$ of face-width at least $2^{\Omega(g(\Sigma))}$, and $F_1, \ldots, F_m$ be a collection of facial subgraphs of $G$ which are pairwise of distance at least $2^{\Omega(g(\Sigma))}$ apart. Let $x_1y_1, \ldots, x_my_m$ be a collection of edges in $G$, where $x_iy_i\in E(F_i)$ for each $i=1,\ldots, m$. Let $L$ be a list-assignment for $G$ such that 
\begin{enumerate}[label=\arabic*)]
\itemsep-0.1em
\item for each $v\in V(G)\setminus\left(\bigcup_{i=1}^mV(C_i)\right)$, $|L(v)|\geq 5$; AND
\item For each $i=1,\ldots, m$, $x_iy_i$ is $L$-colorable, and, for each $v\in V(F_i)\setminus\{x_i, y_i\}$, $|L(v)|\geq 3$.
\end{enumerate}
Then $G$ is $L$-colorable.  \end{theorem}

Theorem \ref{5ListHighRepFacesFarMainRes} is slightly stronger than imposing pairwise-far apart components with ordinary 2-colorings. That is, an immediate consequence of Theorem \ref{5ListHighRepFacesFarMainRes} is the following slightly weaker result.

\begin{theorem}\label{WVersionThmPrecCompFW}  Let $\Sigma$ be a surface, $G$ be an embedding on $\Sigma$ of face-width at least $2^{\Omega(g(\Sigma))}$, and $L$-be a list-assignment for $V(G)$ in which every vertex has a list of size at least five, except for the vertices of some connected subgraphs $K_1, \cdots, K_m$ of $G$ which are pairwise of distance at least $2^{\Omega(g(\Sigma))}$ apart, where, for each $i=1, \cdots, m$, there is an $L$-coloring of $K_i$ which is an ordinary 2-coloring. Then $G$ is $L$-colorable. 
 \end{theorem}

\section{Conventions and Structure of this Paper}

Unless otherwise specified, all graphs are regarded as embeddings on a previously specified surface, and all surface are compact, connected, and have zero boundary. If we want to talk about a graph $G$ as an abstract collection of vertices and edges, without reference to sets of points and arcs on a surface then we call $G$ an \emph{abstract graph}. 

\begin{defn}\label{ContractNatCPartDefn}\emph{Let $\Sigma$ be a surface, let $G$ be an embedding on $\Sigma$, and let $C$ be a contractible cycle in $G$. Let $U_0, U_1$ be the two open connected components of $\Sigma\setminus C$. The unique \emph{natural $C$-partition} of $G$ is the pair $\{G_0, G_1\}$ of subgraphs of $G$ where, for each $i\in\{0,1\}$, $G_i=G\cap\textnormal{Cl}(U_i)$.} \end{defn}

\begin{defn}
\emph{Given a graph $G$, a subgraph $H$ of $G$, a subgraph $P$ of $G$, and an integer $k\geq 1$, we call $P$ a \emph{$k$-chord} of $H$ if $|E(P)|=k$ and $P$ is of the following form.}
\begin{enumerate}[label=\emph{\arabic*)}]
\itemsep-0.1em
\item \emph{$P:=v_1\cdots v_kv_1$ is a cycle with $v_1\in V(H)$ and $v_2, \cdots, v_k\not\in V(H)$}; OR
\item \emph{$P:=v_1\cdots v_{k+1}$, and $P$ is a path with distinct endpoints, where $v_1, v_{k+1}\in V(H)$ and $v_2,\cdots, v_k\not\in V(H)$.}
\end{enumerate}

\end{defn}

$P$ is called \emph{proper} if it is not a cycle, i.e $P$ intersects $H$ on two distinct vertices. Otherwise it is called \emph{improper}. Note that, for $1\leq k\leq 2$, any $k$-chord of $H$ is proper, as $G$ has no loops or duplicated edges. A 1-chord of $H$ is simply referred to as a \emph{chord} of $H$. In some cases, we are interested in analyzing $k$-chords of $H$ where the precise value of $k$ is not important. We call $P$ a \emph{generalized chord} of $H$ if there exists an integer $k\geq 1$ such that $P$ is a $k$-chord of $H$. We call $P$ a \emph{proper} generalized chord of $H$ if there is an integer $k\geq 1$ such that $P$ is a proper $k$-chord of $H$. (A proper generalized chord of $H$ is also called an \emph{$H$-path}). We define \emph{improper} generalized chords of $H$ analogously. For any $A, B\subseteq V(G)$, an \emph{$(A,B)$-path} is a path $P=x_0\cdots x_k$ with $V(P)\cap A=\{x_0\}$ and $V(P)\cap B=\{x_k\}$. Given a surface $\Sigma$, an embedding $G$ on $\Sigma$, a cyclic facial subgraph $C$ of $G$, and a proper generalized chord $Q$ of $C$, there is, under certain circumstances, a natural way to partition of $G$ specified by $Q$. 

\begin{defn}\label{ContractNatCQPartChordDefn} \emph{Let $\Sigma$ be a surface, let $G$ be an embedding on $\Sigma$,  let $C$ be a cyclic facial subgraph of $G$ and let $Q$ be a generalized chord of $C$, where each cycle in $C\cup Q$ is contractible. The unique \emph{natural $(C,Q)$-partition} of $G$ is the pair $\{G_0, G_1\}$ of subgraphs of $G$ such that $G=G_0\cup G_1$ and $G_0\cap G_1=Q$, where, for each $i\in\{0,1\}$, there is a unique open connected region $U$ of $\Sigma\setminus (C\cup Q)$ such that $G_i$ consists of all the edges and vertices of $G$ in the closed region $\textnormal{Cl}(U)$.}

\emph{If the facial cycle $C$ is clear from the context then we usually just refer to $\{G_0, G_1\}$ as the \emph{natural $Q$-partition} of $G$. Note that this is consistent with Definition \ref{ContractNatCPartDefn} in the sense that, if $Q$ is not a proper generalized chord of $C$ (i.e $Q$ is a cycle) then the natural $Q$-partition of $G$ is the same as the natural $(C,Q)$-partition of $G$. If $\Sigma$ is the sphere (or plane) then the natural $(C, Q)$-partition of $G$ is always defined for any $C,Q$.}
\end{defn}

\begin{defn} \emph{Let $\Sigma$ be a surface and let $G$ be an embedding on $\Sigma$. A \emph{separating cycle} in $G$ is a contractible cycle $C$ in $G$ such that each of the two connected components of $\Sigma\setminus C$ has nonempty intersection with $V(G)$. We call $G$ \emph{short-inseparable} if $\textnormal{ew}(G)>4$ and $G$ does not contain any separating cycle of length 3 or 4.}\end{defn}

Finally, we recall the following standard notation. 

\begin{defn}

\emph{For any graph $G$, vertex set $X\subseteq V(G)$, integer $j\geq 0$, and real number $r\geq 0$, we set $D_j(X, G):=\{v\in V(G): d(v, X)=j\}$ and $B_r(X, G):=\{v\in V(G): d(v, X)\leq r\}$. Given a subgraph $H$ of $G$, we usually just write $D_j(H, G)$ to mean $D_j(V(H), G)$, and likewise, we usually write $B_r(H, G)$ to mean $B_r(V(H), G)$.}
\end{defn}

If $G$ is clear from the context, then we drop the second coordinate from the above notation to avoid clutter. We now introduce some additional notation related to list-assignments. We frequently analyze the situation where we begin with a partial $L$-coloring $\phi$ of a graph $G$, and then delete some or all of the vertices of $\textnormal{dom}(\phi)$ and remove the colors of the deleted vertices from the lists of their neighbors in $G\setminus\textnormal{dom}(\phi)$. We thus define the following.  

\begin{defn}\label{ListLSvRemove}\emph{Let $G$ be a graph with list-assignment $L$. Let $\phi$ be a partial $L$-coloring of $G$ and $S\subseteq V(G)$. We define a list-assignment $L^S_{\phi}$ for $G\setminus (\textnormal{dom}(\phi)\setminus S)$ as follows.}
$$L^S_{\phi}(v):=\begin{cases} \{\phi(v)\}\ \textnormal{if}\ v\in\textnormal{dom}(\phi)\cap S\\ L(v)\setminus\{\phi(w): w\in N(v)\cap (\textnormal{dom}(\phi)\setminus S)\}\ \textnormal{if}\ v\in V(G)\setminus \textnormal{dom}(\phi) \end{cases}$$ \end{defn}

If $S=\varnothing$, then $L^{\varnothing}_{\phi}$ is a list-assignment for $G\setminus\textnormal{dom}(\phi)$ in which the colors of the vertices in $\textnormal{dom}(\phi)$ have been deleted from the lists of their neighbors in $G\setminus\textnormal{dom}(\phi)$. The situation where $S=\varnothing$ arises so frequently that, in this case, we drop the superscript and let $L_{\phi}$ denote the list-assignment $L^{\varnothing}_{\phi}$ for $G\setminus\textnormal{dom}(\phi)$. In some cases, we specify a subgraph $H$ of $G$ rather than a vertex-set $S$. In this case, to avoid clutter, we write $L^H_{\phi}$ to mean $L^{V(H)}_{\phi}$. Finally, given two partial $L$-colorings $\phi$ and $\psi$ of $V(G)$, where $\phi(x)=\psi(x)$ for all $x\in\textnormal{dom}(\phi)\cap\textnormal{dom}(\psi)$, and $\phi(x)\neq\psi(y)$ for all edges $xy$ with $x\in\textnormal{dom}(\phi)$ and $y\in\textnormal{dom}(\psi)$, there is natural well-defined $L$-coloring $\phi\cup\psi$ of $\textnormal{dom}(\phi)\cup\textnormal{dom}(\psi)$, where
$$(\phi\cup\psi)(x)=\begin{cases}\phi(x)\ \textnormal{if}\ x\in\textnormal{dom}(\phi)\\ \psi(x)\ \textnormal{if}\ x\in\textnormal{dom}(\psi)\end{cases}$$

We frequently deal with situations where we have a set $Z$ of vertices that we want to delete, and it is desirable to color as few of them as possible in such a way that we can safely delete the remaining vertices of $Z$ without coloring them. We thus introduce the following definition.

\begin{defn}\emph{Let $G$ be a graph with a list-assignment $L$. Given a subset $Z\subseteq V(G)$ and a partial $L$-coloring $\phi$ of $V(G)$, we say that $Z$ is \emph{$(L, \phi)$-inert in $G$} if every extension of $\phi$ to an $L$-coloring of $G\setminus (Z\setminus\textnormal{dom}(\phi))$ extends to an $L$-coloring of all of $G$.}\end{defn}

The structure of the remainder of this paper is as follows. To prove Theorem \ref{5ListHighRepFacesFarMainRes}, we first need to prove an intermediate result, which is Theorem \ref{Main4CycleAnnulusThm}. We prove Theorem \ref{Main4CycleAnnulusThm} over Sections \ref{LensIntroSec}-\ref{ThmCutIntoCompWithThomFacesSec}. Then, using Theorem \ref{Main4CycleAnnulusThm} as a black box, we prove Theorem \ref{5ListHighRepFacesFarMainRes} over the course of Sections \ref{SimpleExemaxLemmaFormodblockRem}-\ref{thisisidicritCOMPcomcriTT}. 

\section{An Annulus Between Short Facial Cycles}\label{LensIntroSec}

To state Theorem \ref{Main4CycleAnnulusThm}, we first need the following definition. 

\begin{defn} \emph{Let $G$ be a graph and $L$ be a list-assignment for $G$. A facial subgraph $F$ of $G$ is said to be a \emph{Thomassen facial subgraph} of $G$ with respect to $L$ if there is a list $xy\in E(F)$ such that $xy$ is $L$-colorable and each vertex of $F\setminus\{x, y\}$ has an $L$-list of size at least three.}  \end{defn}

\begin{theorem}\label{Main4CycleAnnulusThm} Let $G$ be a short-inseparable planar graph and let $F, F'$ be two facial subgraphs of $G$, each of which is a cycle of length at most four, where $F$ is the outer face of $G$. Let $P$ be an $(F, F')$-path, and let $n:=2|E(P)|+8$ and $r:=\lceil\log_2(n)\rceil+2$. Let $L$ be a list-assignment for $V(G)$ and $\phi$ be an $L$-coloring of $V(F\cup F'\cup P)$ such that
\begin{enumerate}[label=\arabic*)]
\itemsep-0.1em
\item $\phi$ extends to $L$-color $B_{r(n-1)}(F\cup F'\cup P)$; AND
\item For each $v\in B_{rn+1}(F\cup F'\cup P)$, every facial subgraph of $G$ containing $v$, except possibly $F, F'$, is a triangle, and, if $v\not\in V(F\cup F')$, then $|L(v)|\geq 5$.
\end{enumerate}
Then there is a 2-edge-connected subgraph $K$ of $G$ with $F\cup F'\cup P\subseteq K$ and an extension of $\phi$ to a partial $L$-coloring $\psi$ of $V(K)$ such that
\begin{enumerate}[label=\arabic*)]
\itemsep-0.1em
\item $V(K)$ is $(L, \psi)$-inert in $G$ and $V(K)\subseteq B_{rn}(F\cup F'\cup P)$; AND
\item For each connected component $H$ of $G\setminus K$, the outer face of $H$ is a Thomassen facial subgraph of $H$ with respect to $L_{\phi}$.
\end{enumerate}
 \end{theorem}

We use Theorem \ref{Main4CycleAnnulusThm} in this step as a black box to reduce to a smaller counterexample to Theorem \ref{5ListHighRepFacesFarMainRes}. The proof of Theorem \ref{Main4CycleAnnulusThm} has two main ingredients: Propositions \ref{k+rboundpart} and \ref{OneStepIntermProp}. We prove Proposition \ref{k+rboundpart} in the remainder of Section \ref{LensIntroSec} and prove Proposition \ref{OneStepIntermProp} in Section \ref{NonsplitLensesSec}. Finally, in Section \ref{ThmCutIntoCompWithThomFacesSec}, we combine these two propositions to prove Theorem \ref{Main4CycleAnnulusThm}. Note that Propositions \ref{k+rboundpart} is a purely structural result, i.e. it does not deal with list-assignments. 

\begin{defn}
\emph{Let $G$ be a short-inseparable, 2-connected planar graph.  A facial subgraph $F$ of $G$ is said to be \emph{inward-facing} if we have either $G=F$ or $F$ is not the outer face of $G$. Given a cycle $C\subseteq G$, we define the following.}
\begin{enumerate}[label=\emph{\arabic*)}]
\itemsep-0.1em
\item\emph{For any integer $k\geq 0$, we say that $C$ is \emph{$k$-triangulated} if, for every $v\in B_k(C)\cap V(\textnormal{Int}(C))$, every facial subgraph of $\textnormal{Int}(C)$ containing $v$, except possibly $C$, is a triangle.}
\item \emph{For any integer $k\geq 1$, we define $\Delta^{\geq k}(C):=\{u\in V(\textnormal{Int}(C))\setminus V(C): |N(v)\cap V(C)|\geq k\}$}
\item \emph{We define $\Delta^{2p}(C)$ to be the set of $u\in\Delta^{\geq 3}(C)$ such that either $C[N(v)\cap V(C)]$ is a subpath of $C$ of length two or $|V(C)|=3$ and $V(C)\subseteq N(v)$.}
\end{enumerate} \end{defn} 

\begin{defn}\label{splitandnonsplitdefinitions}\emph{Let $G$ be a 2-connected, short-inseparable planar graph and let $C\subseteq G$ be a cycle with $|V(C)|>3$. We say that $C$ is \emph{$G$-split} if at least one of the following holds.}
\begin{enumerate}[label=\emph{\arabic*)}]
\itemsep-0.1em
\item\emph{$C$ is not 0-triangulated}; OR
\item $\Delta^{\geq 3}(C)=\varnothing$; OR
\item\emph{$C$ has a chord in $\textnormal{Int}_G(C)$}; OR
\item\emph{There is a 2-chord of $C$ in $\textnormal{Int}_G(C)$ whose endpoints are of distance at least three apart in $C$.}
\end{enumerate}
\emph{We say that $C$ is \emph{$G$-nonsplit} if it not $G$-split.}
\end{defn}

We have the following by a simple induction argument:

\begin{obs}\label{inducprecollen}
Let $G$ be a short-inseparable, 2-connected planar graph and $C\subseteq G$ be a $G$-nonsplit cycle, where $|V(C)|>3$. Then there is a sequence of cycles $(C^i: i=0,1,2\cdots)$ in $G$, and a sequence of subgraphs $(H^i: i=0,1,\cdots)$ of $G$, such that $H^0=C^0=C$, where $H^0\subseteq H^1\subseteq\cdots$, and, for each $i=0,1,\cdots$, the following hold.

\begin{enumerate}[label=\arabic*)]
\itemsep-0.1em
\item $H^i$ is 2-connected and $H^i=\textnormal{Ext}(C^i)\cap\textnormal{Int}(C)$; AND
\item $|E(C^i)|=|E(C)|$ and every facial subgraph of $H^i$, except for $C, C^i$, is a triangle; AND
\item $H^{i+1}$ is the graph obtained from $C^i$ by adding to the disc bounded by $C^i$ all the vertices of $\Delta^{2p}(C^i)\cap B_1(C)$ and all the edges between $V(C^i)$ and $\Delta^{2p}(C^i)\cap B_1(C)$; AND
\item $V(C^i\setminus C)\subseteq V(C^{i+1}\setminus C)$. 
\end{enumerate}

\end{obs}

Note that the elements of $(C^i: i=0,1,\cdots)$ are uniquely specified since $|V(C)|>3$. Since $G$ is a finite graph, there exists an index $r$ such that $\Delta^{2p}(C^r)\cap B_1(C)=\varnothing$, and, in particular, $H^r=H^j$ for all $r\geq j$, where $(H^i: i=0,1,\cdots)$ is as an Observation \ref{inducprecollen}. We let $w_G(C)$ denote this unique terminal element of $(H^i: i=0,1,\cdots)$ and we let $F_G(C)$ denote the unique terminal cycle of $(C^i: i=0, 1, \cdots)$. That is, letting $r$ be as above, we have $w_G(C):=H^r$ and $F_G(C):=C^r$. We usually just drop the subscript $G$ from this notation to avoid clutter, as the underlying $G$ is usually clear from the context. Note that. since $C$ is $G$-nonsplit, we have $\Delta^{2p}(C)\neq\varnothing$ by definition, and, in particular, $r>0$. Observation \ref{inducprecollen} describes an operation which allows us to define an ascending sequence of subgraphs of a given short-inseparable, 2-connected planar graph. 

\begin{obs}\label{trianglesplitobs} \emph{Let $G$ be a short-inseparable, 2-connected planar graph and let $C\subseteq G$ be a cycle with $|V(C)|>3$. For each $k\geq 0$, there is a subgraph $w^k(C)$ of $G$ and a facial cycle $F^k(C)$ of $w^k(C)$ specified inductively as follows. We set $w^0(C):=C$ and $F^0(C):=C$. Then, for each $k\geq 0$,}
\begin{enumerate}[label=\alph*)]
\item If $F^{k}(C)$ is $G$-split, then $w^{k+1}(C):=w^{k}(C)$ and $F^{k+1}(C):=F^k(C)$; AND 
\item Otherwise $F^k(C)$ is a cycle of length $|V(C)|$ and we have $F^{k+1}(C):=F(F^k(C))$ and $w^{k+1}(C):=\textnormal{Int}(C)\cap\textnormal{Ext}(F^{k+1}(C))=w^k(C)\cup w(F^k(C))$.
\end{enumerate}
\end{obs}

The above motivates the following.

\begin{defn}\emph{Let $G$ and let $C$ be a cycle in $G$. We define the \emph{nonsplit depth} of $C$, denoted by $\textit{NDepth}(C)$, as follows.}
\begin{enumerate}[label=\emph{\arabic*)}]
\itemsep-0.1em
\item\emph{If $|V(C)|=3$ then $\textit{NDepth}(C)=\infty$} 
\item\emph{Otherwise, $\textit{NDepth}(C)$ is the smallest integer $r$ such that $F^r(C)$ is $G$-split.}
\end{enumerate}
\end{defn}

Note that if $|V(C)|>3$, then $NDepth(C)<\infty$, since $|V(G)|$ is finite. The quantity $\textit{NDepth}(C)$ is a measure of how far we have to travel from $C$ in $\textnormal{Int}(C)$ before our graph deviates from a certain structure. 

\begin{obs} Let $G$ be a 2-connected, short-inseparable planar graph, $C\subseteq G$ be a cycle with $|V(C)|>3$, and $k\geq 0$. Then $V(w^k(C))\subseteq B_k(C)\cap\textnormal{Int}(C)$, and each facial subgraph of $w^k(C)$, except $C$ and $F^k(C)$, is a triangle. \end{obs}

\begin{defn} \emph{Let $G$ be a 2-connected, short-inseparable planar graph and $C\subseteq G$ be a cycle. Given integers $k,r\geq 0$, we say that $C$ is \emph{$(k, r)$-partitionable in $G$} if there exists a 2-connected subgraph $K$ of $\textnormal{Int}(C)$, with $C\subseteq K$ and $V(K)\subseteq B_k(C)$, such that, for every inward-facing facial subgraph $D$ of $K$, one of the following holds.
\begin{enumerate}[label=\arabic*)]
\itemsep-0.1em
\item $\Delta^{\geq 3}(D)=\varnothing$ and $D$ has no chord in $\textnormal{Int}(D)$; OR
\item $NDepth(D)\geq r$ and $|V(D)|\leq |V(C)|$.
\end{enumerate}
 The graph $K$ is called a \emph{$(k,r)$-skeleton} of $C$.} \end{defn} 

We can now state the prove the first of the two propositions we need for the proof of Theorem \ref{MainThmSSFSec}. 

\begin{prop}\label{k+rboundpart} Let $G$ be a 2-connected planar graph and $C\subseteq G$ be a cycle. Let $k, r$ be nonnegative integers with $k\geq r(|V(C)|-3)$, where $C$ is $(k+r)$-triangulated. Then $C$ is $(k,r)$-partitionable. \end{prop}

\begin{proof} Suppose not and let $G,C, k, r$ be chosen to be a counterexample to the proposition which minimizes $|V(\textnormal{Int}(C))|$. By assumption, $k\geq r(|V(C)|-3)$ and $C$ is not $(k,r)$-partitionable. 

\begin{claim}\label{rVCTrivialCasesCl} $r>0$ and $|V(C)|>3$ and furthermore, $\textit{NDepth}(C)\geq r$. \end{claim}

\begin{claimproof} If $|V(C)|=3$ then $\textit{NDepth}(C)=\infty$. In particular, if any of the three statements above do not hold, then $C$  is a $(k,r)$-skeleton of $C$, contradicting our choice of certificate. \end{claimproof}

Let $m:=\textit{NDepth}(C)$ and let $C^*:=F^m(C)$. Now, $C^*$ is a cyclic facial subgraph of $w^m(C)$ and each facial subgraph of $w^m(C)$, except for $C, C^*$, is a triangle. Furthermore, $|V(C)|=|V(C^*)|$.

\begin{claim}\label{Chordor2ChordCL} $\textnormal{Int}(C^*)$ either contains a 2-chord of $C^*$ whose endpoints are of distance at least three apart in $C^*$, or it contains a chord of $C^*$. \end{claim}

\begin{claimproof} If $\textnormal{Int}(C^*))$ contains a chord of $C^*$, then we are done, so suppose that there is no such chord. By assumption, $C^*$ is $G$-split. Since every facial subgraph of $w^m(C)$, except for $C$ and $C^*$, is a triangle, it follows that $\Delta^{\geq 3}(C^*)\neq\varnothing$, or else $w^m(C)$ is a $(k,r)$-skeleton of $C$, contradicting our choice of certificate. By assumption, $C$ is $k+r$-triangulated. Since $\Delta^{\geq 3}(C^*)\neq\varnothing$ and $V(C^*)\subseteq B_m(C)$, it follows from Definition \ref{splitandnonsplitdefinitions} that $\textnormal{Int}(C^*)$ contains a 2-chord of $F^m(C)$ whose endpoints are of distance at least three apart in $C^*$, so we are done. \end{claimproof}

It follows from Claim \ref{Chordor2ChordCL} that there is a 2-connected graph $H$ which is obtained from $C^*$ either by adding a chord of $F^m(C)$ in $\textnormal{Int}(C^*)$, or adding a 2-chord of $C^*$ in $\textnormal{Int}(C^*)$, where the endpoints of thise 2-chord have dsitance at least three apart in $C^*$. In either case, $H$ has precisely two inward-facing facial subgraphs $A_0, A_1$ (i.e the facial subgraphs of $H$ which are distinct from $C^*$). If $H$ differs from $C^*$ by a 2-chord of $C^*$, then the endpoints of this 2-chord are of distance at least three apart in $C^*$. In any case, since $|V(C)|=|V(C^*)|$, we have $\max\{|V(A_0)|, |V(A_1)|\}<|V(C)|$. Now, by Claim \ref{rVCTrivialCasesCl}, $|V(C)|>3$ and $r\geq m+1$. Since $k\geq r(|V(C)-3)|$ it follows that $k-(m+1)\geq 0$. 

\begin{claim}\label{AiSatisfiesIndHyp} For each $i=0,1$, $A_i$ is $(k+r-(m+1))$-triangulated and $k-(m+1)\geq r(|V(A_i)|-3)$. \end{claim}

\begin{claimproof}  Let $i\in\{0,1\}$. We have $V(A_i)\subseteq B_1(F^m(C))\cap V(\textnormal{Int}_G(C^*))$. Since $V(C^*)\subseteq B_m(C, G)$  and $C$ is $(k+r)$-triangulated, it follows that $A_i$ is $((k+r)-(m+1))$-triangulated. Suppose toward a contradiction that $k-(m+1)<r(|V(A_i)|-3)$. Since $|V(A_i)|<|V(C)|$, we have $k-(m+1)<r(|V(C)|-4)$, so $k+r-(m+1)<r(|V(C)|-3)$. Since $m+1\leq r$, we have $r(|V(C)|-3)>k$, contradicting our initial assumption. \end{claimproof}

 For each $i=0,1$, we have $|V(\textnormal{Int}(A_i))|<|V(\textnormal{Int}(C))|$. By our choice of counterexample certificate, it follows from Claim \ref{AiSatisfiesIndHyp} that, for each $i=0,1$, $A_i$ is $(k-(m+1), r)$-partitionable and thus admits a $(k-(m+1), r)$-skeleton $K_i\subseteq\textnormal{Int}(A_i)$. Let $K^*:=w^m(C)\cup K_0\cup K_1$. Note that $K^*\subseteq\textnormal{Int}(C)$ and $C\subseteq w^m(C)\subseteq K^*$. For each inward-facing facial subgraph $D$ of $K^*$, there is an $i\in\{0,1\}$ such that $D\subseteq\textnormal{Int}(A_i)$ and $D$ is an inward-facing facial subgraph of $K_i$. Thus, either $D$ is an induced subgraph of $\textnormal{Int}(D)$ with $\Delta^{2p}(D)=\varnothing$, or we have $\textit{NDepth}(D)\geq r$ and $|V(D)|\leq\max\{|V(A_0)|, |V(A_1)|\}$. In the latter case, we have $|V(D)|\leq |V(C)|$. Furthermore, since $V(K_i)\subseteq B_{k-(m+1)}(A_i)$ for each $i=0,1$, we have $V(K^*)\subseteq B_k(C)$. Thus, $K^*$ is a $(k,r)$-skeleton of $C$, contradicting our choice of counterexample certificate. This proves Proposition \ref{k+rboundpart}. \end{proof}

\section{Dealing with the case of Nonsplit Depth $\geq 2$}\label{NonsplitLensesSec}

To state Proposition \ref{OneStepIntermProp}, we first introduce the following definition.

\begin{defn} \emph{Let $G$ be a planar graph and $C\subseteq G$ be a cycle. Given a list-assignment $L$ for $V(G)$ and a partial $L$-coloring $\psi$ of $V(\textnormal{Ext}_C(G))$, we set $\Delta_{L}(C, \psi):=\{v\in \Delta^{2p}(C): |L_{\psi}(v)|<3\}$.}\end{defn}

\begin{prop}\label{OneStepIntermProp} Let $G$ be a 2-connected short-inseparable, planar graph with outer cycle $C$, where $|V(C)|>3$ and $\textit{NDepth}(C)\geq 2$. Let $L$ be a list-assignment for $G$ and let $\psi$ be a partial $L$-coloring of $V(C)$ such that $V(C)\setminus\textnormal{dom}(\psi)$ is $(L, \psi)$-inert in $G$ and every vertex of $B_2(C)\setminus V(C)$ has an $L$-list of size at least five. Then there exists an extension of $\psi$ to a partial $L$-coloring $\psi^*$ of $V(w^1(C))$ such that the following hold.
\begin{enumerate}[label=\emph{\arabic*)}]
\itemsep-0.1em
\item $|\Delta_L(F^1(C), \psi^*)|\leq\left\lceil\frac{|\Delta_{L}(C, \psi)|}{2}\right\rceil$; AND
\item $V(w^1(C))\setminus\textnormal{dom}(\psi^*)$ is $(L, \psi^*)$-inert in $G$
\end{enumerate}
 \end{prop}

\begin{proof} We set $C^*:=F^1(C)$, and $U^{=2}:=\{v\in V(C^*\setminus C): |N(v)\cap V(C)|=2\}$ and $U^{=1}:=\{v\in V(C^*\setminus C): |N(v)\cap V(C)|=1\}$.  For any partial $L_{\psi}$-coloring $\phi$ of $V(C^*\setminus C)$, we set $\textnormal{Tw}(\phi):=\{v\in B_2(C): |L_{\psi\cup\phi}(v)|=2\}$. To avoid clutter and unnecessary repetition of background data, we set $\Delta_L:=\Delta_L(C, \psi)$. Since $\textit{NDepth}(C)\geq 2$, each of $C$ and $C^*$ is 0-triangulated. Since $|V(C^*)|<|V(C)|$, we immediately have the following:

\begin{claim}\label{BtTwoStepClaim} 
\textcolor{white}{aaaaaaaaaaaa}
\begin{enumerate}[label=\arabic*)]
\itemsep-0.1em
\item $\Delta^{2p}(C^*)\subseteq B_2(C)$ and any two vertices of $\Delta^{2p}(C)$ are nonadajcent in $G$; AND
\item For any subpath $Q$ of $C^*\setminus C$, if $|V(Q)|\geq 2$ and both endpoints of $Q$ lie $\Delta^{2p}(C)$, then $V(\mathring{Q})\cap U^{=1}\neq\varnothing$.
\end{enumerate} \end{claim}

We now introduce the following notation. Given a subgraph $H$ of $C^*\setminus C$, we let $\textnormal{Mid}^{\downarrow}(H)$ be the set of $v\in V(H)$ such that, for some $y\in\Delta^{2p}(C^*)$, the 2-path $G[N(y)\cap V(C^*)]$ has midpoint $v$. Furthermore, we let $\textnormal{Halve}^{+}(H)$ denote the set of partial $L_{\psi}$-colorings $\phi$ of $V(H)$ such that t
\begin{enumerate}[label=\emph{\roman*)}]
\itemsep-0.1em
\item $V(H)\setminus\textnormal{dom}(\phi)$ is $(L, \psi\cup\phi)$-inert in $G$ and a subset of $\textnormal{Mid}^{\downarrow}(H)$; AND
\item $|\textnormal{Tw}(\phi)|\leq\left\lceil\frac{|\Delta_L\cap V(H)|}{2}\right\rceil$.
\end{enumerate}

The main ingredient in the proof of Proposition \ref{OneStepIntermProp} is Claim \ref{MainClaimforIntermPropSSF} below. During the proof of Claim \ref{MainClaimforIntermPropSSF}, we often need to partially color and delete the vertices of a subpath of $C^*\setminus C$ in a specified way made precise below.

\begin{defn} \emph{Given a subpath $Q$ of $C^*\setminus C$ with endpoints $p, p'$, we let $\textnormal{Link}(Q)$ denote the set of $L_{\phi}$-colorings $\sigma$ of $\{p\, p'\}\cup (V(\mathring{Q})\setminus\textnormal{Mid}^{\downarrow}(\mathring{Q})$ such that $\textnormal{Mid}^{\downarrow}(\mathring{Q})$ is $(L, \phi\cup\sigma)$-inert in $G$.}
\end{defn}

The following is straightforward to check by induction. 

\begin{claim}\label{LinkColBlackBoxCh1} Let $Q$ be an induced subpath of $C^*\setminus C$, where each internal vertex of $Q$ has an $L$-list of size at least three. Let $p, p'$ be the endpoints of $Q$. Then, for any $c\in L(p)$, there is a $\phi\in\textnormal{Link}(Q)$ with $\phi(p)=c$. Furthermore, if $|V(Q)|\geq 2$ then, for any nonempty sets $A\subseteq L(p)$ and $A'\subseteq L(p')$ with $|A|+|A'|\geq 4$, there is a $\psi\in\textnormal{Link}(Q)$ such that $\psi(p)\in A$ and $\psi(p')\in A'$. \end{claim}

We now prove Claim \ref{MainClaimforIntermPropSSF}.

\begin{claim}\label{MainClaimforIntermPropSSF} Let $Q$ be an induced subpath of $C^*\setminus C$, let $x$ be an endpoint of $Q$, and let $c\in L_{\psi}(x)$. Then there is an element of $\textnormal{Halve}^{+}(Q)$ using $c$ on $x$. \end{claim}

\begin{claimproof} Suppose not and let $Q$ be a vertex-minimal counterexample to the claim. Let $Q:=x_1\cdots x_m$ for some $m\geq 1$, where there is a $c\in L_{\psi}(x_1)$ such that no element of $\textnormal{Halve}^+(Q)$ uses $c$ on $x_1$. 

\vspace*{-8mm}
\begin{addmargin}[2em]{0em}\begin{subclaim}\label{trivcaseleq3subCL} $\Delta_L\cap V(Q-x_1)\neq\varnothing$ and $|V(Q)|>3$ \end{subclaim}

\begin{claimproof} Suppose $\Delta_L\cap V(Q-x_1)=\varnothing$. In that case, by Claim \ref{LinkColBlackBoxCh1}, there is a $\sigma\in\textnormal{Link}(Q,)$ using $c$ on $x_1$, and thus $\textnormal{Tw}(\sigma)=\varnothing$ and $\sigma\in\textnormal{Halve}^+(Q)$, contradicting our assumption on $c$. Thus, $\Delta_L\cap V(Q-x_1)\neq\varnothing$. Now suppose $|V(Q)|\leq 3$. Let $\phi$ be an $L_{\psi}$-coloring of $V(Q)$ using $c$ on $x_1$. Note that $|\textnormal{Tw}(\phi)|\leq 1$ and $\left\lceil\frac{|\Delta_L\cap V(Q)|}{2}\right\rceil\geq 1$, so $\phi\in\textnormal{Halve}^+(Q)$, contradicting our assumption on $c$. \end{claimproof}\end{addmargin}

\vspace*{-8mm}
\begin{addmargin}[2em]{0em}\begin{subclaim}\label{m-1midQSubCL1} $x_{m-1}\in\textnormal{Mid}^{\downarrow}(Q)$, and furthermore, precisely one of $x_{m-1}, x_m$ lies in $\Delta_L$.  \end{subclaim}

\begin{claimproof} Suppose toward a contradiction that $x_{m-1}\not\in\textnormal{Mid}^{\downarrow}(Q)$. Since $|V(Q)|>3$, it follows from the minimality of $Q$ that there is a $\sigma\in\textnormal{Halve}^+(Q-x_m)$ with $\sigma(x_1)=c$. Now, $\sigma$ extends to an $L_{\psi}$-coloring $\sigma^*$ of $\textnormal{dom}(\sigma)\cup\{x_m\}$. Since $x_{m-1}\not\in\textnormal{Mid}^{\downarrow}(Q)$, we have $\textnormal{Tw}(\sigma^*)=\textnormal{Tw}(\sigma)$. Since $\left\lceil\frac{|\Delta_L\cap V(Qx_{m-1})|}{2}\right\rceil\leq\left\lceil\frac{|\Delta_L\cap V(Q)|}{2}\right\rceil$, and each endpoint of $Q$ lies in $\textnormal{dom}(\sigma^*)$, we have $\sigma^*\in\textnormal{Halve}^+(Q)$, contradicting our assumption on $c$. Thus, $x_{m-1}\in\textnormal{Mid}^{\downarrow}(Q)$. Now suppose toward a contradiction that neither $x_{m-1}$ nor $x_m$ lies in $\Delta_L$. Again by the minimality of $Q$, there is a $\tau\in\textnormal{Halve}^+(Qx_{m-2})$ with $\tau(x_1)=c$. By assumption, $|L_{\psi}(x_{m-1})|\geq 3$ and $|L_{\psi}(x_m)|\geq 3$, so there is a $d\in L_{\psi}(x_m)$ such that $|L_{\psi\cup\tau}(x_{m-1})\setminus\{d\}|\geq 2$. Now, $\tau$ extends to an $L_{\psi}$-coloring $\tau^*$ of $\textnormal{dom}(\tau)\cup\{x_m\}$ with $\tau^*(x_m)=d$. Note that $\{x_{m-1}\}$ is $(L, \psi\cup\tau^*)$-inert in $G$ and $x_{m-1}\not\in\textnormal{dom}(\tau^*)$. Thus, $\textnormal{Tw}(\tau^*)=\textnormal{Tw}(\tau)$ and $\tau^*\in\textnormal{Halve}^+(Q)$, contradicting our assumption on $c$. \end{claimproof}\end{addmargin}

\vspace*{-8mm}
\begin{addmargin}[2em]{0em} \begin{subclaim}\label{Ifxm-1NotInDeltaLThenSubCL5} $x_{m-2}\not\in\Delta_L$. Furthermore, either $x_{m-1}\in\Delta_L$ or, for any $x\in U^{=1}\cap V(\mathring{Q})$, at least one internal vertex of $xQ$ lies in $\Delta_L$. \end{subclaim}

\begin{claimproof} Suppose $x_{m-2}\in\Delta_L$. Since $\Delta_L\subseteq\Delta^{2p}(C)$, it follows from 1) of Claim \ref{BtTwoStepClaim} that $x_{m-1}\not\in\Delta_L$ and thus, by Subclaim \ref{m-1midQSubCL1}, $x_m\in\Delta_L$. By 2) of Claim \ref{BtTwoStepClaim}, $x_{m-1}\in U^{=1}$. Since $U^{=1}\cap\textnormal{Mid}^{\downarrow}(Q)=\varnothing$, this contradicts Subclaim \ref{m-1midQSubCL1}. Thus, $x_{m-2}\not\in\Delta_L$. Now suppose Subclaim \ref{Ifxm-1NotInDeltaLThenSubCL5} does not hold. Since $x_{m-1}\not\in\Delta_L$, we have $x_m\in\Delta_L$ by Subclaim \ref{m-1midQSubCL1}. By assumption, there is an $\ell\in\{2, \cdots, m-1\}$, where $x_{\ell}\in U^{=1}$ and no internal vertex of $x_{\ell}Q$ lies in $\Delta_L$. As $x_{\ell}\in U^{=1}$, we have $x_{\ell}\not\in\textnormal{Mid}^{\downarrow}(Q)$. Consider the following cases. 

\textbf{Case 1:} $x_{\ell-1}\not\in\textnormal{Mid}^{\downarrow}(Q)$

By the minimality of $Q$, there is a $\phi\in\textnormal{Halve}^+(Qx_{\ell-1})$ with $\phi(x_1)=c$. It follows from Claim \ref{LinkColBlackBoxCh1} that there is a $\sigma\in\textnormal{Link}(x_{\ell}Q)$ with $\sigma(x_{\ell})\in L_{\psi\cup\phi}(x_{\ell})$.  Since $\sigma\in\textnormal{Link}(x_{\ell}Q)$ and neither $x_{\ell-1}$ nor $x_{\ell}$ lies in $\textnormal{Mid}^{\downarrow}(Q)$, we have $\textnormal{Tw}(\phi\cup\sigma)=\textnormal{Tw}(\phi)$. Thus, $\phi\cup\sigma\in\textnormal{Halve}^+(Q)$, contradicting our assumption on $c$. 

\textbf{Case 2:} $x_{\ell-1}\in\textnormal{Mid}^{\downarrow}(Q)$ 

In this case, $\ell-1>1$. By the minimality of $Q$, there is a $\phi\in\textnormal{Halve}^+(Qx_{\ell-2})$ with $\phi(x_1)=c$.

\textbf{Subcase 2.1:} $x_{\ell-1}\in\Delta_L$

 Since $|L_{\psi\cup\phi}(x_{\ell-1})|\geq 1$, there is a $d\in L_{\psi\cup\phi}(x_{\ell-1})$. Since $|L_{\psi}(x_{\ell})\setminus\{d\}|\geq 3$, it follows from Claim \ref{LinkColBlackBoxCh1} that there is a $\sigma\in\textnormal{Link}(x_{\ell}Q)$ with $\sigma(x_{\ell})\neq d$. Thus, $\phi\cup\sigma$ extends to a proper $L_{\psi}$-coloring $\phi^*$ of $\textnormal{dom}(\phi\cup\sigma)\cup\{x_{\ell-1}\}$ with $\phi^*(x_{\ell-1})=d$. Let $y$ be the unique vertex of $B_2(C)$ such that $G[N(y)\cap V(C^*)=x_{\ell-2}x_{\ell-1}x_{\ell}$. Since $\sigma\in\textnormal{Link}_{L_{\psi}}(x_{\ell}Q, C^*, \textnormal{Int}(C^*))$, we have $\textnormal{Tw}(\phi^*)\subseteq\textnormal{Tw}(\phi)\cup\{y\}$. 

By assumption, we have $|\textnormal{Tw}(\phi)|\leq\frac{|\Delta_L\cap V(Qx_{\ell-2})|}{2}$. We also have $\Delta_L\cap V(x_{\ell-1}Q)=\{x_{\ell-1}, x_m\}$. Since $\left\lceil\frac{|\Delta_L\cap V(Qx_{\ell-2})|+2}{2}\right\rceil=\left\lceil\frac{|\Delta_L\cap V(Qx_{\ell-2})|}{2}\right\rceil+1$ and $|\textnormal{Tw}(\phi^*)|\leq |\textnormal{Tw}(\phi)|+1$, we have $\phi^*\in\textnormal{Halve}^+(Q)$, contradicting our assumption on $c$.

\textbf{Subcase 2.2:} $x_{\ell-1}\not\in\Delta_L$

In this case, we have $|L_{\psi\cup\phi}(x_{\ell-1})|\geq 2$. Since $|L_{\psi}(x_{\ell})|\geq 4$, there is a set two colors $\{a_0, a_1\}\subseteq L_{\psi}(x_{\ell})$ such that, for each $i=0,1$, $|L_{\psi\cup\phi}(x_{\ell-1})\setminus\{a_i\}|\geq 2$. By Claim \ref{LinkColBlackBoxCh1}, since $|L_{\psi}(x_m)|\geq 2$ and $\Delta_L\cap V(x_{\ell+1}Q)=\{x_m\}$, there is a $\tau\in\textnormal{Link}(x_{\ell}Q)$ such that $\tau(x_{\ell})\in\{a_0, a_1\}$. Note that $\{x_{\ell-1}\}$ is $(L, \psi\cup\phi\cup\tau)$-inert in $G$. Now, $x_{\ell-1}$ is uncolored, and since $\tau\in\textnormal{Link}_{\psi}(x_{\ell}Q, C^*, \textnormal{Int}(C^*))$ and $x_{\ell}\not\in\textnormal{Mid}^{\downarrow}(Q)$, we have $\textnormal{Tw}(\phi\cup\tau)=\textnormal{Tw}(\phi)$. Since $\phi\in\textnormal{Halve}^+(Qx_{\ell-2})$, it follows that $\phi\cup\tau\in\textnormal{Halve}^+(Q)$, contradicting our assumption on $c$. \end{claimproof}\end{addmargin}

\vspace*{-8mm}
\begin{addmargin}[2em]{0em} 
\begin{subclaim}\label{ListEllFactSubCL3} For any $\ell\in\{2,\cdots, m-1\}$, at least one of the following holds.
\begin{enumerate}[label=\arabic*)]
\itemsep-0.1em
\item $x_{\ell}\in\Delta_L$; OR
\item $x_{\ell}\in\textnormal{Mid}^{\downarrow}(Q)$ and $x_{\ell+1}\in\Delta_L$; OR
\item Each of $|\Delta_L\cap V(Qx_{\ell})|$ and $|\Delta_L\cap V(x_{\ell}Q)|$ is odd.  
\end{enumerate}
\end{subclaim}

\begin{claimproof}  Suppose there is an $\ell\in\{2, \cdots, m-1\}$ for which none of 1), 2), or 3) holds. We break this into two cases. 

\textbf{Case 1:} $x_{\ell}\not\in\textnormal{Mid}^{\downarrow}(Q)$

Since $\ell<m$, it follows from the minimality of $Q$ that there is a $\phi\in\textnormal{Halve}^+(Qx_{\ell})$ with $\phi(x_1)=c$. Likewise, since $\ell>1$ and $x_{\ell}\in\textnormal{dom}(\phi)$, it follows from the minimality of $Q$ that there is a $\sigma\in\textnormal{Halve}^+(x_{\ell}Q)$ with $\sigma(x_{\ell})=\phi(x_{\ell})$. As $x_{\ell}\not\in\textnormal{Mid}^{\downarrow}(Q)$, we have $\textnormal{Tw}(\phi\cup\sigma)=\textnormal{Tw}(\phi)\cup\textnormal{Tw}(\sigma)$ as a disjoint union. Let $A:=\Delta_L\cap V(Qx_{\ell})$ and let $B:=\Delta_L\cap V(x_{\ell}Q)$. Since $\ell$ does not satisfy 1), $x_{\ell}\not\in\Delta_L$, so we have $\Delta_L\cap V(Q)=A\cup B$ as a disjoint union. Now, $|\textnormal{Tw}(\phi)|\leq\left\lceil\frac{|A|}{2}\right\rceil$ and $|\textnormal{Tw}(\sigma)|\leq\left\lceil\frac{|B|}{2}\right\rceil$. Since $\ell$ does not satisfy 3) of Subclaim \ref{ListEllFactSubCL3}, at least one of $|A|, |B|$ is even. In particular, we have $\left\lceil\frac{|A|+|B|}{2}\right\rceil=\left\lceil\frac{|A|}{2}\right\rceil+\left\lceil\frac{|B|}{2}\right\rceil$. Since $|\Delta_L\cap V(Q)=|A|+|B|$, we have $|\textnormal{Tw}(\phi\cup\sigma)|\leq\left\lceil\frac{|\Delta_L\cap V(Q)|}{2}\right\rceil$, so we get $\phi\cup\sigma\in\textnormal{Halve}^+(Q)$, contradicting our assumption on $c$. 

\textbf{Case 2:} $x_{\ell}\in\textnormal{Mid}^{\downarrow}(Q)$ 

Since $\ell$ does not satisfy does not satisfy 2) of Subclaim \ref{ListEllFactSubCL3}, we have $x_{\ell+1}\not\in\Delta_L$. As $\ell>1$, it follows from the minimality of $Q$, there is a $\phi'\in\textnormal{Halve}^+(Qx_{\ell-1})$ with $\phi'(x_1)=c$. Since $x_{\ell}\not\in\Delta_L$, we have $|L_{\psi\cup\phi'}(x_{\ell})|\geq 2$. Since $x_{\ell+1}\not\in\Delta_L$, we have $|L_{\psi}(x_{\ell+1})|\geq 3$, so there is a $d\in L_{\psi}(x_{\ell+1})$ with $|L_{\psi\cup\phi'}(x_{\ell})\setminus\{d\}|\geq 2$. By the minimality of $Q$, there is a $\sigma'\in\textnormal{Halve}^+(x_{\ell+1}Q)$ with $\sigma'(x_{\ell+1})=d$. As $x_{\ell}\in\textnormal{Mid}^{\downarrow}(Q)$, it follows from our choice of $d$ that $\{x_{\ell}\}$ is $(L, \psi\cup\phi'\cup\sigma')$-inert in $G$. We have $\textnormal{Tw}(\phi'\cup\sigma')=\textnormal{Tw}(\phi')\cup\textnormal{Tw}(\sigma')$ as a disjoint union. Let $A':=\Delta_L\cap V(Qx_{\ell-1})$ and let $B':=\Delta_L\cap V(x_{\ell+1}Q)$. Since 1) is not satisfied, $x_{\ell}\not\in\Delta_L$, so $\Delta_L\cap V(Q)=A'\cup B'$ as a disjoint union. We have $|\textnormal{Tw}(\phi')|\leq\left\lceil\frac{|A'|}{2}\right\rceil$ and $|\textnormal{Tw}(\sigma')|\leq\left\lceil\frac{|B'|}{2}\right\rceil$. As $\ell$ does not satisfy 3) of Subclaim \ref{ListEllFactSubCL3} and $x_{\ell}\not\in\Delta_L$, at least one of $|A'|, |B'|$ is even. Thus, $\left\lceil\frac{|A'|+|B'|}{2}\right\rceil=\left\lceil\frac{|A'|}{2}\right\rceil+\left\lceil\frac{|B'|}{2}\right\rceil$. Since $|\Delta_L\cap V(Q)|=|A'|+|B'|$, we have $|\textnormal{Tw}(\phi'\cup\sigma')|\leq\left\lceil\frac{|\Delta_L\cap V(Q)|}{2}\right\rceil$, so we get $\phi'\cup\sigma'\in\textnormal{Halve}^+(Q)$, contradicting our assumption on $c$. \end{claimproof}\end{addmargin}

We have an analogous fact to Subclaim \ref{m-1midQSubCL1} for the other side.

\vspace*{-8mm}
\begin{addmargin}[2em]{0em}
\begin{subclaim}\label{3FactSubCl2List}
All three of the following hold:
\begin{enumerate}[label=\arabic*)]
\itemsep-0.1em
\item $x_2\in\textnormal{Mid}^{\downarrow}(Q)$; AND
\item Precisely one of $x_2, x_3$ lies in $\Delta_L$; AND
\item $|V(Q)|>5$ and $\Delta_L\cap V(Q-x_1)\neq \{x_2, x_{m-1}\}$
\end{enumerate}  \end{subclaim}

\begin{claimproof} We first show that $x_2\in\textnormal{Mid}^{\downarrow}(Q)$. Suppose not. Since $|L_{psi}(x_2)\setminus\{c\}|\geq 1$, it follows from the minimality of $Q$ that there is a $\phi\in\textnormal{Halve}^+(Q-x_1)$ with $\phi(x_2)\neq c$. Now, $\phi$ extends to an $L_{\psi}$-coloring $\phi^*$ of $\cup\{x_1\}$ with $\phi^*(x_1)=c$. Since $x_2\not\in\textnormal{Mid}^{\downarrow}(Q)$, we have $\textnormal{Tw}(\phi^*)=\textnormal{Tw}(\phi)$, and since $|\textnormal{Tw}(\phi)|\leq\left\lceil\frac{|\Delta_L\cap V(Q-x_1)|}{2}\right\rceil$, we have $|\textnormal{Tw}(\phi^*)|\leq\left\lceil\frac{|\Delta_L\cap V(Q)|}{2}\right\rceil$, so $\phi^*\in\textnormal{Halve}^+(Q)$, contradicting our assumption on $c$. Thus, $x_2\in\textnormal{Mid}^{\downarrow}(Q)$. This proves 1). Now suppose that $|\{x_2, x_3\}\cap\Delta_L|\neq 1$. By 1) of Claim \ref{BtTwoStepClaim}, $\{x_2, x_3\}\cap\Delta_L=\varnothing$, and each of $x_2, x_3$ has an $L_{\psi}$-list of size at least three. Thus, there is a $d\in L_{\psi}(x_3)$ with $|L_{\psi}(x_2)\setminus\{c,d\}|\geq 2$. By minimality, there is a $\sigma\in\textnormal{Halve}^+(x_3Q)$ with $\sigma(x_3)=d$, and $\sigma$ extends to an $L$-coloring $\sigma^*$ of $\{x_1\}\cup\textnormal{dom}(\sigma)$ with $\sigma^*(x_1)=c$. Since $x_2\in\textnormal{Mid}^{\downarrow}(Q)$, it follows that $\{x_2\}$ is $(L, \sigma^*)$-inert in $G$. Furthermore, $\textnormal{Tw}(\sigma^*)=\textnormal{Tw}(\sigma)$, so $\sigma^*\in\textnormal{Halve}^+(Q)$, contradicting our assumption on $c$. 

Now we prove 3). We first show that $|V(Q)|\leq 5$. Suppose not. By Subclaim \ref{m-1midQSubCL1}, $x_{m-1}\in\textnormal{Mid}^{\downarrow}(Q)$. By 1), $x_2\in\textnormal{Mid}^{\downarrow}(Q)$. By Subclaim \ref{trivcaseleq3subCL}, $m-1>2$, and since $x_2, x_{m-1}\in\textnormal{Mid}^{\downarrow}(Q)$, we have $|V(Q)|=5$. Thus, again by Subclaim \ref{m-1midQSubCL1}, $|\{x_4, x_5\}\cap\Delta_L|=1$. As $\Delta_L$ has nonempty intersection with each of $\{x_2, x_3\}$ and $\{x_4, x_5\}$, and since $U^{=1}\cap\textnormal{Mid}^{\downarrow}(Q)=\varnothing$, it follows from 2) of Claim \ref{BtTwoStepClaim} that $x_3\in U^{=1}$. Since $U^{=1}\cap\Delta_L=\varnothing$, we have $x_2\in\Delta_L$. Since $|L_{\psi}(x_2)\setminus\{c\}|\geq 1$, there is a $d\in L_{\psi}(x_2)\setminus\{c\}$. Since $x_3\in U^{=1}$, we have $|L_{\psi}(x_3)\setminus\{d\}|\geq 3$. Consider the following cases. 

\textbf{Case 1:} $x_4\in\Delta_L$

In this case, $\Delta_L\cap V(Q)=\{x_2, x_4\}$. Since $|L_{\psi}(x_5)|\geq 3$ and $|L_{\psi}(x_3)\setminus\{d\}|\geq 3$, there is an $L_{\psi}$-coloring $\phi^*$ of $V(Q)\setminus\{x_4\}$ such that $\phi^*(x_1)=c$ and $\phi^*(x_2)=d$, where neither $\phi(x_3)$ nor $\phi(x_5)$ lies in $L_{\psi}(x_4)$. Thus, $\{x_4\}$ is $(L, \psi\cup\phi^*)$-inert in $G$ and $|\textnormal{Tw}(\phi^*)|\geq 1$. Since $|\Delta_L\cap V(Q)|\geq 2$, we have $\phi^*\in\textnormal{Halve}^+(Q)$, contradicting our assumption on $c$. 

\textbf{Case 2:} $x_4\not\in\Delta_L$

In this case, we have $\Delta_L\cap V(Q)=\{x_2, x_5\}$ and, in particular, $|L_{\psi}(x_4)|\geq 3$. Since $|L_{\psi}(x_3)\setminus\{d\}|\geq 3$ and $|L_{\psi}(x_5)|\geq 2$, there is an $L_{\psi}$-coloring $\phi^*$ of $V(Q)\setminus\{x_4\}$ such that $\phi^*(x_1)=c$ and $\phi^*(x_2)=d$, where $|L_{\psi\cup\phi^*}(x_4)|\geq 2$. Thus, $\{x_4\}$ is $(L, \psi\cup\phi^*)$-inert in $G$ and $|\textnormal{Tw}(\phi^*)|\geq 1$. Since $|\Delta_L\cap V(Q)|\geq 2$, we have $\phi^*\in\textnormal{Halve}^+(Q)$, contradicting our assumption on $c$. 

Thus, we have $|V(Q)|>5$, as desired. To finish, we prove that $\Delta_L\cap V(Q-x_1)\neq \{x_2, x_{m-1}\}$. Suppose not. As above, we fix a $d\in L_{\psi}(x_2)\setminus\{c\}$. Let $A:=L_{\psi}(x_3)\setminus\{d\}$. Since $x_{m-1}\in\Delta_L$, each of $x_{m-2}, x_m$ has an $L_{\psi}$-list of size at least three, so let $f\in L_{\psi}(x_m)\setminus L_{\psi}(x_{m-1})$ and let $A':=L_{\psi}(x_{m-2})\setminus L_{\psi}(x_{m-1})$. Note that $|A'|\geq 1$ and $|A|\geq 2$. Since $|V(Q)|>5$, the path $x_3Qx_{m-2}$ has length at least one. We claim now that there is $\sigma\in\textnormal{Link}(x_3Qx_{m-2})$ with $\sigma(x_3)\in A$ and $\sigma(x_{m-2})\in A'$. If $|A|+|A'|\geq 4$, then this follows from 2i) of Claim \ref{LinkColBlackBoxCh1}, so suppose that $|A|+|A'|\leq 3$. In this case, $|A|=2$ and $|A'|=1$, and thus $|L_{\psi}(x_{m-2})|=3$. Likewise, since $|A|=2$ and $x_3\not\in\Delta_L$, we have $|L_{\psi}(x_3)|=3$. By 2) of Claim \ref{BtTwoStepClaim}, there is a $j\in\{4, \cdots, m-3\}$ with $x_j\in U^{=1}$, since $\Delta_L\cap V(Q-x_1)\neq \{x_2, x_{m-1}\}$. Since $U^{=1}\cap\textnormal{Mid}^{\downarrow}(Q)=\varnothing$ and $|L_{\psi}(x_j)|\geq 4$, it follows from two applications of Claim \ref{LinkColBlackBoxCh1} that such a $\sigma$ exists. In any case, there is a $\sigma\in\textnormal{Link}(x_3Qx_{m-2})$ with $\sigma(x_3)\in A$ and $\sigma(x_{m-2})\in A'$. Now, $\sigma$ extends to an $L_{\psi}$-coloring $\sigma^*$ of $\textnormal{dom}(\sigma)\cup\{x_1, x_2, x_m\}$, where $\sigma^*(x_m)=f$ and $\sigma^*(x_1)=c$. As $\sigma^*(x_{m-2})\in A'$, the set $\{x_{m-1}\}$ is $(L, \psi\cup\sigma^*)$-inert in $G$. Since $x_{m-1}$ is uncolored and $\textnormal{Tw}(\sigma)=\varnothing$, we have $|\textnormal{Tw}(\sigma^*)|\leq 1$, as $x_3, x_{m-2}\not\in\textnormal{Mid}^{\downarrow}(Q)$. Since $|\Delta_L\cap V(Q)|=2$, we have $\sigma^*\in\textnormal{Halve}^+(Q)$, contradicting our assumption on $c$.  \end{claimproof}\end{addmargin}

\vspace*{-8mm}
\begin{addmargin}[2em]{0em}\begin{subclaim}\label{xSubm-1inDeltaLSubCL4} $x_{m-1}\in\Delta_L$. \end{subclaim}

\begin{claimproof} Suppose not. Recall that, by Subclaim \ref{m-1midQSubCL1}, $x_m\in\Delta_L$. By 3) of Subclaim \ref{3FactSubCl2List}, $m>5$. We first show that $\Delta_L\cap V(\mathring{Q})=\varnothing$. Suppose not, and let $k$ be the maximal index among $\{2, \cdots, m-1\}$ such that $x_k\in\Delta_L$. By 2) of Claim \ref{BtTwoStepClaim}, at least one internal vertex of $x_kQ$ lies in $U^{=1}$, so there is an $\ell\in\{k+1, \cdots, m-1\}$ with $x_{\ell}\in U^{=1}$. But then, by Subclaim \ref{Ifxm-1NotInDeltaLThenSubCL5}, at least one internal vertex of $x_{\ell}Q$ lies in $\Delta_L$, contradicting our choice of $k$. We conclude that $\Delta_L\cap V(\mathring{Q})=\varnothing$, as desired. Since $\Delta_L\cap V(\mathring{Q})=\varnothing$, it follows from Claim \ref{LinkColBlackBoxCh1} that there is a $\phi\in\textnormal{Link}(Q-x_m)$ such that $\phi(x_1)=c$. Since $|L_{\psi\cup\phi}(x_m)|\geq 1$, there is an extension of $\phi$ to an $L_{\psi}$-coloring $\phi^*$ of $\textnormal{dom}(\phi)\cup\{x_m\}$. By Subclaim \ref{m-1midQSubCL1}, $x_{m-1}\in\textnormal{Mid}^{\downarrow}(Q)$, and there is a $y\in B_2(C)$ with $G[N(y)\cap V(C^*)]=x_{m-2}x_{m-1}x_m$. Since $\phi\in\textnormal{Link}_{L_{\psi}}(Q-x_m, C^*, \textnormal{Int}(C^*))$, we have $\textnormal{Tw}\subseteq\{y\}$, and since $1\leq |\Delta_L\cap V(Q)|\leq 2$, we have $\left\lceil\frac{\Delta_L\cap V(Q)}{2}\right\rceil=1$, so $\phi^*\in\textnormal{Halve}^+(Q)$, contradicting our assumption on $c$. \end{claimproof}\end{addmargin}

\vspace*{-8mm}
\begin{addmargin}[2em]{0em}
\begin{subclaim}\label{PathFrom2tom-2SubCL4} $V(x_2Qx_{m-2})\cap U^{=1}\neq\varnothing$ and furthermore, $x_{m-3}\in\textnormal{Mid}^{\downarrow}(Q)$ \end{subclaim}

\begin{claimproof} Suppose toward a contradiction that $V(x_2Qx_{m-2})\cap U^{=1}=\varnothing$. By Subclaim \ref{xSubm-1inDeltaLSubCL4}, we have $x_{m-1}\in\Delta_L$, so it follows from 2) of Claim \ref{BtTwoStepClaim} that $\Delta_L\cap V(Qx_{m-3})=\varnothing$. By Subclaim \ref{Ifxm-1NotInDeltaLThenSubCL5}, $x_{m-2}\not\in\Delta_L$, so $\Delta_L\cap V(Qx_{m-2})=\varnothing$ and, since $x_m\not\in\Delta_L$, we have $\Delta_L\cap V(Q)=\{x_{m-1}\}$, so every vertex of $x_2Qx_{m-2}$ has an $L_{\psi}$-list of size at least three. It now follows from Claim \ref{LinkColBlackBoxCh1} that there is a $\phi\in\textnormal{Link}(Qx_{m-2})$ with $\phi(x_1)=c$. Let $\phi^*$ be an extension of $\phi$ to an $L_{\psi}$-coloring of $\textnormal{dom}(\phi)\cup\{x_{m-1}, x_m\}$. We have $x_{m-1}\in\textnormal{Mid}^{\downarrow}(Q)$, but since $x_{m-2}, x_m\not\in\textnormal{Mid}^{\downarrow}(Q)$ and $\textnormal{Tw}(\phi)=\varnothing$, we get $|\textnormal{Tw}(\phi^*)|\leq 1$. Since $\left\lceil\frac{|\Delta_L\cap V(Q)|}{2}\right\rceil=1$, we have $\phi^*\in\textnormal{Halve}^+(Q)$, contradicting our assumption on $c$. We conclude that $V(x_2Qx_{m-2})\cap U^{=1}\neq\varnothing$, as desired.

Now we show that $x_{m-3}\in\textnormal{Mid}^{\downarrow}(Q)$. Since $|L_{\psi}(x_{m-2})|\geq 3$ and $|L_{\psi}(x_m)|\geq 3$, we fix $f\in L_{\psi}(x_{m-2})$ and $f'\in L_{\psi}(x_m)$ with $f, f'\not\in L_{\psi}(x_{m-1})$. We break the remainder of the proof of Subclaim \ref{PathFrom2tom-2SubCL4} into two cases.

\textbf{Case 1:} $x_{m-2}\in U^{=1}$

By the minimality of $Q$, there is a $\sigma\in\textnormal{Halve}^+(Qx_{m-3})$ with $\sigma(x_1)=c$. Since $x_{m-2}\in U^{=1}$, we have $|L_{\psi\cup\sigma}(x_{m-2})|\geq 3$. Since $|L_{\psi}(x_m)|\geq 3$, there is an extension of $\sigma$ to an $L_{\psi}$-coloring $\sigma^*$ of $\textnormal{dom}(\sigma)\cup\{x_{m-2}, x_m\}$ such that $\sigma^*(x_{m-2})=f$ and $\sigma^*(x_m)=f'$. Thus, $\{x_{m-1}\}$ is $(L, \psi\cup\sigma^*)$-inert in $G$. Since $x_{m-3}\not\in\textnormal{Mid}^{\downarrow}(Q)$ by assumption, and $x_{m-1}$ is uncolored, we have $\textnormal{Tw}(\sigma^*)=\textnormal{Tw}(\sigma)$, so $\sigma^*\in\textnormal{Halve}^+(Q)$, contradicting our choice of $c$. 

\textbf{Case 2:} $x_{m-2}\not\in U^{=1}$

Since $V(x_Qx_{m-2})\cap U^{=1}\neq\varnothing$, let $r$ be the maximal index among $\{2, \cdots, m-2\}$ such that $x_r\in U^{=1}$. By the assumption of Case 2, $r<m-2$. Furthermore, since $U^{=1}\cap\Delta_L=\varnothing$ and $U^{=1}\cap\textnormal{Mid}^{\downarrow}(Q)=\varnothing$, we have $x_r\not\in\Delta_L$ and $x_{m-1}\not\in U^{=1}$. Since $x_{m-1}\in\Delta_L$ and $x_m\not\in\Delta_L$, we have $\Delta_L\cap V(x_rQ)=\{x_{m-1}\}$, or else it follows from 2) of Claim \ref{BtTwoStepClaim} that there is an $r'\in\{r+1, \cdots, m-2\}$ with $x_{r'}\in U^{=1}$, contradicting our choice of $r$. We break this into two cases. 

\textbf{Subcase 2.1} $x_{r-1}\not\in\textnormal{Mid}^{\downarrow}(Q)$

By the minimality of $Q$, there is a $\zeta\in\textnormal{Halve}^+(Qx_{r-1})$ with $\zeta(x_1)=c$, and $|L_{\psi\cup\zeta}(x_r)|\geq 3$. By Claim \ref{LinkColBlackBoxCh1}, there is a $\varphi\in\textnormal{Link}(x_rQx_{m-2})$ with $\varphi(x_{m-2})=f$ and $\varphi(x_r)\in L_{\psi\cup\zeta}(x_r)$. Thus, $\zeta\cup\varphi$ is a proper $L_{\psi}$-coloring of its domain and extends to an $L_{\psi}$-coloring $\varphi'$ of $\textnormal{dom}(\zeta\cup\varphi)\cup\{x_r\}$ with $\varphi'(x_r)=f'$. Since $x_{r-1}, x_r\not\in\textnormal{Mid}^{\downarrow}(Q)$ and $x_{m-1}$ is uncolored, we have $\textnormal{Tw}(\varphi')=\textnormal{Tw}(\varphi)$. By our choice of $f, f'$, $\{x_{m-1}\}$ is $(L, \psi\cup\varphi')$-inert in $G$. Thus, $\varphi'\in\textnormal{Halve}^+(Q)$, contradicting our choice of $c$.  

\textbf{Subcase 2.2} $x_{r-1}\in\textnormal{Mid}^{\downarrow}(Q)$

In this case, we have $r-1>1$. By the minimality of $Q$, there is a $\tau\in\textnormal{Halve}^+(Qx_{r-2})$ with $\tau(x_1)=c$. Since $x_{r-1}\in\textnormal{Mid}^{\downarrow}(Q)$, we have $x_{r-2}\not\in\textnormal{Mid}^{\downarrow}(Q)$. We note now that $x_{r-1}\not\in\Delta_L$. Suppose that $x_{r-1}\in\Delta_L$. By 1) of Claim \ref{BtTwoStepClaim}, we have $x_{r-2}\not\in\Delta_L$, and thus $\Delta_L\cap V(Qx_{r-2})=\{x_{r-1}, x_{m-1}\}$. If $r-2>1$, then, since $|\Delta_L\cap V(Qx_{r-2})|$ is even, we contradict Subclaim \ref{ListEllFactSubCL3}, where we take $\ell=r-2$ in the statement of Subclaim \ref{ListEllFactSubCL3}. On the other hand, if $r-2\leq 1$, then, since $x_{r-1}\in\textnormal{Mid}^{\downarrow}(Q)$, we have $r=3$ and $\Delta_L\cap V(Q-x_1)=\{x_{r-1}, x_{m-1}\}$, contradicting 3) of Subclaim \ref{3FactSubCl2List}. We conclude that $x_{r-1}\not\in\Delta_L$, so we get $\Delta_L\cap V(x_{r-1}Q)=\{x_{m-1}\}$. 

Since $|L_{\psi}(x_r)|\geq 4$, there exist two colors $\{a_0, a_1\}\subseteq L_{\psi}(x_r)\setminus L_{\psi}(x_{r-1})$ such that, for each $i=0,1$, $|L_{\psi\cup\tau}(x_{r-1})\setminus\{a_i\}|\geq 2$. Now, by Claim \ref{LinkColBlackBoxCh1}, there is a $\varphi\in\textnormal{Link}(x_rQx_{m-3})$ with $\varphi(x_r)\in\{a_0, a_1\}$ and $\varphi(x_{m-3})\in L_{\psi}(x_{m-3})\setminus\{f\}$. Thus, $\tau\cup\varphi$ is a proper $L_{\psi}$-coloring of its domain and extends to an $L_{\psi}$-coloring $\varphi^{\dagger}$ of $\textnormal{dom}(\tau\cup\varphi)\cup\{x_{m-2}, x_m\}$ such that $\varphi^{\dagger}(x_{m-2})=f$ and $\varphi^{\dagger}(x_m)=f'$. By our choice of $f,f'$ and $\varphi(x_r)$, we get that $\{x_{r-1}, x_{m-1}\}$ is $(L, \psi\cup\varphi^{\dagger})$-inert in $G$. By assumption, $x_{m-3}\not\in\textnormal{Mid}^{\downarrow}(Q)$ and $x_{m-2}, x_r\not\in\textnormal{Mid}^{\downarrow}(Q)$ as well. Since $x_{r-1}, x_{m-1}$ are uncolored and $\varphi\in\textnormal{Link}(x_rQx_{m-3})$ it follows that $\textnormal{Tw}(\varphi^{\dagger})=\textnormal{Tw}(\tau)$, so $\varphi^{\dagger}\in\textnormal{Halve}^+(Q)$, contradicting our choice of $c$. This completes the proof of Subclaim \ref{PathFrom2tom-2SubCL4}. \end{claimproof}\end{addmargin}

Combining Subclaims \ref{xSubm-1inDeltaLSubCL4} and \ref{PathFrom2tom-2SubCL4}, we get that $x_{m-1}, x_{m-3}\in\textnormal{Mid}^{\downarrow}(Q)$, so we have $|V(Q)|\geq 5$.  Furthermore, since $x_{m-1}, x_{m-3}\in\Delta_L$, we have $x_{m-1}, x_{m-3}\in\Delta^{2p}(C)$. In particular, by 2) of Claim \ref{BtTwoStepClaim}, $x_{m-2}\in U^{=1}$ and $|L_{\psi}(x_{m-2})|\geq 4$. By 1) of Claim \ref{BtTwoStepClaim}, we have $\Delta_L\cap V(x_{m-3}Q)=\{x_{m-3}, x_{m-1}\}$. Since $x_{m-3}\in\Delta_L$, we have $x_{m-4}\not\in\Delta_L$ by 1) of Claim \ref{BtTwoStepClaim}, and, in particular, $\Delta_L\cap V(x_{m-4}Q)=\{x_{m-3}, x_{m-1}\}$. Since $x_{m-3}\in\textnormal{Mid}^{\downarrow}(Q)$, we have $x_{m-4}\not\in\textnormal{Mid}(Q)$. By 3) of Subclaim \ref{3FactSubCl2List}, $|V(Q)|>5$, so $m-4>1$, and thus we contradict Subclaim \ref{ListEllFactSubCL3}, where we take $\ell=m-4$ in the statement of Subclaim \ref{ListEllFactSubCL3}. This completes the proof of Claim \ref{MainClaimforIntermPropSSF}. \end{claimproof}

Analogous to the notation $\textnormal{Halve}^+(Q)$, we now introduce the following natural notation, where the ceilings have been replaced by floors. Given a subpath $Q$ of $C^*\setminus C$, let $\textnormal{Halve}^{-}(Q)$ denote the set of partial $L_{\psi}$-colorings $\phi$ of $V(Q)$ such that the following hold.
\begin{enumerate}[label=\arabic*)]
\itemsep-0.1em
\item\emph{$V(Q)\setminus\textnormal{dom}(\phi)$ is $(L, \psi\cup\phi)$-inert in $G$ and a subset of $\textnormal{Mid}^{\downarrow}(Q)$}; AND
\item $|\textnormal{Tw}(\phi)|\leq\left\lfloor\frac{|\Delta_L\cap V(Q)|}{2}\right\rfloor$.
\end{enumerate}

\begin{claim}\label{HalveMinusNonEmptyCL} For any induced subpath $Q$ of $C^*\setminus C$, we have $\textnormal{Halve}^-(Q)\neq\varnothing$.  \end{claim}

\begin{claimproof} Suppose not and let $Q:=x_1\cdots x_m$ be a vertex-minimal counterexample to the claim. 

\vspace*{-8mm}
\begin{addmargin}[2em]{0em} 
\begin{subclaim}\label{TrivialEdgeCaseSubCL3} $|V(Q)|\geq 3$ and $x_{m-1}\in \textnormal{Mid}^{\downarrow}(Q)$. \end{subclaim}

\begin{claimproof} If $|V(Q)|\leq 2$ then any $L_{\psi}$-coloring of $Q$ lies in $\textnormal{Halve}^-(Q)$, contradicting our assumption, so $|V(Q)|\geq 3$. Suppose toward a contradiction that $x_{m-1}\not\in\textnormal{Mid}^{\downarrow}(Q)$. By the minimality of $Q$, there is a $\phi\in\textnormal{Halve}^-(Q)$. Since $Q$ is an induced subgraph of $G$, we have $|L_{\psi\cup\phi}(x_m)|\geq 1$, so $\phi$ extends to an $L_{\psi}$-coloring $\phi^*$ of $\textnormal{dom}(\phi)\cup\{x_m\}$. Since $x_{m-1}\not\in\textnormal{Mid}^{\downarrow}(Q)$, we have $\textnormal{Tw}(\phi^*)=\textnormal{Tw}(\phi)$, so $\phi^*\in\textnormal{Halve}^-(Q)$, contradicting our assumption.  \end{claimproof}\end{addmargin}

\vspace*{-8mm}
\begin{addmargin}[2em]{0em}
\begin{subclaim}\label{DeltaLInCAtMost1SubCL} $\Delta_L\cap\{x_{m-1}, x_m\}=\{x_m\}$.  \end{subclaim}

\begin{claimproof} We first show that $x_{m-1}\not\in\Delta_L$. Suppose toward a contradiction that $x_{m-1}\in\Delta_L$. By 1) of Claim \ref{BtTwoStepClaim}, $x_{m-2}, x_m\not\in\Delta_L$, so let $c\in$ and $d\in$, where $c,d\not\in L_{\psi}(x_{m-1})$. By Claim \ref{MainClaimforIntermPropSSF}, there is a $\phi\in\textnormal{Halve}^+(Qx_{m-2})$ with $\phi(x_{m-2})=c$, so let $\phi^*$ be an extension of $\phi$ to an $L_{\psi}$-coloring of $\textnormal{dom}(\phi)\cup\{x_m\}$, where $\phi^*(x_m)=d$. 

By Subclaim \ref{TrivialEdgeCaseSubCL3}, $x_{m-1}\in\textnormal{Mid}^{\downarrow}(Q)$, so $\{x_{m-1}\}$ is $(L, \psi\cup\phi^*)$-inert in $G$ by our choice of $c,d$. Since $x_{m-1}\in\Delta_L$ and $x_m\not\in\Delta_L$, we have $|\Delta_L\cap V(Qx_{m-2})|=|\Delta_L\cap V(Q)|-1$. Since $\phi\in\textnormal{Halve}^+(Qx_{m-2})$ and $x_{m-1}$ is uncolored, we have $|\textnormal{Tw}(\phi^*)|\leq\left\lceil\frac{|\Delta_L\cap V(Q)|-1}{2}\right\rceil$. Note that $\left\lceil\frac{|\Delta_L\cap V(Q)|-1}{2}\right\rceil=\left\lfloor\frac{|\Delta_L\cap V(Q)|}{2}\right\rfloor$, so $\phi^*\in\textnormal{Halve}^-(Q)$, contradicting our assumption. We conclude that $x_{m-1}\not\in\Delta_L$. 

Now suppose toward a contradiction that $x_m\not\in\Delta_L$. As shown above, $x_{m-1}\not\in\Delta_L$ as well, so each of $x_{m-1}, x_m$ has an $L_{\psi}$-list of size at least three. By the minimality of $Q$, there is a $\sigma\in\textnormal{Halve}^-(Qx_{m-2})$. Since $Q$ is an induced subgraph of $G$ and each of $x_{m-1}, x_m$ has an $L_{\psi}$-list of size at least three, it follows that $\sigma$ extends to an $L_{\psi}$-coloring $\sigma^*$ of $\textnormal{dom}(\sigma)\cup\{x_m\}$ such that $|L_{\psi\cup\sigma^*}(x_{m-1})|\geq 2$. Thus, $\{x_{m-1}\}$ is $(L, \psi\cup\sigma^*)$-inert in $G$ and $\textnormal{Tw}(\sigma^*)=\textnormal{Tw}(\sigma)$. It follows that $\sigma^*\in\textnormal{Halve}^-(Q)$, contradicting our assumption.  \end{claimproof}\end{addmargin}

Now we have enough to finish the proof of Claim \ref{HalveMinusNonEmptyCL}. By Subclaim \ref{DeltaLInCAtMost1SubCL}, $|L_{\psi}(x_m)|=2$ and $|L_{\psi}(x_{m-1})|\geq 3$. Since $|L_{\psi}(x_{m-2})|\geq 2$, there is a $c\in L_{\psi}(x_{m-2})$ such that $L_{\psi}(x_{m-1})\setminus\{c\}\neq L_{\psi}(x_m)$. By Claim \ref{MainClaimforIntermPropSSF}, there is a $\sigma\in\textnormal{Halve}^+(Qx_{m-2})$ with $\sigma(x_{m-2})=c$. Since $L_{\psi}(x_{m-1})\setminus\{c\}\neq L_{\psi}(x_m)$, there is an extension of $\sigma$ to an $L_{\psi}$-coloring $\sigma^*$ of $\textnormal{dom}(\sigma)\cup\{x_m\}$ such that $|L_{\psi\cup\sigma^*}(x_{m-1})|\geq 2$. Since $x_{m-1}\in\textnormal{Mid}^{\downarrow}(Q)$, $\{x_{m-1}\}$ is $(L, \psi\cup\sigma^*)$-inert in $G$. Furthermore, $|\Delta_L\cap V(Qx_{m-2})|=|\Delta_L\cap V(Q)|-1$, since $x_m\in\Delta_L$. 

Since $x_{m-1}$ is unolored, we have $\textnormal{Tw}(\sigma^*)=\textnormal{Tw}(\sigma)$. Finally, we have $|\textnormal{Tw}(\sigma)|\leq\left\lceil\frac{|\Delta_L\cap V(Q)|-1}{2}\right\rceil$. Since $\left\lceil\frac{|\Delta_L\cap V(Q)|-1}{2}\right\rceil=\left\lfloor\frac{|\Delta_L\cap V(Q)|}{2}\right\rfloor$, we have $\sigma^*\in\textnormal{Halve}^-(Q)$, contradicting our assumption. This completes the proof of Claim \ref{HalveMinusNonEmptyCL}. \end{claimproof}

With Claims \ref{MainClaimforIntermPropSSF} and \ref{HalveMinusNonEmptyCL} in hand, we are almost ready to complete the proof of Proposition \ref{OneStepIntermProp}. We have the following observation, which immediately follows from the fact that $\textit{NDepth}(C)\geq 2$. 
'
\begin{claim}\label{IfCC*nonemptyIntersecCaseCL}
Both of the following hold:
\begin{enumerate}[label=\arabic*)]
\itemsep-0.1em
\item If $C^*\cap C\neq\varnothing$, then each connected component of $C^*\setminus C$ is a path, and furthermore, for any two distinct connected components $P, P'$ of $C^*\setminus C$, we have $d_G(P, P')\geq 3$; AND
\item  If $C^*\cap C=\varnothing$ then $C^*=G[B_1(C)]$.
\end{enumerate}
  \end{claim}

We now introduce one final definition. A partial $L_{\psi}$-coloring $\sigma$ of $V(C^*\setminus C)$ is called a \emph{good coloring} if $|\Delta_L(C^*, \psi\cup\sigma)|\leq\left\lceil\frac{|\Delta_{L}}{2}\right\rceil$ and $V(C\cup C^*)\setminus\textnormal{dom}(\psi\cup\sigma)$ is $(L, \psi\cup\sigma)$-inert in $G$. To prove Proposition \ref{OneStepIntermProp}, it suffices to prove that there is a good coloring. Suppose now toward a contradiction that there is no good coloring.

\begin{claim} $C^*\cap C=\varnothing$. \end{claim}

\begin{claimproof} Suppose toward a contradiction that $C^*\cap C\neq\varnothing$. Since $\textit{NDepth}(C)\geq 2$, it follows from Definition \ref{splitandnonsplitdefinitions} that $\Delta^{2p}(C)\neq\varnothing$, so, in particular, we have $C^*\neq C$, and $C^*\setminus C$ is nonempty. By Claim \ref{IfCC*nonemptyIntersecCaseCL}, each connected component of $C^*\setminus C$ is an induced path in $G$. Let $P_1, \cdots, P_n$ be the connected components of $C^*\setminus C$. It follows from Claim \ref{HalveMinusNonEmptyCL} that, for each $i=1,\cdots, n$, there is a $\sigma_i\in\textnormal{Halve}^-(Q)$. Let $\sigma$ be the union $\sigma_1\cdots\sigma_n$. Again by Claim \ref{IfCC*nonemptyIntersecCaseCL}, $d(P, P')\geq 3$ for any two distinct connected components $P, P'$ of $C^*\setminus C$, so $\sigma$ is a proper $L_{\psi}$-coloring of its domain and $\textnormal{Tw}(\sigma)=\textnormal{Tw}(\sigma_1)\cup\cdots\cup\textnormal{Tw}(\sigma_n)$. Thus, we have
$$|\textnormal{Tw}(\sigma)|\leq\sum_{i=1}^n\left\lfloor\frac{|\Delta_L\cap V(P_i)|}{2}\right\rfloor\leq\left\lfloor\sum_{i=1}^n\frac{|\Delta_L\cap V(P_i)|}{2}\right\rfloor$$
Since $d(P, P')\geq 3$ for any two distinct connected components $P, P'$ of $C^*\setminus C$, we have the disjoint union $\Delta_L=\bigcup_{i=1}^n (\Delta_L\cap V(P_i))$, so we get $\textnormal{Tw}(\sigma)|\leq\left\lfloor\frac{|\Delta_L(C, \psi)|}{2}\right\rfloor$. Since $\textnormal{Tw}(\sigma)=\Delta_L(C^*, \psi\cup\sigma)$, and $V(C\cup C^*)\setminus\textnormal{dom}(\psi\cup\sigma)$ is $(L, \psi\cup\sigma)$-inert in $G$, it follows that $\sigma$ is a good coloring, contradicting our assumption. \end{claimproof}

By Claim \ref{IfCC*nonemptyIntersecCaseCL}, since $C^*\cap C=\varnothing$, we get that $C^*$ is the subgraph of $G$ induced by the set of vertices of distance precisely one from $C$. Furthermore, since $\Delta^{2p}(C)\neq\varnothing$, we have $U^{=1}\neq\varnothing$ as well, or else $|V(C^*)|<|V(C)|$, which is false. Note that since $\textit{Ndepth}(C)\geq 2$ and $G$ is short-inseparable, we have $|V(C^*)|>4$, or else $C^*$ is a 4-cycle which separates$ C$ from an element of $\Delta^{2p}(C^*)$. 

\begin{claim}\label{xinU1nbrsCL} For any $x\in U^{=1}$, neither neighbor of $x$ in $C^*$ lies in $\textnormal{Mid}^{\downarrow}(C^*)\setminus\Delta_L$. \end{claim}

\begin{claimproof} Suppose not. Thus, there is an edge $xx'\in E(C^*)$, where $x\in U^{=1}$ and $x'\in\textnormal{Mid}^{\downarrow}(C^*)\setminus\Delta_L$. Note that, since $\textit{NDepth}(C)\geq 2$, it follows that $C^*-x'$ is an induced subpath of $G$. Since $|L_{\psi}(x)|\geq 4$ and $x'\not\in\Delta_L$, there is a $c\in L_{\psi}(x)$ with $|L_{\psi}(x')\setminus\{c\}|\geq 3$.  By Claim \ref{MainClaimforIntermPropSSF}, there is a $\sigma\in\textnormal{Halve}^+(C^*-x')$ with $\sigma(x)=c$. Thus, $\{x'\}$ is $(L, \psi\cup\sigma)$-inert in $G$ by our choice of $c$, and $\sigma$ is a good coloring, contradicting our assumption. \end{claimproof}

\begin{claim}\label{xinUeq1BothNbrsCL} For any $x\in U^{=1}$, both neighbors of $x$ in $C^*$ lie in $\textnormal{Mid}^{\downarrow}(C^*)\cap\Delta_L$. \end{claim}

\begin{claimproof} Suppose not. Thus, there is a subpath $x_1x_2x_3$ of $C^*$ with $x_2\in U^{=1}$ and $x_1\not\in\textnormal{Mid}^{\downarrow}(C^*)\cap\Delta_L$. By Claim \ref{xinU1nbrsCL}, we have $x_1\not\in\textnormal{Mid}^{\downarrow}(C^*)$. Consider the following cases. 

\textbf{Case 1:} $x_3\not\in\textnormal{Mid}^{\downarrow}(C^*)$

By Claim \ref{HalveMinusNonEmptyCL}, since $C^*-x_2$ is a path and an induced subgraph of $G$, there is a $\sigma\in\textnormal{Halve}^-(C^*-x_2)$. Since $x_2\in U^{=1}$, we have $|L_{\psi\cup\sigma}(x_2)|\geq 2$, so $\sigma$ extends to an $L_{\psi}$-coloring $\sigma^*$ of $\textnormal{dom}(\sigma)\cup\{x_2\}$. Since $x_1, x_3\not\in\textnormal{Mid}^{\downarrow}(C^*)$ and $U^{=1}\cap\textnormal{Mid}^{\downarrow}(C^*)=\varnothing$, we have $\textnormal{Tw}(\sigma^*)=\textnormal{Tw}(\sigma)$, so $\sigma^*$ is a good coloring, contradicting our assumption. 

\textbf{Case 2:} $x_3\in\textnormal{Mid}^{\downarrow}(C^*)$

In this case, by Claim \ref{xinU1nbrsCL}, we have $x_3\in\Delta_L$. Let $z$ be the unique neighbor of $x_3$ on the path $C^*-x_2$. Since $x_3\in\Delta_L$, we have $z\not\in\Delta_L$, so there is a $c\in L_{\psi}(z)\setminus L_{\psi}(x_3)$. Let $Q:=C^*\setminus\{x_2, x_3\}$. By Claim \ref{MainClaimforIntermPropSSF}, there is a $\tau\in\textnormal{Halve}^-(Q)$ with $\tau(z)=c$.  Note that $|L_{\psi\cup\tau}(x_2)|\geq 3$, so $\tau$ extends to an $L_{\psi}$-coloring $\tau^*$ of $\textnormal{dom}(\tau)\cup\{x_2\}$ such that $\tau^*(x_2)\not\in L_{\psi}(x_3)$. Thus, by our choice of $c$, $\{x_3\}$ is $(L, \psi\cup\tau^*)$-inert. Since $x_3$ is uncolored and $x_1, x_2\not\in\textnormal{Mid}^{\downarrow}(C^*)$, we have $\textnormal{Tw}(\tau^*)=\textnormal{Tw}(\tau)$. Since $\tau\in\textnormal{Halve}^+(Q)$ and $|\Delta_L|=|\Delta_L\cap V(Q)|-1$, it follows that $\tau^*$ is a good coloring, contradicting our assumption.  \end{claimproof}

Now we have enough to finish the proof of Proposition \ref{OneStepIntermProp}. Since $U^{=1}\neq\varnothing$, let $x_1x_2x_3$ be a subpath of $C^*$ of length two, where $x_2\in U^{=1}$. By Claim \ref{xinUeq1BothNbrsCL}, $x_1, x_3\in\textnormal{Mid}\cap\Delta_L$. Let $y$ be the unique neighbor of $x_1$ on the path $C^*-x_2$ and let $y'$ be the unique neighbor of $x_3$ on the path $C^*-x_2$. Recall that $|V(C^*)|\geq 5$, so $y, y'\not\in\{x_1, x_2, x_3\}$ and $y\neq y'$. Let $Q:=C^*\setminus\{x_1, x_2, x_3\}$. Since $y\not\in\Delta_L$, there is a $c\in$ such that $c\not\in L_{\psi}(x_1)$. By Claim \ref{MainClaimforIntermPropSSF}, since $Q$ is an induced subgraph of $G$, there is a $\phi\in\textnormal{Halve}^+(Q)$ with $\phi(y)=c$. Now, we have $|L_{\psi\cup\phi}(x_3)|\geq 1$, and since $x_2\in U^{=1}$, we have $|L_{\psi}(x_2)\setminus L_{\psi}(x_1)|\geq 2$, so $\phi$ extends to an $L_{\psi}$-coloring $\phi^*$ of $\textnormal{dom}(\phi)\cup\{x_2, x_3\}$ such that $\phi^*(x_2)\not\in L_{\psi}(x_1)$. By our choice of $c, \phi^*(x_2)$, we get that $\{x_1\}$ is $(L, \psi\cup\phi^*)$-inert in $G$. Since $x_3\in$ and $x_1$ is uncolored, we have $|\textnormal{Tw}(\phi^*)|\leq |\textnormal{Tw}(\phi)|+1$. Since $|\Delta_L\cap V(Q)|=|\Delta_L|-2$ and $\phi\in\textnormal{Halve}^+(Q)$, we have $|\textnormal{Tw}(\phi^*)|\leq\left\lceil\frac{|\Delta_L|}{2}\right\rceil$, so $\phi^*$ is a good coloring, contradicting our assumption. This completes the proof of Proposition \ref{OneStepIntermProp}. \end{proof}

\section{Completing the proof of Theorem \ref{Main4CycleAnnulusThm}}\label{ThmCutIntoCompWithThomFacesSec}

Recalling the notation introduced in Observation \ref{trianglesplitobs}, we get the following consequence of Proposition \ref{OneStepIntermProp}.

\begin{lemma}\label{RepUseCorMainSSF1} Let $G$ be a 2-connected planar graph with outer cycle $C$, where $|V(C)|>3$. Let $L$ be a list-assignment for $V(G)$ and let $\psi$ be an $L$-coloring of $V(C)$.  Let $r:=\lceil\log_2(|V(C)|)\rceil+2$ and suppose further that $\textit{NDepth}(C)\geq r$ and every vertex of $B_{r}(C)\setminus V(C)$ has an $L$-list of size at least five. Then there is an extension of $\psi$ to an $L$-coloring $\psi^*$ of $w^{r-1}(C)$ such that $V(w^{r-1}(C))$ is $(L, \psi^*)$-inert in $G$ and the outer face of $G\setminus V(w^{r-1}(C))$ is a Thomassen facial subgraph of $G\setminus V(w^{r-1}(C))$ with respect to $L_{\psi^*}$. \end{lemma}

\begin{proof} For each $k=0, 1,\cdots, r-1$, let $C^*_k:=F^k(C)$. Note that, for each $k=0, \cdots, r-1$, we have $\textit{NDepth}(C^*_k)\geq 2$, since $\textit{NDepth}(C)\geq r+2$. We trivially get that $V(C)\setminus\textnormal{dom}(\psi)$ is $(L, \psi)$-inert in $G$, since $\textnormal{dom}(\psi)=V(C)$. Thus, applying Proposition \ref{OneStepIntermProp}, we get that, for each $k=0, 1,\cdots, r-1$ there is a partial $L_{\psi}$-coloring $\psi_k$ of $V(w^k(C))$ such that $\psi_0=\psi$ and, for  $k>0$, the following hold.
\begin{enumerate}[label=\arabic*)]
\itemsep-0.1em
\item $\psi_k$ restricted to $\textnormal{dom}(\psi_k)\cap V(w^{k-1}(C))$ is $\psi_{k-1}$\emph{; AND}
\item $V(w^k)$ is $(L, \psi_k)$-inert in $G$\emph{; AND}
\item $|\Delta_L(\psi_k, C^*_k)|\leq\left\lceil\frac{|\Delta_L(\psi_{k-1}, C^*_{k-1})|}{2}\right\rceil$.
\end{enumerate}

As $r=\lceil\log_2(|V(C)|)\rceil+2$, we have  $|\Delta_L(\psi_{r-1}, C^*_{r-1})|\leq 1$, and since $w^{r-1}(C)$ is 0-triangulated, the outer face of $G\setminus w^{r-1}(C)$ is a Thomassen facial subgraph of $G\setminus w^{r-1}(C)$ with respect to $L_{\psi_{r-1}}$, so we are done. \end{proof}

To prove Theorem \ref{Main4CycleAnnulusThm}, we need one more lemma. We first introduce the following definition. 

\begin{defn}\emph{Let $k\geq 0$ be an integer. A 4-tuple $\mathcal{L}=(G, C, L, \psi)$ is called a $k$-\emph{lens} if $G$ is a 2-connected, short-inseparable planar graph with cyclic outer face $C$, $L$ is a list-assignment for $V(G)$, and the following hold.}
\begin{enumerate}[label=\emph{\arabic*)}]
\itemsep-0.1em
\item\emph{$C$ is $k$-triangulated and $|L(v)|\geq 5$ for all $v\in B_{k+1}(C)\setminus V(C)$}; AND
\item\emph{$\psi$ is a partial $L$-coloring of $V(C)$ and $V(C)\setminus\textnormal{dom}(\psi)$ is $L_{\psi}$-inert in $G$}; AND
\end{enumerate}
\emph{We call $\mathcal{L}$ a \emph{lens} if there exists a $k\geq 0$ such that $\mathcal{L}$ is a $k$-lens.}
\end{defn}

\begin{lemma}\label{MainThmSSFSec} Let $G$ be a 2-connected, short-inseparable planar graph with outer cycle $C$ of length $n$ for some integer $n\geq 3$. Let $r:=\lceil\log_2(n)\rceil+2$, let $L$ be a list-assignment for $V(G)$ and let $\tau$ be an $L$-coloring of $V(C)$ such that $(G, C, L, \tau)$ is an $rn$-lens and $\tau$ extends to an $L$-coloring of $B_{r(n-1)}(C)$. Then there exists a 2-connected subgraph $K^*$ of $G$ with $C\subseteq K^*$ and an extension of $\tau$ to a partial $L$-coloring $\psi$ of $V(K^*)$ such that the following hold. 
\begin{enumerate}[label=\arabic*)]
\itemsep-0.1em
\item $V(K^*)$ is $(L, \psi)$-inert in $G$ and $V(K^*)\subseteq B_{rn}(C)$; AND
\item For each connected component $H$ of $G\setminus K^*$, the outer face of $H$ is a Thomassen facial subgraph of $H$ with respect to $L_{\psi}$. 
\end{enumerate}\end{lemma}

\begin{proof} Let $k:=r\cdot (n-3)$. Since $\mathcal{L}$ is an $rn$-lens, we get that $C$ is $(k+r)$-triangulated. By Proposition \ref{k+rboundpart}, there is a $(k,r)$-skeleton $K$ of $C$, and, by definition, $V(K)\subseteq B_k(C)$. Since $K$ is a $(k,r)$-skeleton of $C$, we get by definition that, for each inward-facial subgraph $D$ of $K$, either $D$ has nonsplit depth at least $r$ or $D$ is an induced subgraph of $\textnormal{Int}(D)$ with $\Delta^{\geq 3}(D)=\varnothing$. We now define a 2-connected subgraph $K^*$ of $G$, where $K\subseteq K^*$, in the following way. For each inward facing facial subgraph $D$ of $K^*$, we add some edges and vertices to the closed disc bounded by $D$. If either $|V(D)|=3$ or $\textit{NDepth}(D)<r$, then we add nothing to the closed disc bounded by $D$. If $|V(D)|>3$ and $\textit{NDepth}(D)\geq r$, then we add to the closed disc bounded by $D$ the graph $w^r(D)$, where $w^r(D)$ is as in Observation \ref{trianglesplitobs}. We let $K^*$ be the graph obtained from $K$ by doing this for each inward-facing facial subgraph of $K$. We claim now that $K^*$ satisfies both of 1) and 2). Note that $V(K^*)\subseteq B_{k+r}(C)$, and $k+r\leq rn$, so $V(K^*)\subseteq B_{rn}(C)$. Let $\mathcal{D}$ be the set of inward-facing facial subgraphs $D$ of $K$ with $\textit{NDepth}(D)\geq r$ and $|V(D)|>3$. For each $D\in\mathcal{D}$, we now define an extension of $\sigma$ to an $L$-coloring $\sigma_D$ with $\textnormal{dom}(\sigma_D)\setminus\textnormal{dom}(\sigma)\subseteq\textnormal{Int}(D)\setminus V(D)$. Since $\mathcal{L}$ is an $rn$-lens and $V(D)$ is precolored by $\sigma$, it follows from Lemma \ref{RepUseCorMainSSF1} that $\sigma$ extends to a partial $L$-coloring $\sigma_D$ of $\textnormal{dom}(\sigma)\cup V(w^{r-1}(D))$ such that $V(w^{r-1}(D))\setminus\textnormal{dom}(\sigma_D)$ is $(L, \sigma_D)$-inert in $\textnormal{Int}(D)$ and the outer face of $\textnormal{Int}(D)\setminus V(w^{r-1}(D))$ is a Thomassen facial subgraph of $\textnormal{Int}(D)\setminus V(w^{r-1}(D))$ with respect to $L_{\sigma_D}$. 

Note that, for any two distinct  inward-facing facial sugraphs $D, D'$ of $K$, the open discs bounded by $D, D'$ respectively are disjoint. Thus, $\sigma$ extends to a partial $L$-coloring $\sigma^*$ of $V(K^*)$ such that, for each $D\in\mathcal{D}$, $\sigma^*$ restricts to $\sigma_D$. In particular, $V(K^*)$ is $(K, \sigma^*)$-inert in $G$. For any inward-facing facial subgraph $D$ of $K^*$ with $D\not\in\mathcal{D}$, we have either $|V(D)|=3$ or $D$ is an induced subgraph of $\textnormal{Int}(D)$ with $\Delta^{\geq 3}(D)=\varnothing$. In the former case, we have $\textnormal{Int}(D)\setminus V(D)=\varnothing$, and in the latter case, we trivially get that the outer face of $\textnormal{Int}(D)\setminus V(D)$ is a Thomassen facial subgraph with respect, so we are done. \end{proof}

With Lemma \ref{MainThmSSFSec} in hand, we can now prove Theorem \ref{Main4CycleAnnulusThm}, which we restate below. 

\begin{thmn}[\ref{Main4CycleAnnulusThm}] Let $G$ be a 2-connected, short-inseparable planar graph and let $F, F'$ be two facial subgraphs of $G$, each of which is a cycle of length at most four, where $F$ is the outer face of $G$. Let $P$ be an $(F, F')$-path, and let $n:=2|E(P)|+8$ and $r:=\lceil\log_2(n)\rceil+2$. Let $L$ be a list-assignment for $V(G)$ and $\phi$ be an $L$-coloring of $V(F\cup F'\cup P)$ such that
\begin{enumerate}[label=\arabic*)]
\itemsep-0.1em
\item $\phi$ extends to $L$-color $B_{r(n-1)}(F\cup F'\cup P)$; AND
\item For each $v\in B_{rn+1}(F\cup F'\cup P)$, every facial subgraph of $G$ containing $v$, except possibly $F, F'$, is a triangle, and, if $v\not\in V(F\cup F')$, then $|L(v)|\geq 5$.
\end{enumerate}
Then there is a 2-edge-connected subgraph $K$ of $G$ with $F\cup F'\cup P\subseteq K$ and an extension of $\phi$ to a partial $L$-coloring $\psi$ of $V(K)$ such that
\begin{enumerate}[label=\arabic*)]
\itemsep-0.1em
\item $V(K)$ is $(L, \psi)$-inert in $G$ and $V(K)\subseteq B_{rn}(F\cup F'\cup P)$; AND
\item For each connected component $H$ of $G\setminus K$, the outer face of $H$ is a Thomassen facial subgraph of $H$ with respect to $L_{\phi}$.
\end{enumerate}
 \end{thmn}

\begin{proof} The idea is to create an auxiliary graph in which the precolored subgraph $F\cup F'\cup P$ has been replaced by a precolored cycle of length $|E(F)|+|E(F')|+2|E(P)|$ so that we can apply Lemma \ref{MainThmSSFSec}. Let $P:=v_1\cdots v_k$, where $v_1\in V(F)$ and $v_k\in V(F')$. Now, there is an open neighborhood $U$ of $\mathbb{R}^2$ such that $U\cap V(G)=V(P)$. By splitting each vertex of $P$ within $U$, we obtain an auxiliary planar embedding $G^{\dagger}$, where each $v_i$ has be replaced by two vertices $u_i, u_i^*$, and:
\begin{enumerate}[label=\arabic*)]
\itemsep-0.1em
\item For each $i\in\{1, \cdots, k\}$, $u_iu_i^*\not\in E(G^{\dagger})$ and each vertex of $G$ of distance one from $F\cup F'\cup P$ is adjacent to precisely one of $u_i, u_i^*$ in $G^{\dagger}$; AND
\item $G^{\dagger}$ has outer cycle $(F-v_1)+u_1u_2\cdots u_k+(F'-v_k)+(u_k^*u_{k-1}^*\cdots u_1^*)$; AND
\item $G$ is recovered from $G^{\dagger}$ by identifying $u_i, u_i^*$ for each $i=1, \cdots, k$, then deleting the resulting duplicate edges.
\end{enumerate}
Now, let $L^{\dagger}$ be a list-assignment for $V(G^{\dagger})$ where $L^{\dagger}(x)=L(x)$ for each $x\in V(G^{\dagger})\setminus\{u_1, \cdots, u_k, u_1^*, \cdots, u_k^*\}$, and $L^{\dagger}(u_i)=L^{\dagger}(u_i^*)=L(v_i)$ for each $i=1,\cdots, k$. Let $C^{\dagger}$ denote the outer cycle of $G^{\dagger}$. Note that, since $G$ is 2-connected and short-inseparable, $G^{\dagger}$ is as well. Furthermore, since $u_iu_i^*\not\in E(G^{\dagger})$ for each $1\leq i\leq k$, there is an $L^{\dagger}$-coloring $\tau$ of $V(C^{\dagger})$ obtained from
$\phi$, where $\tau(u_i)=\tau(u_i^*)=\phi(v_i)$ for each $1\leq i\leq k$ and $\tau(x)=\phi(x)$ for each $x\in V(D\cup D')\setminus\{v_1, v_k\}$. Now we obtain Theorem \ref{Main4CycleAnnulusThm} by first applying Lemma \ref{MainThmSSFSec} to the lens $(G^{\dagger}, C^{\dagger}, L^{\dagger}, \tau)$ and then recovering $G$ by identifying $u_i, u_i^*$ for each $1\leq i\leq k$ and deleting the resulting parallel edges. \end{proof}

\section{An Edge-Maximality Lemma}\label{SimpleExemaxLemmaFormodblockRem}

With Theorem \ref{Main4CycleAnnulusThm} in hand, we now prove Theorem \ref{5ListHighRepFacesFarMainRes} over Sections \ref{SimpleExemaxLemmaFormodblockRem}-\ref{thisisidicritCOMPcomcriTT}. We first need the lemma below. The reason we need Lemma \ref{TriangulationCorMainLmemmaused1} is that our proof of Theorem \ref{5ListHighRepFacesFarMainRes} relies on a result from \cite{JNevinPapIISeq} as a black box, which is stated below in Section \ref{RecallResFromPaperIISeq} as Theorem \ref{PaIIBlackBoxTessMain}, but this result only applies when certain triangulation conditions are satisfied. 

\begin{lemma}\label{TriangulationCorMainLmemmaused1}
Let $\alpha\geq 1$ be an integer, $\Sigma$ be a surface, and $G$ be a connected embedding on $\Sigma$, and let $\mathcal{C}=\{C_1,\cdots, C_m\}$ be a collection of facial subgraphs of $G$ such that $d_G(C_i, C_j)\geq\alpha$ for each $1\leq i<j\leq m$. There exists an embedding $G'$ on $\Sigma$, where $G'$ is obtained from $G$ by adding edges to $G$, such that the following hold.
\begin{enumerate}[label=\arabic*)]
\itemsep-0.1em
\item $\textnormal{fw}(G')=\textnormal{fw}(G)$ and  $d_{G'}(C_i, C_j)\geq\alpha$ for each $1\leq i<j\leq m$; AND
\item Each element of $\mathcal{C}$ is also a a facial subgraph of $G'$; AND
\item For every facial subgraph $H$ of $G'$ with $H\not\in\mathcal{C}$, every induced cycle in $G'[V(H)]$ is a triangle.
\end{enumerate}
 \end{lemma}

\begin{proof} Suppose first that, for every facial subgraph $H$ of $G$ with $H\not\in\mathcal{C}$, every induced cycle in $G[V(H)]$ is a triangle. If that holds, then we take $G'=G$ and we are done. Now suppose there exists a facial subgraph $H$ of $G$, with $H\not\in\mathcal{C}$ and a chordless cycle $H'\subseteq G[V(H)]$ such that $H'$ is not a triangle. Let $H'=v_1\cdots v_kv_1$ for some $k\geq 3$. By definition, there is  a component $U$ of $\Sigma\setminus G$ with $H=\partial(H)$, so $V(H')\subseteq\partial(U)$. Now, for any indices $i,j\in\{1, \cdots, k\}$ with $|i-j|>1$, there is an embedding $G^{\dagger}$ obtained from $G$ by adding $v_iv_j$ to $\textnormal{Cl}(U)$, where at least one of the components of $\Sigma\setminus G'$ contained in $U$ is a disc. Thus, $\textnormal{fw}(G^{\dagger})=\textnormal{fw}(G)$, so, to prove Lemma \ref{TriangulationCorMainLmemmaused1}, it suffices to show that there exists an index $j\in\{1,\cdots, k\}$ such that, setting $G^{\dagger}:=G+v_jv_{j+2}$, we have $d_{G^{\dagger}}(C_s, C_t)\geq\alpha$ for any distinct indices $s,t\in\{1,\cdots, m\}$. If we show that the above holds, then we simply iterate until, after a finite number of steps, we obtain a embedding $G'$ which satisfies properties 1)-3) above. Suppose toward a contradiction that there does not exist an index $j\in\{1, \cdots, k\}$ satisfying the above property. For the remainder of the proof of Lemma \ref{TriangulationCorMainLmemmaused1}, a distance between two vertices of $V(G)$ without a subscript  denotes a distance between these two vertices in the initial graph $G$. Note that, for any distinct indices $s,t\in\{1,\cdots, m\}$ and any $j\in\{1,\cdots, k\}$, since $d(C_s, C_t)\geq\alpha$, we have $d(v_j, C_s)+d(v_{j+2}, C_t)\geq\alpha-2$. For each $j\in\{1,\cdots, k\}$, let $B_j$ be the set of pairs $(s,t)\in\{1, \cdots, m\}\times\{1, \cdots, m\}$ such that $d(C_s, v_j)+d(C_t, v_{j+2})=\alpha-2$. If there exists a $j\in\{1, \cdots, k\}$ such that $B_j=\varnothing$, then, setting $G^{\dagger}:=G+v_jv_{j+2}$, we have $d_{G^{\dagger}}(C_s, C_t)\geq\alpha$ for all $1\leq s<t\leq m$, contradicting our assumption. Thus, we have $B_j\neq\varnothing$ for each $j\in\{1, \cdots, k\}$. Let $B:=\bigcup_{j=1}^kB_j$.

\begin{claim}\label{onestepsubclaim}

Let $j\in\{1,\cdots,k\}$ and let $(s,t)\in B_j$. Note that the following distance conditions hold:

\begin{enumerate}[label=\arabic*)]
\itemsep-0.1em 
\item $d(C_t, v_j)=d(C_t, v_{j+2})+2$; AND
\item $d(C_s, v_{j+2})=d(C_s, v_j)+2$; AND
\item $d(C_t, v_{j+1})=d(C_t, v_{j+2})+1$; AND
\item $d(C_s, v_{j+1})=d(C_s, v_j)+1$.
\end{enumerate}

\end{claim}

It immediately follows from Claim \ref{onestepsubclaim} that $d(C_s, v_r)+d(C_t, v_r)=\alpha$ for each $r\in\{j, j+1, j+2\}$. 

\begin{claim}\label{onlyoneclose} Let $j\in\{1,\cdots,k\}$, let $(s,t)\in B_j$, and suppose that $d(C_s, v_j)\leq\frac{\alpha}{2}-1$. Then the following hold.
\begin{enumerate}[label=\arabic*)]
\itemsep-0.1em
\item For every pair $(p,q)\in B_{j-1}$, either $p=s$ or $q=s$; AND
\item For every $(p,q)\in B_{j+1}$, either $p=s$ or $q=s$. 
\end{enumerate}
\end{claim}

\begin{claimproof} Since $B_{j-1}\neq\varnothing$, there is a pair $(p,q)\in B_{j-1}$. Suppose that $p\neq s$. Now, by Claim \ref{onestepsubclaim}, we have $d(C_p, v_j)=d(C_p, v_{j-1})+1$. Since $d(C_s, v_j)\leq\frac{\alpha}{2}-1$ and $s\neq q$, we have $d(C_p, v_j)\geq\frac{\alpha}{2}+1$, or else $d(C_p, C_s)<\alpha$. Thus, we have $d(C_p, v_{j-1})\geq\frac{\alpha}{2}$. Since $d(C_p, v_{j-1})+d(C_q, v_{j+1})=\alpha-2$, we have $d(C_q, v_{j+1})\leq\frac{\alpha}{2}-2$. Thus we have $d(C_q, v_{j})\leq\frac{\alpha}{2}-1$, so $q=s$, or else we have distinct cycles $C_s, C_q$ such that $d(C_s, C_q)\leq\alpha-2$, violating our distance conditions. Now let $(p,q)\in B_{j+1}$ and suppose that $p\neq s$. As above, since $d(C_s, v_j)\leq\frac{\alpha}{2}-1$ and $p\neq s$, we have $d(C_p, v_j)\geq\frac{\alpha}{2}+1$, or else $d(C_p, C_s)<\alpha$. Thus, we have $d(C_p, v_{j+1})\geq\frac{\alpha}{2}$. Since $d(C_p, v_{j+1})+d(C_q, v_{j+3})=\alpha-2$, we have $d(C_q, v_{j+2})\leq\frac{\alpha}{2}-2$. Now, since $d(C_s, v_j)\leq\frac{\alpha}{2}-1$, we have $d(C_s, v_{j+2})\leq\frac{\alpha}{2}+1$. Thus, we have $q=s$, or else there are distinct cycles $C_s, C_q$ such that $d(C_s, C_q)\leq\alpha-1$, contradicting our distance conditions. \end{claimproof}

Now we choose an index $j^{\star}\in\{1,\cdots, k\}$ and a pair $(s^{\star}, t^{\star})\in B$ such that the quantity $\min\{d(C_{s^{\star}}, v_{j^{\star}}), d(C_{t^{\star}}, v_{j^{\star}+2})\}$ is minimized. Consider the following cases:

\textbf{Case 1:}  $\min\{d(C_{s^{\star}}, v_{j^{\star}}), d(C_{t^{\star}}, v_{j^{\star}+2})\}=d(C_{t^{\star}}, v_{j^{\star}+2})$.

In this case, since $d(C_{s^{\star}}, v_{j^{\star}})+d(C_{t^{\star}}, v_{j^{\star}+2})=\alpha-2$, we have $d(C_{s^{\star}}, v_{j^{\star}})\leq\frac{\alpha}{2}-1$. 

\begin{claim} $d(C_{s^{\star}}, v_{j^{\star}-1})=d(C_{s^{\star}}, v_{j^{\star}})$. \end{claim}

\begin{claimproof} Suppose not. Then we have $d(C_{s^{\star}}, v_{j^{\star}-1})=d(C_{s^{\star}}, v_j)\pm 1$. If $d(C_{s^{\star}}, v_{j^{\star}-1})=d(C_{s^{\star}}, v_{j^{\star}})-1$, then  applying Claim \ref{onlyoneclose}, $B_{j^{\star}-1}$ either contains a pair of the form $(s^{\star}, q)$, or a pair of the form $(p, s^{\star})$. In either case, we contradict the minimality of $(s^{\star}, t^{\star})$. Thus, we have $d(C_{s^{\star}}, v_{j^{\star}-1})=d(C_{s^{\star}}, v_{j^{\star}})+1$. Applying Claim \ref{onestepsubclaim}, we have $d(C_{s^{\star}}, v_{j^{\star}-1})=d(C_{s^{\star}}, v_{j^{\star}+1})=d_G(C_{s^{\star}}, v_{j^{\star}})+1$. By Claim \ref{onlyoneclose}, $B_{j^{\star}-1}$ either contains a pair of the form $(s,q)$ or a pair of the form $(q,s)$. If $B_{j^{\star}-1}$ contains a pair of the form $(s^{\star} ,q)$, then, by Claim \ref{onestepsubclaim}, we have $d(C_{s^{\star}}, v_{j^{\star}+1})=d(C_s, v_{j^{\star}-1})$, contradicting the fact that $d(C_{s^{\star}}, v_{j^{\star}-1})=d(C_{s^{\star}}, v_{j+1})=d(C_{s^{\star}}, v_{j^{\star}})+1$. Thus, $B_{j^{\star}-1}$ contains a pair of the form $(q, s^{\star})$ for some $q\in\{1,\cdots, m\}$. Thus, by Claim \ref{onestepsubclaim}, we have $d(C_{s^{\star}}, v_{j^{\star}-1})=d(C_{s^{\star}}, v_{j^{\star}+1})+2$. But we also have $d(C_{s^{\star}}, v_{j^{\star}+1})=d(C_{s^{\star}}, v_{j^{\star}})+1$, applying Claim \ref{onestepsubclaim} to the pair $(s^{\star}, t^{\star})$, so $d(C_{s^{\star}}, v_{j^{\star}-1})=d(C_{s^{\star}}, v_{j^{\star}})+3$, which is false since $v_{j^{\star}}, v_{j^{\star}-1}$ are adjacent. \end{claimproof} 

Since $B_{j^{\star}-1}\neq\varnothing$,  there exists a $(p,q)\in B_{j^{\star}-1}$. Since $d(C_{s^{\star}}, v_{j^{\star}})\leq\frac{\alpha}{2}-1$, Claim \ref{onlyoneclose} implies that either $p=s^{\star}$ or $q=s^{\star}$. If $p=s^{\star}$ then we have $d(C_{s^{\star}}, v_{j^{\star}})=d(C_{s^{\star}}, v_{j^{\star}-1})+1$, contradicting the fact that $d(C_{s^{\star}}, v_{j^{\star}})=d(C_{s^{\star}}, v_{j^{\star}-1})$. Thus, $q=s^{\star}$, and it follows from Claim \ref{onestepsubclaim} applied to $(p, s^{\star})$ that $d(v_{j^{\star}}, C_{s^{\star}})=d(v_{j^{\star}+2}, C_{s^{\star}})+2$. Yet, by Claim \ref{onestepsubclaim}, we also have $d(v_{j^{\star}+2}, C_{s^{\star}})=d(v_{j^{\star}}, C_{s^{\star}})+2$, a contradiction.

\textbf{Case 2:}  $\min\{d(C_{s^{\star}}, v_{j^{\star}}), d(C_{t^{\star}}, v_{j^{\star}+2})\}=d(C_{t^{\star}}, v_{j^{\star}+2})$.

In this case, we simply reverse the orientation and apply the same argument as above. For each $j\in\{1,\cdots, k\}$, we set $\hat{B}_j$ to be the set of pairs $(s,t)\in\{1, \cdots, m\}\times\{1, \cdots, m\}$ such that $d(C_s, v_j)+d(C_s, v_{j-2})=\alpha-2$. Then we are back to Case 1 with $B_1, \cdots, B_k$ replaced by $\hat{B}_1, \cdots, \hat{B}_k$. This completes the proof of Lemma \ref{TriangulationCorMainLmemmaused1}. \end{proof}

\section{Charts and Tilings}\label{RecallResFromPaperIISeq}

Over the course of the proof of Theorem \ref{5ListHighRepFacesFarMainRes} below, we deal with the following structures. 

\begin{defn}\label{MainChartDefn}\emph{Let $k, \alpha\geq 1$ be integers. A tuple $\mathcal{T}=(\Sigma, G, \mathcal{C}, L, C_*)$ is called an $(\alpha, k)$-\emph{chart} if $\Sigma$ is a surface, $G$ is an embedding on $\Sigma$ with list-assignment $L$, and $\mathcal{C}$ is a family of facial subgraphs of $G$ such that}

\begin{enumerate}[label=\emph{\arabic*)}]
\itemsep-0.1em 
\item \emph{$C_*\in\mathcal{C}$ and, for any distinct $H_1, H_2\in\mathcal{C}$, we have $d(H_1, H_2)\geq\alpha$}; AND
\item \emph{$|L(v)|\geq 5$ for all $v\in V(G)\setminus \left(\bigcup_{H\in\mathcal{C}}V(H)\right)$}; AND
\item \emph{For each $H\in\mathcal{C}$, there exists a connected subgraph $\mathbf{P}_{\mathcal{T},H}$ of $H$ satisfying the following.}
\begin{enumerate}[label=\emph{\roman*)}]
\itemsep-0.1em 
\item \emph{$|E(\mathbf{P}_{\mathcal{T}, H})|\leq k$ and $V(\mathbf{P}_{\mathcal{T},H})$ is $L$-colorable}; AND
\item \emph{$|L(v)|\geq 3$ for all $v\in V(H)\setminus V(\mathbf{P}_{\mathcal{T},H})$.}
\end{enumerate}
\end{enumerate}

\end{defn}

The definition above entails that $\mathcal{C}\neq\varnothing$, as $C_*$ is an element of $\mathcal{C}$, but, given an $H\in\mathcal{C}$, we possibly have $\mathbf{P}_{\mathcal{T}, H}=\varnothing$. If the underlying chart is clear from the context, we usually drop the $\mathcal{T}$ from the notation $\mathbf{P}_{\mathcal{T}, H}$ to avoid clutter.

\begin{defn}\label{ChartMoreTerms}
\emph{Given a surface $\Sigma$, we define the following.}
\begin{enumerate}[label=\emph{\arabic*)}]
\itemsep-0.1em
\item\emph{A tuple $\mathcal{T}$ is called a \emph{chart} if there exist integers $k,\alpha\geq 1$ such that $\mathcal{T}$ is an $(\alpha, k)$-chart.}
\item\emph{A chart $\mathcal{T}=(\Sigma, G, \mathcal{C}, L, C_*)$ is called \emph{colorable} if $G$ is $L$-colorable. We call $G$ the \emph{underlying graph} of the chart and we call $\Sigma$ the \emph{underlying surface} of the chart.}
\item \emph{Given a chart $\mathcal{T}$, we say that $\mathcal{T}$ is \emph{short-inseparable} if the underlying graph of $\mathcal{T}$ is short-inseparable. We say that $\mathcal{T}$ is \emph{chord-triangulated} if, for every facial subgraph $H$ of $G$ with $H\not\in\mathcal{C}$, every induced cycle of $G[V(H)]$ is a triangle. We call $\mathcal{T}$ a \emph{tessellation} if it is short-inseparable and chord-triangulated. Given integers $k, \alpha\geq 1$, we call $\mathcal{T}$ an $(\alpha, k)$-\emph{tessellation} if it is both a tessellation and an $(\alpha, k)$-chart.}
\end{enumerate}
\end{defn}

We now state two black boxes we need in order to prove Theorem \ref{5ListHighRepFacesFarMainRes}. The theorem below is the key black box we need from \cite{JNevinPapIISeq}, which shows that Theorem \ref{5ListHighRepFacesFarMainRes} holds for charts which are both chord-triangulated and short-inseparable. Actually, the main result of \cite{JNevinPapIISeq} is stronger than Theorem \ref{PaIIBlackBoxTessMain},  but the result below is enough for our purposes. 

\begin{theorem}\label{PaIIBlackBoxTessMain} Let $\beta$ be a constant with $\beta\geq 4\cdot 10^6$ and let $\Sigma$ be a surface with $g:=g(\Sigma)$. Let $(\Sigma, G, \mathcal{C}, L, C_*)$ be a $(2.1\beta\cdot 6^{g}, 4)$-tessellation, where $\textnormal{fw}(G)\geq 2.1\beta\cdot 6^{g}$ and furthermore, for each $H\in C$, either $\mathbf{P}_H$ is either a path of length or $H$ is a cycle of length at most four with $\mathbf{P}_H=H$. Then $G$ is $L$-colorable. 
 \end{theorem}

We also need the following result from \cite{HyperbolicFamilyColorSurfPap}. 

\begin{theorem}\label{cylindertheorem}
There is a constant $\gamma\geq 0$ such that the following holds: Let $G$ be a planar embedding and let $C_1, C_2$ be facial cycles of $G$, each of length at most four, where $d(C_1, C_2)\geq\gamma$. Let $L$ be a list-assignment for $V(G)$, where $V(C_1\cup C_2)$ is $L$-colorable and each vertex of $G\setminus (C_1\cup C_2)$ has a list of size at least five. Then $G$ is $L$-colorable. \end{theorem}

To prove Theorem \ref{5ListHighRepFacesFarMainRes}, we first define a slightly modified version of face-width which is easier to work with for our purposes. The issue we encounter with imposing face-width bounds on a minimal counterexample $G$ to Theorem \ref{5ListHighRepFacesFarMainRes} is that we need to exploit the presence of separating cycles of length at most four in $G$, but given a contractible 4-cycle $D\subseteq G$ and a component $U$ of $\Sigma\setminus D$, the embedding $G\cap\textnormal{Cl}(U)$ possibly has face-width $\textnormal{fw}(G)-1$. We get around this by defining the following. 

\begin{defn}\label{ModifiedFwStarDefn} \emph{Let $G$ be an embedding on a surface $\Sigma$. We define a $G$-\emph{covering triple} to be a triple $(N, \mathcal{A}, \mathcal{B})$, where $N\subseteq\Sigma$ is a noncontractible cycle and $\mathcal{A}, \mathcal{B}$ are collections of subgraphs of $G$, and furthermore, there is a family $\{U_F: F\in\mathcal{A}\cup\mathcal{B}\}$ of open subsets of $\Sigma$ such that:}
\begin{enumerate}[label=\alph*)]
\itemsep-0.1em
\item $N\subseteq\bigcup_{F\in\mathcal{A}\cup\mathcal{B}}\textnormal{Cl}(U_F)$; AND
\item For each $F\in\mathcal{A}$, $F$ is a facial subgraph of $G$ and $U_F$ is a component of $\Sigma\setminus G$ with $\partial(U_F)=F$; AND
\item For each $F\in\mathcal{B}$, $F$ is a contractible cycle of length at most four, and $U_F$ is both an open disc and a component of $\Sigma\setminus F$ .
\end{enumerate}
\emph{We define $\textnormal{fw}^*(G)$ to be the minimum of $|\mathcal{A}\cup\mathcal{B}|$ taken over all $G$-covering pairs. If there are no noncontractible closed curves, i.e $\Sigma=\mathbb{S}^2$, then we set $\textnormal{fw}^*(G)=\infty$. }\end{defn}

The usefulness of the definition above for our purposes lies in the following proposition. 

\begin{prop}\label{FWstarToFWObs} Let $G$ be an embedding on a surface $\Sigma$. Then the following hold:
\begin{enumerate}[label=\Alph*)]
\itemsep-0.1em
\item $\textnormal{fw}(G)\geq\textnormal{fw}^*(G)\geq\textnormal{fw}(G)/5$; AND 
\item Let $D\subseteq G$ be a contractible cycle of length at most four, $U$ be a components of $\Sigma\setminus D$, and $G':=G\cap\textnormal{Cl}(U)$. Then $\textnormal{fw}^*(G')\geq\textnormal{fw}^*(G)$. 
\end{enumerate}
 \end{prop}

\begin{proof} B) is immediate so we just prove A). If $\Sigma=\mathbb{S}^2$, then we are done, so suppose $\Sigma\neq\mathbb{S}^2$. 

\begin{claim}\label{ForCovTripleWeightBound} For any covering triple $(N, \mathcal{A}, \mathcal{B})$, we have $|\mathcal{A}|+4|\mathcal{B}|\geq\textnormal{fw}(G)$. \end{claim}

\begin{claimproof} Let $m$ be the minimum value of  $4|\mathcal{A}|+|\mathcal{B}|$ taken over all covering triples $(N, \mathcal{A}, \mathcal{B})$, and let $\textbf{Cov}_m$ be the set of covering triples $(N, \mathcal{A}, \mathcal{B})$ such that $|\mathcal{A}|+4|\mathcal{B}|=m$. By definition, $\textbf{Cov}_m\neq\varnothing$. It suffices to prove that $m\geq\textnormal{fw}(G)$. Suppose not. Among the elements of $\textbf{Cov}_m$, we choose $(N, \mathcal{A}, \mathcal{B})$ so that $|\mathcal{B}|$ is minimized. Now, if $\mathcal{B}=\varnothing$, then $|\mathcal{A}|=m$, so $m\geq\textnormal{fw}(G)$. contradicting our assumption. Thus, there is a $T\in\mathcal{B}$. Since $\Sigma\neq\mathbb{S}^2$, there is a unique component $U$ of $\Sigma\setminus T$ which is an open disc. Let $U'$ denote the other component of $\Sigma\setminus T$ and let $G':=G\cap\textnormal{Cl}(U')$. Let $\mathcal{A}':=\{F\in\mathcal{A}:\textnormal{Cl}(U_F)\subseteq G'\}$ and $\mathcal{B}':=\{F\in\mathcal{B}: \textnormal{Cl}(U_F)\subseteq G'\}$. Note that $T\not\in\mathcal{B}'$, since $U'$ is not a disc. Since $U$ is a disc, there is a noncontractible cycle $N'\subseteq\textnormal{Cl}(U')$ such that $N'\subseteq\partial(T)\cup\bigcup_{F\in\mathcal{A}'\cup\mathcal{B}'}\textnormal{Cl}(U_F)$, so $(N', \mathcal{A}', \mathcal{B}'\cup\{T\})$ is also a covering triple. On the other hand, there is a family $\mathcal{D}$ of at most four facial subgraphs of $G$ containing $T$ and a family $\{U_D: D\in\mathcal{D}\}$ of components of $\Sigma\setminus G$, where $\partial(U_D)=D$ for each $D\in\mathcal{D}$, such that $|\mathcal{D}|\leq 4$ and $\partial(T)\subseteq\bigcup_{D\in\mathcal{D}}D$. Thus, $(N', \mathcal{A}'\cup\mathcal{D}, \mathcal{B}')$ is also a covering triple. By the minimality of $m$, we have $m=(|\mathcal{A}|+4)+4(|\mathcal{B}|-1)\geq|\mathcal{A}'\cup\mathcal{D}|+4|\mathcal{B}'|\geq m$. Since $|\mathcal{B}'|<|\mathcal{B}|$, this contradicts the minimality of $|\mathcal{B}|$. \end{claimproof}

Now, it is immediate that $\textnormal{fw}(G)\geq\textnormal{fw}^*(G)$. By definition, there is a covering triple $(N, \mathcal{A}, \mathcal{B})$ with $|\mathcal{A}\cup\mathcal{B}|=\textnormal{fw}^*(G)$, so, by Claim \ref{ForCovTripleWeightBound}, we have $\textnormal{fw}^*(G)\geq\max\{|\mathcal{A}|, |\mathcal{B}|\}\geq (|\mathcal{A}+4|\mathcal{B}|)/5\geq\textnormal{fw}^*(G)/5$.  \end{proof}

To prove Theorem \ref{5ListHighRepFacesFarMainRes}, we prove that it holds in the chord-triangulated case. We note below in Proposition \ref{SuffAllTilColProp} that this is actually sufficient. Note that, in the setting below, we have $\alpha=2^{\Omega(g)}$, where the coefficient of $g$ in the exponent is a constant depending on $\gamma$. 

\begin{defn}\label{TilingDefnM} \emph{A \emph{tiling} is a chord-triangulated chart $(\Sigma, G, \mathcal{C}, L, C_*)$ such that, letting $g:=g(\Sigma)$ and $\beta, \gamma$ be constants, where $\gamma$ is as in Theorem \ref{cylindertheorem} and $\beta\geq\max\{4\cdot 10^6, \gamma\}$, and defining $\delta=\delta(g):=2.1\beta\cdot 6^{g}$ and $\alpha=\alpha(g):=7\delta\log_2(\delta)+3\gamma$, the following hold:}
\begin{enumerate}[label=\arabic*)]
\itemsep-0.1em
\item\emph{$(\Sigma, G, \mathcal{C}, L, C_*)$ is a $(\alpha, 4)$-chart with $\textnormal{fw}^*(G)\geq \delta$}; AND
\item \emph{For each $H\in C$, either $\mathbf{P}_H$ is a path of length at most one or $H$ is a cycle of length at most four with $\mathbf{P}_H=H$.}
\end{enumerate}
\end{defn}

\begin{prop}\label{SuffAllTilColProp} If all tilings are colorable, then Theorem \ref{5ListHighRepFacesFarMainRes} holds. \end{prop}

\begin{proof} Suppose all tilings are colorable. Let $\mathcal{T}=(\Sigma, G, \mathcal{C}, L, C_*)$ be a (not necessarily chord-triangulated) chart such that, letting $g:=g(\Sigma)$, letting $\alpha, \beta, \gamma, \delta$ be as in Definition \ref{TilingDefnM}, where $\alpha, \delta$ depend on $g$, the following hold. 
\begin{enumerate}[label=\roman*)]
\itemsep-0.1em
\item\emph{$(\Sigma, G, \mathcal{C}, L, C_*)$ is a $(\alpha, 4)$-chart with $\textnormal{fw}(G)\geq 5\delta$}; AND
\item \emph{For each $H\in C$, either $\mathbf{P}_H$ is a path of length at most one or $H$ is a cycle of length at most four with $\mathbf{P}_H=H$.}
\end{enumerate}
It suffice to prove that $G$ is $L$-colorable. By Lemma \ref{TriangulationCorMainLmemmaused1}, there is an embedding $G'$ obtained from $G$ by adding edges to $G$ such that $\textnormal{fw}(G')\geq 5\delta$, where $G'$ is chord-triangulated and $\mathcal{T}'=(\Sigma, G', \mathcal{C}, L, C_*)$ is an $(\alpha, 4)$-chart. Note that $\mathcal{T}'$ still satisfies ii) above. By A) of Proposition \ref{FWstarToFWObs}, $\textnormal{fw}^*(G')\geq\delta$, so $\mathcal{T}'$ is a tiling and $G'$ is $L$-colorable, so $G$ is $L$-colorable as well. \end{proof}

\section{Properties of Critical Tilings}\label{MainRedThmForTessSec}

We prove the following in the remainder of this paper:

\begin{theorem}\label{ShortReductionTheoremfirst} 
All tilings are colorable. 
\end{theorem}

We fix constants $\beta, \gamma$ and functions $\alpha(g)$ and $\delta(g)$ as in Definition \ref{TilingDefnM}. In the analysis below, whenever the surface is clear from the context, then we just write $\alpha, \delta$ in place of $\alpha(g)$ and $\delta(g)$ respectively to avoid clutter.

\begin{defn} \emph{A tiling $\mathcal{T}=(\Sigma, G, \mathcal{C}, L, C_*)$ is called \emph{critical} if $\mathcal{T}$ is not colorable, and, for any tiling $\mathcal{T}'=(\Sigma', G', \mathcal{C}', L', C_*')$, the following hold:}
\begin{enumerate}[label=\emph{\arabic*)}]
\itemsep-0.1em
\item\emph{If $|V(G')|<|V(G)|$, then $\mathcal{T}'$ is colorable}; AND
\item\emph{If $|V(G')|=|V(G)|$ and $|E(G)|<|E(G')|$, then $\mathcal{T}'$ is colorable.}
\end{enumerate}
\end{defn}

Over Sections \ref{MainRedThmForTessSec}-\ref{thisisidicritCOMPcomcriTT}, we prove that no critical tilings exist. First, in Section \ref{MainRedThmForTessSec}, we gather some basic properties of critical tilings. We first introduce one more natural piece of notation. Given a chart $\mathcal{T}=(\Sigma, G, \mathcal{C}, L, C_*)$ and a subgraph $H$ of $G$, we define $\mathcal{C}^H:=\{C\in\mathcal{C}: C\subseteq H\}$.

\begin{lemma}\label{GeneralizedShortPathObsx}
Let $\mathcal{T}=(\Sigma, G, \mathcal{C}, L, C_*)$ be a critical tiling. Then the following hold $G$ is connected and $G\neq C_*$, and each element of $\mathcal{C}$ is a cycle.
\end{lemma}

\begin{proof} It is an immediate consequence of the minimality of $\mathcal{T}$ that $G$ is connected. Since $G$ is not $L$-colorable, $G\neq C_*$. Suppose now that there is an $H\in\mathcal{C}$ which is not a cycle. Thus, there is a $v\in V(H)$ which is a cut-vertex of $G$, and $G$ admits a partition $G=G_0\cup G_1$, where $G_0\cap G_1=v$, and $G_i-v\neq\varnothing$ for each $i=0,1$. Since $H$ is not a cycle, $\mathbf{P}_H$ is a path of length at most one, so, without loss of generality, let $\mathbf{P}_{H}\subseteq G_0$. Note that, for each $k=0,1$, we have $\textnormal{fw}^*(G_k)\geq \textnormal{fw}^*(G)$. Furthermore, the chart $(\Sigma, G_0, \mathcal{C}^{\subseteq G_0}\cup\{H\cap G_0\}, L, H\cap G_0)$ is still chord-triangulated, and thus a tiling. By the minimality of $\mathcal{T}$, there is an $L$-coloring $\phi$ of $G_0$. Now, $(\Sigma, G_1, \mathcal{C}^{\subseteq G_1}\cup\{H\cap G_1\}, L_{\phi}^v, H\cap G_1)$ is also a chord-triangulated chart, and thus a tiling as well. By minimality, $G_1$ is $L_{\phi}^v$-colorable, so $G$ is $L$-colorable, contradicting our assumption. \end{proof}

For the proof of Theorem \ref{ShortReductionTheoremfirst}, it is useful to regard a chart $(\Sigma, G, \mathcal{C}, L, C_*)$ as having an ``orientation" which is specified by $C_*$, ananlogous to the outer face of an embedding in $\mathbb{R}^2$, which is made precise by the following definitions. 

\begin{defn}\label{Def14InOutContract} \emph{Let $\mathcal{T}=(\Sigma, G, \mathcal{C}, L, C_*)$ be a chart where $C_*$ is a proper subgraph of $G$. Let $D\subseteq G$ be a contractible cycle, and let $U_0, U_1$ be the components of $\Sigma\setminus D$. Let $i\in\{0,1\}$ be the unique index such that both of the following hold: $C_*\subseteq\textnormal{Cl}(U_i)$ and, if $C_*=D$, then $G\cap\textnormal{Cl}(U_i)=C_*$. We let $\textnormal{Ext}_{\mathcal{T}}(D)=G\cap\textnormal{Cl}(U_i)$ and $\textnormal{Int}_{\mathcal{T}}(D)=G\cap\textnormal{Cl}(U_{1-i})$.}
 \end{defn}

We now introduce the following notation, which generalizes the notion of an annulus in a planar graph. 
\begin{defn}\label{GenAnnDefnM}\emph{Let $\Sigma$ be a surface.}
\begin{enumerate}[label=\arabic*)]
\itemsep-0.1em
\item Given an embedding $G$ be on $\Sigma$, we let $\mathit{Sep}(G)$ denote the set of separating cycles of length at most four in $G$.
\item Let $\mathcal{T}=(\Sigma, G, \mathcal{C}, L, C_*)$ be a chart with underlying surface $\Sigma$, where $G\neq C_*$. Let $D\subseteq G$ be a contractible cycle and $\mathcal{F}$ be a family of contractible cycles in $\textnormal{Int}_{\mathcal{T}}(D)$, we define $A_{\mathcal{T}}(D| \mathcal{F})$ to be the graph $$A_{\mathcal{T}}(D| \mathcal{F}):=\textnormal{Int}_{\mathcal{T}}(D)\cap\left(\bigcap_{F\in\mathcal{F}}\textnormal{Ext}_{\mathcal{T}}(F)\right)$$
In particular, if $\mathcal{F}=\varnothing$, then $A_{\mathcal{T}}(D|\mathcal{F})=\textnormal{Int}_{\mathcal{T}}(D)$.
\end{enumerate}
\end{defn}

When using the notation of Definitions \ref{Def14InOutContract} and \ref{GenAnnDefnM}, if the chart $\mathcal{T}$ is clear from the context, then we usually drop the subscript $\mathcal{T}$. At the end of Section \ref{MainRedThmForTessSec}, we  prove that a critical tiling $\mathcal{T}=(\Sigma, G, \mathcal{C}, C_*)$ has the property that every facial subgraph of $G$, except possibly those of $\mathcal{C}$, is a triangle. We first prove the following intermediate result. 

\begin{prop}\label{MainFaceCycleLemma1} 

Let $\mathcal{T}=(\Sigma, G, \mathcal{C}, L, C_*)$ be a critical tiling. Let $F_0$ be a contractible cycle in $G$, and $\mathcal{F}$ be a (possibly empty) collection of contractible cycles in $G$, where each $F\in\mathcal{F}$ is a subgraph of $\textnormal{Int}_{\mathcal{T}}(F_0)$. Let $A:=A_{\mathcal{T}}(F_0|\mathcal{F})$ and $\mathcal{F}^*:=(\mathcal{F}\cup\{F_0\})\setminus\mathcal{C}^{\subseteq A}$, and suppose that the following conditions hold. 
\begin{enumerate}[label=\arabic*)]
\item $A$ is short-inseparable and $\mathcal{C}^{\subseteq A}\subseteq\mathcal{F}$; AND
\item $3\leq |V(F)|\leq 4$ for each $F\in\mathcal{F}^*$. 
\end{enumerate}
Then every face of $A$, except possibly those of $\{F_0\}\cup\mathcal{F}$, is bounded by a triangle. Furthermore, if $d(F, F')\geq\delta$ for all $F\in\mathcal{F}^*$ and $F'\in\mathcal{F}^*\cup\mathcal{C}^{\subseteq A}$, then any $L$-coloring of $\bigcup_{F\in\mathcal{F}^*}V(F)$ extends to an $L$-coloring of $A$.\end{prop}

\begin{proof} We begin with the first part of the proposition:

\begin{claim}\label{AnnulusTriangleMain}
Every face of $A$, except possibly those bounded by a cycle among $\{F_0\}\cup\mathcal{F}$, is bounded by a triangle. \end{claim}

\begin{claimproof} Consider the chart $\mathcal{T}^{\dagger}:=(\Sigma, A, \mathcal{F}\cup\{F_0\}, L, F_0)$. Note that, since $\mathcal{T}$ is a chord-triangulated chart, so is $\mathcal{T}^{\dagger}$. Furthermore, since $F_0\subseteq A$, every facial subgraph of $A$ contains at least one cycle. Suppoe there is a facial subgraph $K$ of $A$ with $K\not\in\mathcal{F}\cup\{F_0\}$, where $K$ is not a triangle. Since $K$ contains at least one cycle and $A$ is short-inseparable, it follows that $V(K)=V(A)$, so $A$ is 2-connected and $K$ is a cycle. Since $K\not\in$, all chords of $K$ lie in $A$, and since $K$ is not a triangle, it follows that there is a triangle $T\subseteq K$, which, in $G$, separates a vertex of $K\setminus T$ from $C_*$. But then $T$ is also a separating cycle in $A$, contradicting our assumption that $A$ is short-inseparable. \end{claimproof} 

Now we prove the second part of the proposition. Suppose $d(F, F')\geq\delta$ for all $F\in\mathcal{F}^*$ and $F'\in\mathcal{F}^*\cup\mathcal{C}^{\subseteq A}$. Possibly $\mathcal{F}^*=\varnothing$. In any case, let $S:=\bigcup_{F\in\mathcal{F}^*}V(F)$ and $\phi$ be an $L$-coloring of $S$. Consider the tuple $\mathcal{T}_A:=(\Sigma, A, \mathcal{F}\cup\{F_0\}, L^S_{\phi}, F_0)$. Since $\textnormal{fw}^*(G)\geq\delta$, we have $\textnormal{fw}^*(A)\geq \delta$. This is immediate from repeated applications of B) of Proposition \ref{FWstarToFWObs}. Thus, by A) of Proposition \ref{FWstarToFWObs}, we have $\textnormal{fw}(A)\geq\delta$. Now, $A$ is short-inseparable by assumption, and since every facial subgraph of $A$ other than those of $\mathcal{F}\cup\{F_0\}$ is a triangle, $\mathcal{T}_A$ is chord-triangulated, and thus a tessellation. Thus, it follows from Theorem \ref{PaIIBlackBoxTessMain} applied to $\mathcal{T}_A$ that $A$ is $L^S_{\phi}$-colorable, as desired. \end{proof}

We now have the following. 

\begin{lemma}\label{SsepclaimAndAlmostTriang}
Let $\mathcal{T}=(\Sigma, G, \mathcal{C}, L, C_*)$ be a critical tiling. Then the following hold.
\begin{enumerate}[label=\arabic*)]
\itemsep-0.1em
\item $\mathit{Sep}(G)\neq\varnothing$ and every facial subgraph $K$ of $G$ with $K\not\in\mathcal{C}$ is a triangle.
\item For each $D\in\textit{Sep}(G)$, each of the induced graphs $G[V(\textnormal{Ext}_{\mathcal{T}}(D))]$ and $G[V(\textnormal{Int}_{\mathcal{T}}(D))]$ is $L$-colorable, and furthermore, $D$ separates $C_*$ from at least one element of $\mathcal{C}\setminus\{C_*\}$. 
\end{enumerate}
\end{lemma}

\begin{proof} We first prove 1). Suppose first that $\textit{Sep}(G)=\varnothing$. Thus, since $\mathcal{T}$ is chord-triangulated, it is also a tessellation, and since $\textnormal{fw}(G)\geq\textnormal{fw}^*(G)\geq\delta$, it follows from Theorem \ref{PaIIBlackBoxTessMain} that $G$ is $L$-colorable, which is false. Thus, $\mathit{Sep}(G)\neq\varnothing$. Let $K$ be a facial subgraph of $G$ with $K\not\in\mathcal{C}$. We show that $K$ is a triangle. Note that, for each $D\in\textit{Sep}(G)$, either $K\subseteq\textnormal{Int}_{\mathcal{T}}(D)$ or $K\subseteq\textnormal{Ext}_{\mathcal{T}}(D)$. Consider the following cases.

\textbf{Case 1:} For each $D\in\textit{Sep}(G)$, $K\subseteq\textnormal{Ext}_{\mathcal{T}}(D)$

In this case, we consider the annulus $A:=A_{\mathcal{T}}(C_*\mid\textit{Sep}(G))$. By definition, $A$ is short-inseparable, and furthermore, $K$ is a facial subgraph of $A$ with $K\not\in\mathcal{C}\cup\textit{Sep}(G)$, so it follows from Proposition \ref{MainFaceCycleLemma1} that $K$ is a triangle. 

\textbf{Case 2:} There exists a $D\in\textit{Sep}(G)$ with $K\subseteq\textnormal{Int}_{\mathcal{T}}(D)$

In this case, we let $\mathcal{F}_K:=\{D\in\textit{Sep}(G): K\subseteq\textnormal{Int}_{\mathcal{T}}(D)\}$. We let $D_m$ be a cycle minimizing $|V(\textnormal{Int}_{\mathcal{T}}(D))|$ over all $D\in\mathcal{F}_K$. We also let $\mathcal{F}'_K:=\{D\in\textit{Sep}(G): D\subseteq\textnormal{Int}_{\mathcal{T}}(D_m)\ \textnormal{and}\ D\neq D_m\}$. Possibly $\mathcal{F}'_K=\varnothing$. In any case, we let $A:=A_{\mathcal{T}}(D_m\mid\mathcal{F}'_K)$. Note that $K$ is also a facial subgraph of $A$, and, again by Proposition \ref{MainFaceCycleLemma1}, $K$ is a triangle. Since each element of $\mathcal{C}$ is a cycle, it follows that every facial subgraph of $G$ is a cycle. This proves 1). 

Now we prove 2). Let $D\in\mathit{Sep}(G)$. Let $H:=G[\textnormal{Ext}_{\mathcal{T}}(D)]$ and $\mathcal{T}_{\textnormal{out}}:=(\Sigma, H, \mathcal{C}^{\subseteq H}, L, C_*)$. By B) of Proposition \ref{FWstarToFWObs}, $\textnormal{fw}^*(H)\geq\textnormal{fw}^*(G)$. Suppose first that either $|E(D)|=3$ or $D$ is not induced in $G$. In that case, $\mathcal{T}_{\textnormal{out}}$ is already chord-triangulated, so $\mathcal{T}_{\textnormal{out}}$ is a tiling and, by vertex-minimality, $H$ is $L$-colorable. Now suppose that $D$ is an induced cycle of length four. In this case, there is a unique component $U$ of $\Sigma\setminus H$ with $\partial(U)=D$. Since $D\not\in\mathcal{C}^{\subseteq H}$, it follows from Lemma \ref{TriangulationCorMainLmemmaused1} that there is an embedding $H'$ obtained from $H$ by adding precisely one edge $e$ to $\textnormal{Cl}(U)$, with opposing endpoints in $D$, such that each cycle in $D+e$ is contractible, each element of $\mathcal{C}^{\subseteq H}$ is still a facial subgraph of $H'$, and the elements of $\mathcal{C}^{\subseteq H}$ are pairwise of distance at least $\alpha$ apart in $H'$. Since $D$ has length four, we have $\textnormal{fw}^*(H')=\textnormal{fw}^*(H)\geq\textnormal{fw}^*(G)$, so, by vertex-minimality, $H'$ is $L$-colorable, and thus $H$ is $L$-colorable. 

\begin{claim}\label{InsideTNotOnlyRing} There is at least one $C\in\mathcal{C}$ with $C\subseteq\textnormal{Int}_{\mathcal{T}}(D)$. \end{claim}

\begin{claimproof} Suppose not. As shown above, $H$ admits an $L$-coloring $\phi$. Consider the tuple $\mathcal{T}':=(\Sigma, \textnormal{Int}_{\mathcal{T}}(D), \{D\}, L_{\phi}^D, D)$. By B) of Proposition \ref{FWstarToFWObs}, $\textnormal{fw}^*(\textnormal{Int}_{\mathcal{T}}(D))\geq\textnormal{fw}^*(G)$, and the pairwise-distance conditions on the designated faces are trivially satisfied, and $\mathcal{T}'$ is still chord-triangulated, so it is a tiling. By minimality, $\textnormal{T}'$ is colorable, so $\phi$ extends to an $L$-coloring of $G$, which is false. \end{claimproof}

Now let $K:=G[\textnormal{Int}_{\mathcal{T}}(D)]$. By Claim \ref{InsideTNotOnlyRing}, there is a $C\in\mathcal{C}^{\subseteq K}$. Let $\mathcal{T}_{\textnormal{in}}:=(\Sigma, K, \mathcal{C}^{\subseteq K}, L, C)$. Now an identical argument to the one for $V(\textnormal{Ext}_{\mathcal{T}}(D))$ above shows that $K$ is also $L$-colorable. This proves Lemma \ref{SsepclaimAndAlmostTriang}. \end{proof}

\section{Cycle Descendants}\label{CycleDescSec}

The main ingredient in the proof of Theorem \ref{ShortReductionTheoremfirst} is Proposition \ref{CriticalLemmaSep}, where we prove that, for a critical tiling $\mathcal{T}=(\Sigma, G, \mathcal{C}, L, C_*)$ and each $D\in\textnormal{Sep}(G)$, there is an ``obstruction" to extending an $L$-coloring of $V(D)$ to an $L$-coloring of $\textnormal{Int}_{\mathcal{T}}(D)$, where this obstruction is contained in $\textnormal{Int}_{\mathcal{T}}(D)$ and is close to $D$ in a sense we make precise below in Definition \ref{DefnofcCloseRdedBluebinrelat}. 

\begin{defn}\emph{Let $\mathcal{T}:=(\Sigma, G, \mathcal{C}, L, C_*)$ be a chart.}
\begin{enumerate}[label=\emph{\arabic*)}]
\item \emph{Given a cycle $F\subseteq G$, we define the following.}
\begin{enumerate}[label=\emph{\alph*)}]
\itemsep-0.1em
\item \emph{A cycle $D\in\mathit{Sep}(G)$ is called a \emph{descendant} of $F$ if $D\neq F$ and $D\subseteq\textnormal{Int}(F)$. We denote the set of descendants of $F$ by $\mathcal{I}(F)$.}
\item \emph{A cycle $D\in\mathcal{I}(F)$ is called an \emph{immediate descendant} of $F$ if, for any $D'\in\mathcal{I}(F)$ such that $D\subseteq\textnormal{Int}(D')$, we have $D'=D$. We denote the set of immediate descendants of $F$ by $\mathcal{I}^m(F)$.}
\end{enumerate}
\item \emph{Give a cycle $D\in\mathit{Sep}(G)$, we define the following.}
\begin{enumerate}[label=\emph{\alph*)}]
\itemsep-0.1em
\item \emph{We say that $D$ is \emph{minimal} if $\mathcal{I}(D)=\varnothing$. Likewise, we say that $D$ is \emph{maximal} if there does not exist a $D'\in\mathit{Sep}(G)$ such that $D\in\mathcal{I}(D')$.}
\item \emph{We say that $D$ is a \emph{blue cycle} if, for every $C\in\mathcal{C}^{\subseteq\textnormal{Int}(D)}$, there exists a $D'\in\mathcal{I}(D)$ such that $C\subseteq\textnormal{Int}(D')$. Otherwise, we say that $D$ is a \emph{red} cycle.}
\item \emph{We let $\mathit{Sep}_r(G)$ denote the set of red cycles in $\mathit{Sep}(G)$, and we let $\mathit{Sep}_b(G)$ denote the set of blue cycles of $\mathit{Sep}(G)$.}
\end{enumerate}
\end{enumerate}
  \end{defn}

\begin{lemma}\label{BlueCycleObsMinimalRed}
Let $\mathcal{T}=(\Sigma, G, \mathcal{C}, L, C_*)$ be a critical tiling. Then the following hold.
\begin{enumerate}[label=\arabic*)]
\itemsep-0.1em
\item For any $D\in\mathit{Sep}_b(G)$, there exists a $D'\subseteq\textnormal{Int}(D)$ with $D'\in\mathit{Sep}_r(G)$; AND
\item For any minimal $D\in\mathit{Sep}(G)$, there exists a $C\in\mathcal{C}^{\subseteq\textnormal{Int}(D)}$ such that $d(C,D)<\delta$.
\end{enumerate}
\end{lemma}

\begin{proof} Let $D\in\mathit{Sep}_b(G)$. Since $D$ is blue, we have $\mathcal{I}(D)\neq\varnothing$ by definition. Since $G$ is finite, there exists a minimal descendant $D'$ of $D$. Thus, $\mathcal{I}(D')=\varnothing$, and so $D'\in\mathit{Sep}_r(G)$. This proves 1). Now we prove 2). Let $D\in\mathit{Sep}(G)$ be minimal. Since $\mathcal{I}(D)=\varnothing$, we have $D\in\mathit{Sep}_r(G)$. Since $D$ is minimal, we have $\mathcal{I}(D)=\varnothing$, and thus $A(D| \mathcal{I}(D))=\textnormal{Int}(D)$. Suppose toward a contradiction that $d(C, D)\geq\delta$ for all $C\in\mathcal{C}^{\subseteq\textnormal{Int}(D)}$. By 2) of Lemma \ref{SsepclaimAndAlmostTriang}, $V(\textnormal{Ext}(D))$ admits an $L$-coloring $\phi$, and,  by Proposition \ref{MainFaceCycleLemma1} applied to $A(D, \mathcal{I}(D))$, we get that $\phi$ extends to an $L$-coloring of $\textnormal{Int}(D)$, and thus $G$ is $L$-colorable, which is false. \end{proof}

We now introduce the following definitions.

\begin{defn}\label{DefnofcCloseRdedBluebinrelat} Let $\mathcal{T}:=(\Sigma, G, \mathcal{C}, L, C_*)$ be a chart.
\begin{enumerate}[label=\emph{\arabic*)}]
\item\emph{We say that a cycle $D\in\mathit{Sep}(G)$ is $\mathcal{C}$-\emph{close} if one of the following holds.}
\begin{enumerate}[label=\emph{\alph*)}]
\itemsep-0.1em 
\item \emph{$D\in\mathit{Sep}_r(G)$ and there exists a $C\in\mathcal{C}^{\subseteq A(D|\mathcal{I}^m(D))}$ such that $d(D,C)<\delta$}; OR
\item \emph{$D\in\mathit{Sep}_b(G)$ and there exists a red descendant $D'$ of $D$ such that $d(D, D')<\gamma$.}
\end{enumerate}
\item \emph{We define a binary relation $\sim$ on $\mathit{Sep}(G)$ as follows: For $D_0, D_1\in\mathit{Sep}(G)$, we say that $D_0\sim D_1$ if there is a $C\in\mathcal{C}$ such that $C\subseteq\textnormal{Int}(D_0)\cap\textnormal{Int}(D_1)$ and $d(C, D_i)\leq\delta+\gamma$ for each $i\in\{0,1\}$.}
\end{enumerate}\end{defn}

\begin{lemma}\label{ObviousDistanceFact}
Let $\mathcal{T}=(\Sigma, G, \mathcal{C}, L, C_*)$ be a critical tiling. Then the following hold.
\begin{enumerate}[label=\arabic*)]
\itemsep-0.1em
\item  If $D\in\mathit{Sep}(G)$ and $H_1, H_2$ are two subgraphs of $G$, then $d(H_1, H_2)\leq d(H_1, D)+d(H_2, D)+2$; AND
\item For any $D\in\mathit{Sep}(G)$ such that every cycle in $\{D\}\cup\mathcal{I}(D)$ is $\mathcal{C}$-close, there is a unique $C\in\mathcal{C}$ with $C\subseteq\textnormal{Int}(D)$ such that $d(C,D)\leq\delta+\gamma$.
\end{enumerate}
\end{lemma}

\begin{proof} 1) is immediate so we prove 2). If $D\in\mathit{Sep}_r(G)$ then, since $D$ is $\mathcal{C}$-close, there is a $C\in\mathcal{C}^{\subseteq A(D, \mathcal{I}^m(D))}$ such that $d(C,D)<\delta$, so we are done in that case. If $D\in\mathit{Sep}_b(G)$, then, since $D$ is $\mathcal{C}$-close, there is a $D'\in\mathit{Sep}_r(G)\cap\mathcal{I}(D)$ such that $d(D, D')<\gamma$. Since $D'$ is $\mathcal{C}$-close, there is a $C\in\mathcal{C}^{\subseteq A(D', \mathcal{I}^m(D'))}$ with $d(C, D')<\delta$. By 1), we have $d(C, D)\leq\delta+\gamma$. Now suppose there is another cycle $C'\in\mathcal{C}^{\subseteq\textnormal{Int}(D)}$ with $d(C', D)\leq\delta+\gamma$. Applying 1) again, we have $d(C, C')\leq 2(\delta+\gamma)+2<\alpha$, contradicting our distance conditions on $\mathcal{T}$. \end{proof}

It immediately follows from 2) of Lemma \ref{ObviousDistanceFact} that the relation $\sim$ partitions the set of $\mathcal{C}$-close cycles of $\mathit{Sep}(G)$ into equivalence classes. We now have the following:

\begin{prop}\label{equivclassobs}

Let $\mathcal{T}=(\Sigma, G, \mathcal{C}, L, C_*)$ be a critical chart and $\mathcal{M}\subseteq\mathit{Sep}(G)$ be a collection of separating cycles in $G$ of length at most four such that the following hold.

\begin{enumerate}[label=\arabic*)]
\itemsep-0.1em
\item For any distinct $D, D'\in\mathcal{M}$, $D\not\in\mathcal{I}(D')$; AND
\item For each $D\in\mathcal{M}$, every cycle of $\{D\}\cup\mathcal{I}(D)$ is $\mathcal{C}$-close.
\end{enumerate}
Let $D_1,\cdots, D_k\in\mathcal{M}$ be a set of representatives of distinct equivalence classes of $\mathcal{M}$. For each $i=1, \cdots, k$, let $[D_i]=\{D\in\mathcal{M}: D\sim D_i\}$. Then the following hold.
\begin{enumerate}[label=\arabic*)]
\item For each $1\leq i<j\leq k$, the graphs $\textnormal{Int}(D_i)$ and $\textnormal{Int}(D_j)$ are disjoint; AND
\item There exist $k$ distinct element $C_1,\cdots, C_k\in\mathcal{C}$ such that, for each $j\in\{1,\cdots, k\}$, $C_j\subseteq\bigcap_{D^*\in [D_j]}\textnormal{Int}(D^*)$ and $d(C_j, D^*)\leq\delta+\gamma$ for each $D^*\in [D_j]$; AND
\item $\bigcup_{i=1}^k V(\textnormal{Int}(D_i))$ is $L$-colorable. 
\end{enumerate}
\end{prop}

\begin{proof} Applying 2) of Lemma \ref{ObviousDistanceFact}, there exist $k$ distinct elements $C_1,\cdots, C_k\in\mathcal{C}$ such that, for each $j\in\{1,\cdots, k\}$, $C_j\subseteq\bigcap_{D^*\in [D_j]}\textnormal{Int}(D^*)$ and $d(C_j, D^*)\leq\delta+\gamma$ for each $D^*\in [D_j]$. Suppose toward a contradiction that there exists a pair of indices $1\leq i<j\leq k$ such that $\textnormal{Int}(D_i)\cap\textnormal{Int}(D_j)\neq\varnothing$. Then, since $D_i\not\in\mathcal{I}(D_j)$ and $D_j\not\in\mathcal{I}(D_i)$, we have $D_i\cap D_j\neq\varnothing$, and thus $d(C_i, C_j)\leq 2(\delta+\gamma+3)<\alpha$, contradicting our distance conditions on $\mathcal{T}$. To finish, it suffices to check that $\bigcup_{i=1}^kV(\textnormal{Int}(D_i))$ is $L$-colorable. For any distinct $i,j\in\{1,\cdots, k\}$, we have $d(D_i, C_i)\leq\delta+\gamma$ and $d(D_j, C_j)\leq\delta+\gamma$. By our conditions on $\mathcal{T}$, we have $d(C_i, C_j)\geq\alpha$. By two successive applications of 1) of Lemma \ref{ObviousDistanceFact}, we have $d(C_i, D_i)+d(D_i, D_j)+d(D_j, C_j)\geq\alpha-4$, so $d(D_i, D_j)\geq (\alpha-4)-2(\delta+\gamma)$. Since $D_i\not\in\mathcal{I}(D_j)$ and $\mathcal{I}(D_j)$, the graph $\bigcup_{i=1}^kG[V(\textnormal{Int}(D_i))$ is a union of connected components, pairwise of distance at least $(\alpha-4)-2(\delta+\gamma)$ apart. For each $i=1,\cdots, k$, $V(\textnormal{Int}(D_i))$ is $L$-colorable by Lemma \ref{SsepclaimAndAlmostTriang}. Since $(\alpha-4)-2(\delta+\gamma)>1$, the union $\bigcup_{i=1}^kV(\textnormal{Int}(D_i))$ is also $L$-colorable. \end{proof}

Note that, by Proposition \ref{equivclassobs}, it follows that, for any cycle $D\subseteq G$, the relation $\sim$ partitions $\mathcal{I}^m(D)$ into equivalence classes, since no cycle of $\mathcal{I}^m(D)$ lies in the interior of another. 

\begin{prop}\label{RedBlueInequality}
Let $\mathcal{T}=(G, \mathcal{C}, L, C_*)$ be a critical chart. Let $D$ be a cycle in $G$ and suppose that, for every cycle $D'\in\mathcal{I}(D)$, every element of $\{D'\}\cup\mathcal{I}(D')$ is $\mathcal{C}$-close. We then have the following.
\begin{enumerate}[label=\arabic*)]
\itemsep-0.1em 
\item For any $D_1, D_2\in\mathcal{I}^m(D)$, where $D_1, D_2$ lie in different equivalence classes of $\mathcal{I}^m(D)$ under $\sim$, we have $d(D_1, D_2)\geq 5\delta\log_2(\delta)+\gamma$; A D
\item For any $D'\in\mathcal{I}^m(D)$ and $C\in\mathcal{C}$ with $C\subseteq G[A(D| \mathcal{I}^m(D))]$, we have $d(D', C)\geq 6\delta\log_2(\delta)+2\gamma$;
\end{enumerate}

\end{prop}

\begin{proof} We first prove 1). We first note that $\textnormal{Int}(D_1)\cap\textnormal{Int}(D_2)=\varnothing$ by Proposition \ref{equivclassobs}. By 2) of Lemma \ref{ObviousDistanceFact}, we have the following. For each $j\in\{1,2\}$, there exists an element $C_j\in\mathcal{C}$ with $C_j\subseteq\textnormal{Int}(D_j)$, where $d(C_j, D_j)\leq\delta+\gamma$. Thus, we obtain $d(C_1, D_1)+d(C_2, D_2)\leq 2(\delta+\gamma)$. Since $\textnormal{Int}(D_1)\cap\textnormal{Int}(D_2)=\varnothing$ and each of $C_1, C_2$ is a cycle, we have $C_1\neq C_2$, and thus $d(C_1, C_2)\geq\alpha$ by our conditions on $\mathcal{T}$. By two successive applications of 1) of Lemma \ref{ObviousDistanceFact}, we have $d(C_1, D_1)+d(D_1, D_2)+d(C_2, D_2)\geq\alpha-4$. Thus, we obtain $d(D_1, D_2)\geq (\alpha-4)-2(\delta+\gamma)$. Since $\alpha\geq 7\delta\log_2(\delta)+3\gamma$, we obtain $d(D_1, D_2)\geq 5\delta\log_2(\delta)+\gamma$, as desired. This proves 1). Now we prove 2). Since $D'$ is $\mathcal{C}$-close by assumption, there exists a $C'\in\mathcal{C}$ such that $C'\subseteq\textnormal{Int}(D')$ and $d(C', D')\leq\delta+\gamma$. Since $C'\not\subseteq A(D| \mathcal{I}^m(D))$, we have $C\neq C'$. Thus, $d(C, C')\geq\alpha$. By 1) of Lemma \ref{ObviousDistanceFact}, we then have $d(C, D')+d(D', C')\geq\alpha-2$, and thus $d(C, D')\geq (\alpha-2)-(\delta+\gamma)\geq 6\delta\log_2(\delta)+2\gamma$.  \end{proof}

To continue, we need some facts about the intersection of cycles. The following fact is very simple and is stated without proof. 

\begin{prop}\label{CycleIntersectionFacts}

Let $\mathcal{T}=(\Sigma, G, \mathcal{C}, L, C_*)$ be a chart, where $\textnormal{ew}(G)>4$. Let $C_0, C_1$ be cycles in $G$ with $|V(C_j)|\leq 4$ for each $j\in\{0,1\}$, and suppose that there are edges $e, f\in E(C_1)\setminus E(C_0)$ such that $e\in E(\textnormal{Int}(C_0))$ and $f\in E(\textnormal{Ext}(C_0))$. Let $A_{ii}, A_{ie}, A_{ei}, A_{ee}$ be the four cycles contained in the graph $C_0\cup C_1$, such that 
\begin{enumerate}[label=\arabic*)]
\itemsep-0.1em 
\item $\textnormal{Int}(A_{ii})=\textnormal{Int}(C_0)\cap\textnormal{Int}(C_1)$ \emph{and}\ $\textnormal{Int}(A_{ie})=\textnormal{Int}(C_0)\cap\textnormal{Ext}(C_1)$; AND 
\item $\textnormal{Int}(A_{ei})=\textnormal{Ext}(C_0)\cap\textnormal{Ext}(C_1)$ \emph{and}\  $\textnormal{Int}(A_{ee})=\textnormal{Int}(A_{ie})\cup\textnormal{Int}(A_{ei})\cup\textnormal{Int}(A_{ie})$; AND
\item $\textnormal{Ext}(A_{ee})=\textnormal{Ext}(C_0)\cap\textnormal{Ext}(C_1)$.
\end{enumerate}
Then the following hold. 
\begin{enumerate}[label=\arabic*)]
\itemsep-0.1em
\item If $C_0, C_1$ are edge-disjoint then $|E(A_{ee})|+|E(A_{ii})|$ and $|E(A_{ei})|+|E(A_{ie})|$ are both equal to $|E(C_0)|+|E(C_1)|$; AND
\item If $C_0, C_1$ are not edge disjoint then $|E(A_{ee})|+|E(A_{ii})|=E(C_0)|+|E(C_1)|$ and $|E(A_{ei})|+|E(A_{ie})|=|E(C_0)|+|E(C_1)|-2$; AND
\item If $|E(C_0)|=|E(C_1)|=4$, then the lengths of the four cycles $A_{ii}, A_{ie}, A_{ei}, A_{ee}$ have the same parity. 
\end{enumerate}\end{prop}

Now we have the following key fact:

\begin{prop}\label{equivrep114} Let $\mathcal{T}=(\Sigma, G, \mathcal{C}, L, C_*)$ be a critical tiling. Let $D$ be a cycle in $G$ and suppose that, for each $D'\in\mathcal{I}(D)$, every element of $\{D'\}\cup\mathcal{I}(D')$ is $\mathcal{C}$-close. Then there exists a system $\mathcal{D}$ of distinct representatives of the $\sim$-equivalence classes of $\mathcal{I}^m(D)$ under the relation $\sim$ such that $A(D| \mathcal{D})$ is short-inseparable. \end{prop}

\begin{proof} Let $\mathcal{D}$ be a system of distinct representatives of distinct equivalence classes of $\mathcal{I}^m(D)$, and, among all choices of systems of distinct representatives of $\sim$-equivalence classes in $\mathcal{I}^m(D)$, we choose $\mathcal{D}$ so as to minimize the quantity $|\mathit{Sep}(A(D| \mathcal{D}))|$. We claim now that $\mathit{Sep}(A(D|\mathcal{D}))=\varnothing$. Suppose toward a contradiction that $\mathit{Sep}(A(D|\mathcal{D}))\neq\varnothing$, and let $T\in\mathit{Sep}(A(D| \mathcal{D}))$. Note that $T$ is also a separating cycle of length at most 4 in $G$, and since $T\subseteq\textnormal{Int}(D)$, we have $T\in\mathcal{I}(D)$, so there exists a $D^*\in\mathcal{I}^m(D)$ such that $T\subseteq\textnormal{Int}(D^*)$. Now, for some unique $F\in\mathcal{D}$, we have $D^*\sim F$. Note that $D^*\neq F$, or else $T$ is not a separating cycle of $A(D|\mathcal{D})$, since $\textnormal{Int}(F)\cap A(D|\mathcal{D})=F$, which is a facial subgraph of $A(D|\mathcal{D})$. By Proposition \ref{equivclassobs}, there is a unique element $X\in\mathcal{C}$ such that $X\subseteq\bigcap (\textnormal{Int}(D'): D'\sim F\ \textnormal{and}\ D'\in\mathcal{I}^m(D))$, and each element of $\mathcal{I}^m(D)$ which is equivalent to $D^*$ under $\sim$ is of distance at most $\delta+\gamma$ from $X$. Now, $X\subseteq\textnormal{Int}(F)\cap\textnormal{Int}(D^*)$, and, since both $D^*, F$ lie in $\mathcal{I}^m(D)$, we have $D^*\not\in\mathcal{I}(F)$ and $F\not\in\mathcal{I}(D^*)$. Thus, we get $V(D^*)\cap V(F)\neq\varnothing$. so we apply Proposition \ref{CycleIntersectionFacts}. There exist four cycles $A_{ii}, A_{ie}, A_{ei}, A_{ee}$ in $G$, each of which  is a subgraph of $D^*\cup F$, such that $\textnormal{Int}(A_{ii})=\textnormal{Int}(D^*)\cap\textnormal{Int}(F)$ and $\textnormal{Int}(A_{ie})=\textnormal{Int}(D^*)\cap\textnormal{Ext}(F)$, and, analogously, $\textnormal{Int}(A_{ei})=\textnormal{Ext}(D^*)\cap\textnormal{Int}(F)$ and $\textnormal{Int}(A_{ee})=\textnormal{Ext}(D^*)\cap\textnormal{Ext}(F)$. 

\begin{claim} $T\subseteq\textnormal{Ext}(F)$ and $|V(A_{ie})|\geq 5$. Furthermore, $|V(A_{ei})|\leq 4$. \end{claim}

\begin{claimproof} Since $T\subseteq A(D| \mathcal{D})$, we have $T\subseteq\textnormal{Ext}(F)$ and thus $T\subseteq\textnormal{Int}(A_{ie})$. Suppose towards a contradiction that $|V(A_{ie})|\leq 4$. Since $T\in\mathit{Sep}(G)$, there is a $v\in V(\textnormal{Int}(T)\setminus V(T))$. Thus, $A_{ie}\in\mathit{Sep}(G)$, since $A$ separates $v$ from a point of $\textnormal{Ext}(F)\setminus V(F)$. Since $A_{ie}\in\mathit{Sep}(G)$ we have $A_{ie}\in\mathcal{I}(D^*)$. Thus, since $A_{ie}$ is $\mathcal{C}$-close by assumption, there is a cycle $C'\in\mathcal{C}^{\subseteq\textnormal{Int}(A_{ie})}$ with $d(C', A_{ie})\leq\delta+\gamma$ by Lemma \ref{ObviousDistanceFact}. Note that $C'\neq X$ since $X\subseteq\textnormal{Int}(F)$. Since $A_{ie}\subseteq D^*\cup F$, we get $d(C', D^*\cup F)\leq\delta+\gamma$, and we have $d(X, D^*)\leq\delta+\gamma$ and $d(X, F)\leq\delta+\gamma$, so $d(C', X)\leq 2(\delta+\gamma)+2<\alpha$, contradicting our distance conditions on $\mathcal{T}$. Thus, $|V(A_{ie})|\geq 5$. By Proposition \ref{CycleIntersectionFacts}, at least one of $A_{ie}, A_{ei}$ has length at most 4, so $|V(A_{ei})|\leq 4$. \end{claimproof}

Applying the above, we have the following. 

\begin{claim}\label{vieiprop1167} $V(\textnormal{Int}(A_{ei}))=V(A_{ei})$. \end{claim}

\begin{claimproof} Suppose toward a contradiction that $V(\textnormal{Int}(A_{ei}))\neq V(A_{ei})$. In that case, since $|V(A_{ei})|\leq 4$, we have $A_{ei}\in\mathit{Sep}(G)$. Since $A_{ei}\in\mathit{Sep}(G)$ we get that $A_{ei}$ is $\mathcal{C}$-close by assumption, and so there is a cycle $C^{\dagger}\in\mathcal{C}^{\subseteq\textnormal{Int}(A_{ei})}$ with $d(C^{\dagger}, A_{ei})\leq\delta+\gamma$ by Lemma \ref{ObviousDistanceFact}. Note that $C^{\dagger}\neq X*$, since $C^{\dagger}\subseteq\textnormal{Ext}(D^*)$ and $X\subseteq\textnormal{Int}(D^*)$ by definition. But since $A_{ei}\subseteq D^*\cup F$, we have $d(C^{\dagger}, D^*\cup F)\leq\delta+\gamma$, and we also have $d(X, D^*)\leq\delta+\gamma$ and $d(X, F)\leq\delta+\gamma$, so $d(C^{\dagger}, X)\leq 2(\delta+\gamma)+2<\alpha$ by 1) of Lemma \ref{ObviousDistanceFact}, contradicting our distance conditions on $\mathcal{T}$. Thus, our assumption that $V(\textnormal{Int}(A_{ei}))\neq V(A_{ei})$ is false. \end{claimproof}

Now consider the set $\mathcal{D}^{\dagger}:=(\mathcal{D}\cup\{D^*\})\setminus\{F\}$. This is also a system of distinct representatives of the equivalence classes of $\mathcal{I}^m(D)$. Furthermore, $T\not\in\mathit{Sep}(A(D| \mathcal{D}^{\dagger}))$, since $T\subseteq\textnormal{Int}(D^*)$ by assumption. If $\mathit{Sep}(A(D|\mathcal{D}^{\dagger}))\subseteq\mathit{Sep}(A(D| \mathcal{D}))$, then we have $|\mathit{Sep}(A(D |\mathcal{D}^{\dagger}))|<|\mathit{Sep}(A(D|\mathcal{D}))|$, contradicting the minimality of $|\mathit{Sep}(A(D| \mathcal{D}))|$. Thus, there is a $T^{\dagger}\in\mathit{Sep}(A(D| \mathcal{D}^{\dagger}))$ with $T^{\dagger}\not\in\mathit{Sep}(A(D| \mathcal{D}))$.

\begin{claim}\label{tdaggerinsidesec26propabrepl} $T^{\dagger}\subseteq\textnormal{Ext}(D^*)\cap\textnormal{Int}(F)$. \end{claim}

\begin{claimproof} Firstly, since $T^{\dagger}\subseteq A(D| \mathcal{D}^{\dagger})$, we have $T^{\dagger}\subseteq\textnormal{Ext}(D^*)$. Now suppose toward a contradiction that $T^{\dagger}\not\subseteq\textnormal{Int}(F)$. Note that $T^{\dagger}$ is also an element of $\mathit{Sep}(G)$, and $T^{\dagger}\subseteq\textnormal{Int}(D)$. Thus, there is a $D^{**}\in\mathcal{I}^m(D)$ such that $T^{\dagger}\subseteq\textnormal{Int}(D^{**})$, and there is a unique $F^{**}\in\mathcal{D}$ with $D^{**}\sim F^{**}$.

\vspace*{-8mm}
\begin{addmargin}[2em]{0em}
\begin{subclaim}\label{contrathissimplelittlsub}
$F^{**}=F$, and furthermore, $D^{**}\neq D^*$, and $T^{\dagger}\neq D^{**}$.
 \end{subclaim}

\begin{claimproof} If $F^{**}\neq F$, then $T^{\dagger}$ is a separating cycle of $A(D| \mathcal{D})$ if and only if $T^{\dagger}$ is a separating cycle of $A(D| \mathcal{D}^{\dagger})$, contradicting our assumption. Thus, we indeed have $F^{**}=F$, and $D^{**}\sim D^*\sim F$. Since $T^{\dagger}\subseteq\textnormal{Int}(D^{**})$ and $T^{\dagger}\in\mathit{Sep}(A(D| \mathcal{D}^{\dagger}))$, it follows that $D^{**}$ is not a facial subgraph of $A(D| \mathcal{D}^{\dagger})$, and thus $D^{**}\neq D^*$. Suppose now that $T^{\dagger}=D^{**}$. Since $D^{**}\sim D^*$ and $D^{**}\neq D^*$, it follows that $E(T^{\dagger})$ has nonempty intersection with $E(\textnormal{Int}(D^*))\setminus E(D^*)$, contradicting the fact that $T^{\dagger}\subseteq A(D| \mathcal{D}^{\dagger})$. Thus, we have $T^{\dagger}\neq D^{**}$.  \end{claimproof}\end{addmargin}

We claim now that $T^{\dagger}\subseteq\textnormal{Ext}(F)$. Suppose not. Then, since $T^{\dagger}\not\subseteq\textnormal{Int}(F)$ by assumption, $T^{\dagger}$ has an edge in $\textnormal{Int}(F)\setminus E(F)$, and an edge in $\textnormal{Ext}(F)\setminus E(F)$, and thus, $T^{\dagger}$ is also an immediate descentant of $D$, so we get $T^{\dagger}=D^{**}$, contradicting Subclaim \ref{contrathissimplelittlsub}. Thus, we have $T^{\dagger}\subseteq\textnormal{Ext}(F)$. Since $T^{\dagger}\subseteq\textnormal{Int}(D)\cap\textnormal{Ext}(F)$ and $D^{**}\sim D^*$, we have $\textnormal{Int}(T^{\dagger})\subseteq A(D|\mathcal{D})$. Since $T^{\dagger}\in\mathit{Sep}(A(D| \mathcal{D}^{\dagger}))$, we have $V(\textnormal{Int}(T^{\dagger})\setminus V(T^{\dagger})\neq\varnothing$.  Since $T^{\dagger}\subseteq\textnormal{Int}(D^{**})$ and $T^{\dagger}\neq D^{**}$, it follows that $T^{\dagger}$ separates a vertex of $\textnormal{Int}(T^{\dagger})\setminus T^{\dagger}$ from $D$, and thus $T^{\dagger}$ is a separating cycle of $A(D|\mathcal{D})$, contradicting our assumption that $T^{\dagger}\not\in\mathit{Sep}(A(D| \mathcal{D}))$. \end{claimproof}

Applying Claim \ref{tdaggerinsidesec26propabrepl}, we have $T^{\dagger}\subseteq\textnormal{Ext}(D^*)\cap\textnormal{Int}(F)$, and so $T^{\dagger}\subseteq\textnormal{Int}(A_{ei})$. By Claim \ref{vieiprop1167}, we have $V(\textnormal{Int}(A_{ei}))=V(A_{ei})$, which is false, as $T^{\dagger}$ is a separating cycle in $G$. Thus, $\mathit{Sep}(A(D| \mathcal{D}))$ is indeed empty, and $A(D| \mathcal{D})$ is short-inseparable, as desired. \end{proof}

\section{Completing the proof of Theorem \ref{ShortReductionTheoremfirst}}\label{thisisidicritCOMPcomcriTT}

We now prove the following, which is the last ingredient we need in order to complete the proof of Theorem \ref{ShortReductionTheoremfirst}.

\begin{prop}\label{CriticalLemmaSep} Let $\mathcal{T}=(G, \mathcal{C}, L, C_*)$ be a critical chart. Then every $D\in\mathit{Sep}(G)$ is $\mathcal{C}$-\emph{close}. \end{prop}

\begin{proof} Suppose toward a contradiction that there exists a $D^{\textnormal{mc}}\in\mathit{Sep}(G)$ which is not $\mathcal{C}$-close, and furthermore, among all elements of $\mathit{Sep}(G)$ which are not $\mathcal{C}$-close, we choose $D^{\textnormal{mc}}$ so as to minimize $|V(\textnormal{Int}(D^{\textnormal{mc}}))|$. Note that $\mathcal{I}^m(D^{\textnormal{mc}})\neq\varnothing$ or else $D^{\textnormal{mc}}$ is a minimal element of $\mathit{Sep}(G)$ and is thus $\mathcal{C}$-close by 2) of Lemma \ref{BlueCycleObsMinimalRed}. For each $D\in\{D^{\textnormal{mc}}\}\cup\mathcal{I}(D^{\textnormal{mc}})$, every element of $\mathcal{I}^m(D)$ is $\mathcal{C}$-close by our choice of $D^{\textnormal{mc}}$. By Proposition \ref{equivrep114}, there exists a system $\mathcal{M}_D\subseteq\mathcal{I}^m(D)$ of distinct representatives of the $\sim$-equivalence classes of $\mathcal{I}^m(D)$ such that $A(D| \mathcal{M}_D)$ is short-inseparable. Given a $D\in\{D^{\textnormal{mc}}\}\cup\mathcal{I}(D^{\textnormal{mc}})$, we say that $D$ is an \emph{obstructing cycle} if there exists a $D'\in\mathcal{M}_D$ such that $d(D, D')<\delta$. If there exist two distinct $D', D''\in\mathcal{M}_D$ such that $d(D, D'')<\delta$ and $d(D, D')<\delta$, then $d(D', D'')\leq 2\delta$ by 1) of Lemma \ref{ObviousDistanceFact}, contradicting Proposition \ref{RedBlueInequality}. Thus, if $D$ is an obstructing cycle, the corresponding $D'\in\mathcal{M}_D$ is unique. 

\begin{claim}\label{MCObstructCycle} $D^{\textnormal{mc}}$ is  an obstructing cycle. \end{claim}

\begin{claimproof} Suppose toward a contradiction that $D^{\textnormal{mc}}$ is not an obstructing cycle, and set $A^{\textnormal{mc}}:=A(D^{\textnormal{mc}}| \mathcal{M}_{D^{\textnormal{mc}}})$. 

\vspace*{-8mm}
\begin{addmargin}[2em]{0em}
\begin{subclaim}\label{usedtobobgsab26lastofrepl}
For each $D\in\{D^{\textnormal{mc}}\}\cup\mathcal{M}_{D^{\textnormal{mc}}}$, the following hold.
\begin{enumerate}[label=\arabic*)]
\item For each $C\in\mathcal{C}^{\subseteq A^{\textnormal{mc}}}$, we have $d(C, D)\geq\delta$; AND
\item For each $D'\in\{D^{\textnormal{mc}}\}\cup\mathcal{M}_{D^{\textnormal{mc}}}$ with $D'\neq D$, we have $d(D, D')\geq\delta$. 
\end{enumerate}
\end{subclaim}

\begin{claimproof} We break this into two cases.

\textbf{Case 1:} $D=D^{\textnormal{mc}}$.

Since $D^{\textnormal{mc}}$ is not an obstructing cycle, we have $d(D^{\textnormal{mc}}, D')\geq\delta$ for each $D'\in\mathcal{M}_{D^{\textnormal{mc}}}$. Now we check that $d(C, D^{\textnormal{mc}})\geq\delta$ for all $C\in\mathcal{C}^{\subseteq A^{\textnormal{mc}}}$. If $D^{\textnormal{mc}}$ is a blue cycle, then this immediately follows from the fact that $D^{\textnormal{mc}}$ is not an obstructing cycle, since each element of $\mathcal{C}^{\subseteq A^*}$ is separated from $D^{\textnormal{mc}}$ by an element of $\mathcal{M}_{D^{\textnormal{mc}}}$. On the other hand, if $D^{\textnormal{mc}}$ is a red cycle, then this is true by our assumption that $D^{\textnormal{mc}}$ is not $\mathcal{C}$-close.

\textbf{Case 2:} $D\neq D^{\textnormal{mc}}$

In this case, for any $C\in\mathcal{C}^{\subseteq A^{\textnormal{mc}}}$, we have $d(C,D)\geq\delta$ by Proposition \ref{RedBlueInequality}. so it suffices to prove that 2) holds for each $D'\in\mathcal{M}_{D^{\textnormal{mc}}}$, since, if $D'=D^{\textnormal{mc}}$, then we are back to Case 1 with the roles of $D, D'$ interchanged. Applying Proposition \ref{RedBlueInequality} again, it follows that, for any $ D'\in\mathcal{M}_{D^{\textnormal{mc}}}$, we have $d(D, D')\geq\delta$, so we are done. \end{claimproof}\end{addmargin}

By Proposition \ref{equivclassobs}, there is an $L$-coloring $\phi$ of $\bigcup (V(\textnormal{Int}(D)): D\in\mathcal{M}_{D^{\textnormal{mc}}})$. By Lemma \ref{SsepclaimAndAlmostTriang}, $V(\textnormal{Ext}(D^{\textnormal{mc}}))$ is $L$-colorable. By Subclaim \ref{usedtobobgsab26lastofrepl}, the graphs $\bigcup (V(\textnormal{Int}(D)): D\in\mathcal{M}_{D^{\textnormal{mc}}})$ and $V(\textnormal{Ext}(D^{\textnormal{mc}}))$ are of distance at least $\delta$ apart in $G$, so $\phi\cup\psi$ is a proper $L$-coloring of its domain. Applying Subclaim \ref{usedtobobgsab26lastofrepl}, together with Proposition \ref{MainFaceCycleLemma1}, $\phi\cup\psi$ extends to an $L$-coloring of $A^{\textnormal{mc}}$, and thus $G$ is $L$-colorable, which is false. This proves Claim \ref{MCObstructCycle}. \end{claimproof}

\begin{claim}\label{ReplPrevCase1} For all obstructing cycles $D\in\{D^{\textnormal{mc}}\}\cup\mathcal{I}(D^{\textnormal{mc}})$, $D$ is blue and $|\mathcal{M}_{D}|=1$. \end{claim}

\begin{claimproof} Suppose not. There exists an obstructing cycle $D\in \{D^{\textnormal{mc}}\}\cup\mathcal{I}(D^{\textnormal{mc}})$ such that either $D$ is red or $|\mathcal{M}_{D}|>1$. By definition there is a unique cycle $D'\in\mathcal{M}_{D}$ with $d(D, D')<\delta$. Let $A:=A_{\mathcal{T}}(D\mid\mathcal{M}_D)$ and let $U$ be the unique component of $\Sigma\setminus D$ with $C_*\not\subseteq\textnormal{Cl})(U)$.

\vspace*{-8mm}
\begin{addmargin}[2em]{0em}
\begin{subclaim}\label{USideC*Planar} $U$ is an open disc. \end{subclaim}

\begin{claimproof} Suppose not. Since $\textnormal{fw}^*(G)>1$, we can regard $\textnormal{Ext}(D)$ as a planar embedding. By 2) of Lemma \ref{SsepclaimAndAlmostTriang}, there is an $L$-coloring $\phi$ of $V(\textnormal{Int}(D))$. Let $H:=\textnormal{Ext}(D)$, regarded as an embedding on $\mathbb{S}^2$. Consider the tuple $\mathcal{T}':=(\mathbb{S}^2, H, \mathcal{C}^{\subseteq H}, L^D_{\phi}, C_*)$. Since $|V(H)|<|V(G)|$ and $\phi$ does not extend to $L$-color $H$, $\mathcal{T}'$ is not a tessellation. Since $\textnormal{fw}^*(H)=\infty$ and $H$ is chord-triangulated, the distance conditions of Definition \ref{TilingDefnM} are violated. Since the genus has only decreased, the distance conditions $\mathcal{T}'$ need to satisfy have only weakened, and so there is an $X\in\mathcal{C}^{\subseteq H}$ with $d(X, D)<\alpha(0)=7\delta(0)\log_2(\delta(0))+3\gamma$ On the other hand, since $D'$ is $\mathcal{C}$-close, it follows from 2) of Lemma \ref{ObviousDistanceFact} that there is an $X'\in\mathcal{C}$ with $X'\subseteq\textnormal{Int}(D')$ and $d(X', D')\leq\delta(g)+\gamma$. Note that $X\neq X'$ and,  by assumption, $d(D, D')<\delta(g)$, so, by two applications of 1) of Lemma \ref{ObviousDistanceFact}, we obtain 
\[d(X, X')<7\delta(0)\log_2(\delta(0))+4\gamma+2\delta(g)+4<\alpha(g)\] 
where the rightmost inequality holds since $g\geq 1$ and $\beta\geq\gamma$, so we contradict our distance conditions on $\mathcal{T}$. 
 \end{claimproof}\end{addmargin}

Note that it follows from Subclaim \ref{USideC*Planar} that $\textnormal{Int}(D)$ can be regarded as a planar embedding with outer cycle $D$. Let $P$ be a shortest $(D, D')$-path in $G$. Note that $P\subseteq A$ and $|E(P)|\leq\delta-1$. We now need to apply Theorem \ref{Main4CycleAnnulusThm}. Let $n:=2|E(P)|+8$ and $n':=\lceil\log_2(n)\rceil+2$. 

\vspace*{-8mm}
\begin{addmargin}[2em]{0em}
\begin{subclaim}\label{SigmaExtFromDD'Inwards}
\textcolor{white}{aaaaaaaaaaaaaaa}
\begin{enumerate}[label=\alph*)]
\itemsep-0.1em
\item $V(A)\not\subseteq B_{n'(n-1)}(D\cup D'\cup P)$, and for each $v\in V(A)\cap B_{n'n+1}(D\cup D'\cup P)$, every facial subgraph of $A$ containing $v$, except possibly $D, D'$, is a triangle, and, if $v\not\in V(D\cup D')$, then $|L(v)|\geq 5$
\item There is an $L$-coloring $\sigma$ of $G\setminus (V(A)\setminus B_{n'(n-1)}(D\cup D'\cup P))$.
\end{enumerate}
\end{subclaim}

\begin{claimproof} By assumption, either $D$ is red or $|\mathcal{M}_D|>1$, so there is a facial cycle $X$ of $A$ where $X\in\mathcal{C}\cup (\mathcal{M}_D\setminus\{D'\})$. Now, by Proposition \ref{RedBlueInequality}, we have $d(X, D')\geq 5\delta\log_2(\delta)+\gamma$, so we get $d(X, D\cup D'\cup P)\geq 5\delta\log_2(\delta)+\gamma-(\delta-1)>nn'+1>n'(n-1)$. Since this holds for each $X\in\mathcal{C}\cup (\mathcal{M}_D\setminus\{D'\})$, we immediately have a). Now we prove b). Let $G'$ be an embedding on $\Sigma$ obtained from $G$ by deleting all the vertices of $V(A)\setminus B_{n'(n-1)}(D\cup D'\cup P)$. Let $\mathcal{C}':=\mathcal{C}^{\subseteq G'}$. By Lemma \ref{TriangulationCorMainLmemmaused1}, there is an embedding $G''$ obtained from $G'$ by adding edges to $G'$ such that $G''$ is chord-triangulated, where the elements of $\mathcal{C}^{\subseteq G'}$ are all facial subgraphs of $G'$ and are still of pairwise distance at least $\alpha(g)$ apart. Now, $G'$ is obtained from $G$ by deleting vertices in $\textnormal{Cl}(U)\setminus D$ and $G''$ is necessarily obtained from $G'$ only by adding edges only to $\textnormal{Cl}(U)$. Since $U$ is a disc and $D$ has length four, we have $\textnormal{fw}^*(G'')=\textnormal{fw}^*(G')=\textnormal{fw}^*(G)$. Thus, $(\Sigma, G'', \mathcal{C}^{\subseteq G''}, L, C_*)$ is a tiling. By a), we have $|V(G'')|<|V(G)|$, so, by minimality, $G''$ is indeed $L$-colorable, and thus so is $G'$. This proves b). \end{claimproof}\end{addmargin}

Now, let $\sigma$ be as in Subclaim \ref{SigmaExtFromDD'Inwards} and let $\phi$ be the restriction of $\sigma$ to the domain $V(D\cup D'\cup P)$. Since $G$ is not $L$-colorable, $\phi$ dos not extend to $L$-color $A$. By Theorem \ref{Main4CycleAnnulusThm}, since $A$ is 2-connected and short-inseparable, there is a 2-edge-connected subgraph $K$ of $A$ and an extension of $\phi$ to a partial $L$-coloring $\psi$ of $V(K)$ such that
\begin{enumerate}[label=\arabic*)]
\itemsep-0.1em
\item $V(K)$ is $(L, \psi)$-inert in $A$ and $V(K)\subseteq B_{nn'}(F\cup F'\cup P)$; AND
\item For each connected component $H$ of $A\setminus K$, the outer face of $H$ is a Thomassen facial subgraph of $H$ with respect to $L_{\psi}$.
\end{enumerate}

Since $\phi$ does not extend to $L$-color $A$ and $K$ is $(L, \psi)$-inert in $A$, there is a connected component $H$ of $A\setminus K$ which is not $L_{\psi}$-colorable. Let $H^o$ be the outer face of $H$. 

\vspace*{-8mm}
\begin{addmargin}[2em]{0em}
\begin{subclaim}\label{ForEachFDistSubCL} For each $F\in\mathcal{M}_D\setminus\{D'\}$ and each $C\in\mathcal{C}$, we have $d(K, F)\geq\delta(0)$ and  $d(K, C)\geq\delta(0)$.
\end{subclaim}

\begin{claimproof} Let $F, C$ be as in the subclaim. Recall that $D'$ is $\mathcal{C}$-close, so there is an $X\in\mathcal{C}$ with $X\subseteq\textnormal{Int}(D')$ and $d(D', X)\leq\delta(g)+\gamma$. Thus, we get $d(D\cup D'\cup P, X)\leq\delta(g)+\gamma+(\delta(g)-1)$, and since $nn'\leq 3\delta(g)\log_2(\delta(g))$, it follows that, for each $v\in V(K)$, we have $d(v, X)\leq\delta(g)+\gamma+(\delta(g)-1)+nn'\leq 4\delta(g)\log_2(\delta(g))+\gamma$.

We conclude that $d(K, C)\geq\alpha(g)-(4\delta(g)\log_2(\delta(g))+\gamma)>\delta(0)$. Likewise, we have $d(X, F)\geq d(D', F)$, so by Proposition \ref{RedBlueInequality}, we get $d(X, F)\geq 5\delta\log_2(\delta)+\gamma$ as well, so we also get  $$d(K, F)\geq 5\delta(g)\log_2(\delta(g))+\gamma-(4\delta(g)\log_2(\delta(g))+\gamma)>\delta(0)$$ as desired. \end{claimproof}\end{addmargin}

Now, let $\mathcal{N}:=\{F\in\mathcal{M}_D\setminus\{D'\}: F\subseteq H\}$ and $S:=\bigcup_{F\in\mathcal{N}}V(F)$. By 3) of Proposition \ref{equivclassobs}, there is an $L$-coloring $\tau$ of $\bigcup_{F\in\mathcal{N}}V(\textnormal{Int}(F))$. By Subclaim \ref{ForEachFDistSubCL}, the union $\tau\cup\psi$ is a proper $L$-coloring of its domain. Furthermore, since $H$ is not $L_{\psi}$-colorable, it is also not $L_{\psi}^S$-colorable. Consider the tuple $\mathcal{T}':=(\mathbb{S}^2, H, \mathcal{C}^{\subseteq H}\cup\{H^o\}\cup\mathcal{N}, L_{\psi\cup\tau}^S, H^o)$. Now, every facial subgraph of $H$, except those among $\mathcal{C}^{\subseteq H}\cup\{H^o\}\cup\mathcal{N}$ is a triangle, so $\mathcal{T}'$ is chord-triangulated. Since $A$ is short-inseparable, $H$ is also short-inseparable, so $\mathcal{T}'$ is a tessellation with underlying surface $\mathbb{S}^2$. Every element of $\mathcal{C}^{\subseteq H}\cup\{H^o\}\cup\mathcal{N}$ is either precolored cycle of length at most four or has a precolored path of length at most one. Furthermore, it follows from the triangulation conditions in Subclaim \ref{SigmaExtFromDD'Inwards} that every vertex of $H^o$ has distance precisely one from $K$, so, by Subclaim \ref{ForEachFDistSubCL}, $\mathcal{T}'$ is a $(\delta(0), 4)$-tessellation. Thus, by Theorem \ref{PaIIBlackBoxTessMain}, $H$ is $L_{\psi\cup\tau}^S$-colorable, a contradiction. Thus proves Claim \ref{ReplPrevCase1}. \end{claimproof}

Since $D^{\textnormal{mc}}$ is an obstructing cycle by assumption, it follows from Claim \ref{ReplPrevCase1} that $D^{\textnormal{mc}}\in\mathit{Sep}_b(G)$, and so $\mathit{Sep}_r(G)\cap\mathcal{I}(D^{\textnormal{mc}})\neq\varnothing$ by 1) of Lemma \ref{BlueCycleObsMinimalRed}. Let $D_r$ be a maximal element of $\mathit{Sep}_r(G)\cap\mathcal{I}(D^{\textnormal{mc}})$.

\begin{claim}\label{simpletwosentblueclaiminlastof25} For each $D\in\{D^{\textnormal{mc}}\}\cup\mathcal{I}(D^{\textnormal{mc}})$, if $D$ is blue, then $D$ is an obstructing cycle. \end{claim}

\begin{claimproof} If $D=D^{\textnormal{mc}}$, then this holds by assumption. If $D\neq D^{\textnormal{mc}}$, then $D$ is $\mathcal{C}$-close by assumption, and thus there exists a $D^*\in\mathcal{I}^m(D)$ such that $d(D, D^*)<\gamma$. Since $\gamma\leq\delta$, $D$ is indeed an obstructing cycle. \end{claimproof}

Now set $A^*:=\textnormal{Int}(D^{\textnormal{mc}})\cap\textnormal{Ext}(D_r)$.

\begin{claim}\label{5listsonlyintheann}
For every $v\in V(A^*)\setminus V(D^{\textnormal{mc}}\cup D_r)$, $|L(v)|\geq 5$.
\end{claim}

\begin{claimproof} Since $D_r$ is a red cycle and $D_r$ is $\mathcal{C}$-close by assumption, there exists an $X\in\mathcal{C}$ with $X\subseteq\textnormal{Int}(D_r)$ and $d(X, D_r)<\delta$. To prove the claim, it suffices to show there does not exist a $C\in\mathcal{C}$ such that $C\subseteq A^*$. Suppose toward a contradiction that such a $C$ exists. Since $D^{\textnormal{mc}}$ is a blue cycle, there exists a red cycle $D_r'\in\mathcal{I}(D^{\textnormal{mc}})$ such that $D_r'\neq D_r$, $D_r'\subseteq A^*$, and $D_r'$ separates $C$ from $D^{\textnormal{mc}}$. Since $D_r'$ is $\mathcal{C}$-close by assumption, there exists a cycle $C^{\dagger}\in\mathcal{C}$ with $C^{\dagger}\subseteq\textnormal{Int}(D_r')$ and $d(C^{\dagger}, D_r')<\delta$. Since $C^{\dagger}\neq X$, we have $d(C^{\dagger}, X)\geq\alpha$ and thus $d(D_r, D_r')\geq\alpha-2\delta$ by 1) of Lemma \ref{ObviousDistanceFact}. It follows that $D_r\not\sim D_r'$. Let $\mathcal{U}:=\{D\in\mathit{Sep}(G): D_r\cup D_r'\subseteq\textnormal{Int}(D)\}$. Note that $\mathcal{U}\neq\varnothing$, since $D^{\textnormal{mc}}\in\mathcal{U}$. Among all elements of $U$, choose $D^u\in\mathcal{U}$ so as to minimize the quantity $|E(\textnormal{Int}_G(D^u))|$. Since $D_r$ is a maximal element of $\mathit{Sep}_r(G)\cap\mathcal{I}(D^{\textnormal{mc}})$ and $D_r'\not\subseteq\textnormal{Int}(D_r)$, we have $D^u\in\mathit{Sep}_b(G)$. By Claim \ref{simpletwosentblueclaiminlastof25}, every blue cycle in $\{D^{\textnormal{mc}}\}\cup\mathcal{I}(D^{\textnormal{mc}})$ is an obstructing cycle, so we have $|\mathcal{M}_{D^u}|=1$. If $D_r\in\mathcal{I}^m(D^u)$, then, since $|\mathcal{M}_{D^u}|=1$ and $D_r\not\sim D_r'$, it follows that $D_r'$ is a descendant of $D_r$, which is false. The same argument shows that $D_r'\not\in\mathcal{I}^m(D^u)$. Let $D'$ be the lone element of $\mathcal{M}_{D^u}$. Note that $D_r\cup D_r'\not\subseteq\textnormal{Int}(D')$, or else $D'$ contradicts the minimality of $|E(\textnormal{Int}_G(D^u))|$. Since neither $D_r$ nor $D_r'$ lies in $\mathcal{I}^m(D^u)$, we have $D_r\not\sim D'$ and $D_r'\not\sim D'$. Thus, at least one of $D_r, D_r'$ is separated from $D^u$ by the deletion of $D'$, and $\mathcal{I}^m(D^u)$ contains at least one equivalence class distinct from that of $D'$, contradicting the fact that $|\mathcal{M}_{D^u}|=1$. \end{claimproof}

Since $D^{\textnormal{mc}}\in\mathit{Sep}_b(G)$ and $D^{\textnormal{mc}}$ is not $\mathcal{C}$-close, we have $d(D_r, D^{\textnormal{mc}})\geq\gamma$. By Lemma \ref{SsepclaimAndAlmostTriang}, there is an $L$-coloring $\phi$ of $V(\textnormal{Int}(D_r))$ and an $L$-coloring $\psi$ of $V(\textnormal{Ext}(C))$. Since $D_r, D^{\textnormal{mc}}$ are of distance at least $\gamma$ apart, $\phi\cup\psi$ a proper $L$-coloring of its domain. Since each vertex of $A^*\setminus (D^{\textnormal{mc}}\cup D_r)$ has an $L$-list of size at least five, it follows from Theorem \ref{cylindertheorem} that $\phi\cup\psi$ extends to an $L$-coloring of $A^*$, so $G$ is $L$-colorable, contradicting the fact that $\mathcal{T}$ is a counterexample. This completes the proof of Proposition \ref{CriticalLemmaSep}. \end{proof}

With the results above in hand, we are ready to finish the proof of Theorem \ref{ShortReductionTheoremfirst}.

\begin{thmn} [\ref{ShortReductionTheoremfirst}] 
All tilings are colorable. 
\end{thmn}

\begin{proof} Suppose not. Thus, there exists a critical tiling $\mathcal{T}=(\Sigma, G, \mathcal{C}, L, C_*)$. By Lemma \ref{SsepclaimAndAlmostTriang}, $\mathit{Sep}(G)\neq\varnothing$, so let $\mathcal{M}$ be the set of maximal elements of $\mathit{Sep}(G)$. By Proposition \ref{CriticalLemmaSep}, for each $M\in\mathcal{M}$ and each $D\in\{M\}\cup\mathcal{I}(M)$, $D$ is $\mathcal{C}$-close. Thus, $\mathcal{M}$ admits a partition into equivalence classes under the relation $\sim$, and furthermore, by Proposition \ref{equivrep114}, there exists a system $\mathcal{M}_*\subseteq\mathcal{M}$ of distinct representatives of the $\sim$-equivalence classes of $\mathcal{M}$ such that $A(C_*| \mathcal{M}_*)$ is short-inseparable. Note that $A(C_*| \mathcal{M}_*)=\bigcap_{M\in\mathcal{M}_*}\textnormal{Ext}(M)$. By Proposition \ref{equivclassobs}, there is an $L$-coloring $\phi$ of $\bigcup_{M\in\mathcal{M}_*}V(\textnormal{Int}(M))$. By Proposition \ref{RedBlueInequality}, the following distance conditions are satisfied.
\begin{enumerate}[label=\arabic*)]
\item For any $D\in\mathcal{M}_*$ and any $C\in\mathcal{C}$ with $C\subseteq A(C_*| \mathcal{M}_*)$, we have $d(C, D)\geq\delta$; \emph{AND}
\item For any distinct $D, D'\in\mathcal{M}_*$, we have $d(D, D')\geq\delta$. 
\end{enumerate}

Thus, by Proposition \ref{MainFaceCycleLemma1}, $\phi$ extends to an $L$-coloring of $G$, contradicting the fact that $\mathcal{T}$ is  not colorable. This completes the proof of Theorem \ref{ShortReductionTheoremfirst} and thus proves Theorem \ref{5ListHighRepFacesFarMainRes}.  \end{proof}

\end{document}